\documentclass{colt2015} 

\usepackage{algorithm,algorithmic}
\usepackage{amssymb, multirow, paralist, color}
\newtheorem{thm}{Theorem}
\newtheorem{prop}{Proposition}
\newtheorem{cor}[thm]{Corollary}
\newtheorem{ass}[thm]{Assumption}
\usepackage{enumitem}
\usepackage{booktabs}

\def \neon {NEON}
\def \neonp {NEON$^+$}
\def \S {\mathbf{S}}
\def \A {\mathcal{A}}

\def \R {\mathbb{R}}

\def \w {\mathbf{w}}
\def \w {\mathbf{z}}
\def \v {\mathbf{v}}

\def \x {\mathbf{x}}

\def \E {\mathrm{E}}

\def \x {\mathbf{x}}
\def \p {\mathbf{p}}
\def \a {\mathbf{a}}

\def \e {\mathbf{e}}

\def \z {\mathbf{z}}

\def \y {\mathbf{y}}
\def \u {\mathbf{u}}

\def \g {\mathbf{g}}

\def \xh {\widehat{\x}}

\def \y {\mathbf{y}}
\def \E {\mathrm{E}}
\def \x {\mathbf{x}}
\def \g {\mathbf{g}}

\def \z {\mathbf{z}}
\def \u {\mathbf{u}}

\def \w {\mathbf{w}}

\def \R {\mathbb{R}}
\def \S {\mathcal{S}}

\def \A {\mathcal{A}}

\def \v {\mathbf{v}}

\def \p {\mathbf{p}}

\def \a {\mathbf{a}}

\def \xh {\widehat{\x}}

\def \u {\mathbf{u}}
\def \uh {\widehat{\u}}

\begin{document}

\title[Extracting Negative Curvature From Noise]{First-order Stochastic Algorithms for Escaping From Saddle Points in Almost Linear Time}
 \author{\Name{Yi Xu}$^\natural$\Email{yi-xu@uiowa.edu}\\
 \Name{Rong Jin}$^\dagger$ \Email{jinrong.jr@ailibaba-inc.com}\\
\Name{Tianbao Yang}$^\natural$\Email{tianbao-yang@uiowa.edu}\\
\addr$^\natural$ Department of Computer Science, The University of Iowa, Iowa City, IA 52242 \\
\addr$^\dagger$Alibaba Group, Bellevue, WA 98004
}

\maketitle
\vspace*{-0.5in}
\begin{center}{First version: November 03, 2017}\end{center}
\vspace*{0.2in}

\begin{abstract}
In this paper, we consider first-order methods for solving stochastic non-convex optimization problems. To escape from saddle points, we propose first-order procedures to extract negative curvature from the Hessian matrix through a principled sequence starting from noise, which are referred to {\bf NEgative-curvature-Originated-from-Noise or NEON} and of independent interest. 
The proposed procedures enable one to design purely first-order stochastic algorithms for escaping from non-degenerate saddle points with a much better time complexity (almost linear time in  the problem's dimensionality). In particular, we develop a general framework of {\bf first-order stochastic algorithms} with a second-order convergence guarantee based on our new technique and existing algorithms that may only converge to a first-order stationary point. For finding a nearly {\bf second-order stationary point} $\x$ such that $\|\nabla F(\x)\|\leq \epsilon$ and $\nabla^2 F(\x)\geq -\sqrt{\epsilon}I$ (in high probability), the best time complexity of the presented algorithms is $\widetilde O(d/\epsilon^{3.5})$, where $F(\cdot)$ denotes the objective function and $d$ is the dimensionality of the problem. To the best of our knowledge, this is the first theoretical result of first-order stochastic algorithms with an almost linear time in terms of problem's dimensionality for finding second-order stationary points, which is  even competitive with  existing stochastic algorithms hinging on the second-order information.
\end{abstract}

\section{Introduction}
The problem of interest in this paper is given by 
\begin{align}\label{eqn:opt}
\min_{\x\in\R^d}F(\x)\triangleq \E_{\xi}[f(\x; \xi)]
\end{align}
where $\xi$ is a random variable and $f(\x; \xi)$ is a random smooth non-convex function of $\x$.  It is notable that our developments are also  applicable to a finite-sum problem with a very large number of components: 
\begin{align}\label{eqn:opt2}
\vspace*{-0.05in}
\min_{\x\in\R^d}\frac{1}{n}\sum_{i=1}^nf(\x; \xi_i)
\vspace*{-0.05in}
\end{align}
The above formulations play an important role for solving non-convex problems in machine learning, e.g.,  deep learning~\citep{goodfellow2016deep}.

A popular choice of algorithms for solving~(\ref{eqn:opt}) is (mini-batch) stochastic gradient descent (SGD) method and its variants~\citep{DBLP:journals/siamjo/GhadimiL13a}. However, these algorithms do  not necessarily  guarantee to escape from a saddle point  (more precisely a non-degenerate saddle point) $\x$ satisfying that: $\nabla F(\x) = 0$ and  the  minimum eigen-value of $\nabla^2 F(\x))$ is less than 0. Recently, new variants of SGD by adding isotropic noise into the stochastic gradient use the following update for escaping from a saddle point:
\begin{align}\label{eqn:noise}
\vspace*{-0.05in}
\x_{t+1} = \x_t - \eta_t(\nabla f(\x_t; \xi_t) + n_t),
\vspace*{-0.05in}
\end{align}
where $n_t\in\R^d$ is an isotropic noise vector either sampled from the sphere of a Euclidean ball~\citep{pmlr-v40-Ge15} (the corresponding variant is referred to as noisy-SGD) or a Gaussian distribution~\citep{DBLP:conf/colt/ZhangLC17} (the corresponding variant is referred to as stochastic gradient Langevin dynamics (SGLD)). These two works provide rigorous analyses of the noise-injected update~(\ref{eqn:noise}) for escaping from a saddle point. Unfortunately, both variants suffer from a polynomial time complexity with a super-linear dependence on the dimensionality $d$ (at least a power of $4$), which renders them not practical for optimizing deep neural networks with millions of parameters. 

\begin{table*}[t]
		\caption{Comparison with existing stochastic algorithms for achieving an $(\epsilon, \gamma)$-SSP to~(\ref{eqn:opt}), where $p$ is a number at least $4$, IFO (incremental first-order oracle) and ISO (incremental second-order oracle) are terminologies borrowed from~\citep{reddi2017generic}, representing $\nabla f(\x; \xi)$ and $\nabla^2 f(\x; \xi)\v$ respectively,  $T_h$ denotes the runtime of ISO and $T_g$ denotes the runtime of IFO. In this paper, $\widetilde O(\cdot)$ hides a poly-logarithmic factor. SM refers to stochastic momentum methods.  For $\gamma$, we only consider as lower as $\epsilon^{1/2}$. 
		 }
		\centering
		\label{tab:2}
		\scalebox{0.74}{\begin{tabular}{lllll}
			\toprule
			algorithm & oracle & convergence in &time complexity\\
			&& expectation or high probability&\\
			\midrule
			Noisy SGD~\citep{pmlr-v40-Ge15} &IFO&$(\epsilon, \epsilon^{1/4})$-SSP, high probability&$\widetilde O\left(T_gd^p\epsilon^{-4}\right)$\\ 
			SGLD~\citep{DBLP:conf/colt/ZhangLC17} &IFO&$(\epsilon, \epsilon^{1/2})$-SSP, high probability&$\widetilde O\left(T_gd^p\epsilon^{-4}\right)$\\ 
	\midrule
			Natasha2~\citep{natasha2} &IFO  + ISO&$(\epsilon, \epsilon^{1/2})$-SSP, expectation&$\widetilde O\left( T_g\epsilon^{-3.5}+T_h\epsilon^{-2.5} \right)$\\ 
			Natasha2~\citep{natasha2} &IFO  + ISO&$(\epsilon, \epsilon^{1/4})$-SSP, expectation&$\widetilde O\left( T_g\epsilon^{-3.25}+T_h\epsilon^{-1.75} \right)$\\ 
            		SNCG~\citep{DBLP:journals/corr/SNCG}&IFO + ISO&$(\epsilon, \epsilon^{1/2})$-SSP, high probability&$\widetilde O\left(T_g\epsilon^{-4} + T_h\epsilon^{-2.5}\right)$\\
                        	SVRG-Hessian~\citep{reddi2017generic}&IFO + ISO&$(\epsilon, \epsilon^{1/2})$-SSP, high probability&$\widetilde O\left(T_g(n^{2/3}\epsilon^{-2} + n\epsilon^{-1.5}) \right.$\\
	& & & ~~~~$\left.+ T_h(n\epsilon^{-1.5} + n^{3/4}\epsilon^{-7/4})\right)$\\

            \midrule
			NEON-SGD, NEON-SM ({\bf this work})&IFO &$(\epsilon, \epsilon^{1/2})$-SSP, high probability&$\widetilde O\left(T_g\epsilon^{-4}\right)$\\
			NEON-SCSG ({\bf this work})&IFO &$(\epsilon, \epsilon^{1/2})$-SSP, high probability&$\widetilde O\left(T_g\epsilon^{-3.5}\right)$\\
                         NEON-SCSG ({\bf this work})&IFO &$(\epsilon, \epsilon^{4/9})$-SSP, high probability&$\widetilde O\left(T_g\epsilon^{-3.33}\right)$\\
  			NEON-Natasha ({\bf this work})&IFO &$(\epsilon, \epsilon^{1/2})$-SSP, expectation&$\widetilde O\left(T_g\epsilon^{-3.5}\right)$\\
                        NEON-Natasha ({\bf this work})&IFO &$(\epsilon, \epsilon^{1/4})$-SSP, expectation&$\widetilde O\left(T_g\epsilon^{-3.25}\right)$\\
                        	NEON-SVRG ({\bf this work})&IFO &$(\epsilon, \epsilon^{1/2})$-SSP, high probability&$\widetilde O\left(T_g\left(n^{2/3}\epsilon^{-2} + n\epsilon^{-1.5} + \epsilon^{-2.75}\right)\right)$\\

			\bottomrule
		\end{tabular}}
	\end{table*}
	
On the other hand, second-order information carried by the Hessian has been utilized to escape from a saddle point, which usually yields an almost linear time complexity in terms of the dimensionality $d$ under the assumption  that the Hessian-vector product can be performed in a linear time. There have emerged a bulk of studies for non-convex optimization by utilizing the second-order information through Hessian or Hessian-vector products~\citep{nesterov2006cubic,DBLP:conf/stoc/AgarwalZBHM17,DBLP:journals/corr/CarmonDHS16,peng16inexacthessian,Cartis2011,Cartis2011b,clement17,DBLP:journals/corr/SNCG,DBLP:journals/corr/noisynegative}, though many of them focused on solving deterministic non-convex optimization with few exceptions~\citep{natasha2,DBLP:conf/icml/KohlerL17,DBLP:journals/corr/SNCG,reddi2017generic}. Although in practice, the Hessian-vector product can be estimated by a finite difference approximation using two gradient evaluations, the rigorous analysis of algorithms using such noisy approximation for calculating the Hessian-vector product in non-convex optimization remains unsolved, and heuristic approaches may suffer from numerical issues. 

So far, the developments of first-order algorithms and second-order algorithms for escaping from saddle points are largely disconnected. The existing analysis for first-order algorithms for escaping from saddle points based on adding isotropic noise is quite involved, and the underlying mechanism of using noise to escape from saddle points is not made explicit in the existing studies, which hinder further developments and improvements for stochastic non-convex optimization. In this paper, we will provide a novel perspective of adding noise into the first-order information by treating it as a tool for extracting negative curvature from the Hessian. To provide a formal analysis, we present  simple first-order procedures (NEON) inspired by the perturbed gradient method proposed in~\citep{jin2017escape} and its connection with the power method~\citep{doi:10.1137/0613066} for computing the largest eigen-vector of a matrix starting from a random noise vector. 
Our key results  show that the proposed NEON procedures can extract a negative curvature direction of the Hessian, which itself is stronger than just showing the objective value decrease around saddle points in existing studies~\citep{jin2017escape,pmlr-v40-Ge15,DBLP:journals/corr/Levy16a}. 
Our main contributions are: 
\begin{itemize}
\item We provide a novel perspective of adding noise into the first-order information by treating it as a tool for extracting negative curvature from the Hessian,  thus connect the existing two classes of methods for escaping from saddle points. We provide a formal analysis of simple procedures based on gradient descent and accelerated gradient method for exacting a negative curvature direction  from the Hessian.
\item We develop a general framework of first-order  algorithms for stochastic non-convex optimization by combining  the proposed first-order NEON procedures to extract negative curvature with existing first-order stochastic algorithms that aim at a first-order critical point. We also establish the time complexities of several interesting instances of  our general framework for finding a nearly $(\epsilon, \gamma)$-second-order stationary point (SSP), i.e., 
$\|\nabla F(\x)\|\leq \epsilon$, and $\lambda_{\min}(\nabla^2 F(\x))\geq -\gamma$,
where $\|\cdot\|$ represents Euclidean norm of a vector and $\lambda_{\min}(\cdot)$ denotes the minimum eigen-value. 
\end{itemize}

To the best of our knowledge, the developed stochastic first-order algorithms  are the first with an almost linear time complexity in terms of the problem's dimensionality for finding an SSP of a stochastic non-convex optimization problem. In addition, the dependence on the prescribed first-order optimality  level $\epsilon$ (i.e., $\|\nabla F(\x)\|\leq \epsilon$) is competitive and even better than standard SGD and their noisy variants~\citep{pmlr-v40-Ge15,DBLP:conf/colt/ZhangLC17}, and also competitive with existing second-order stochastic algorithms. A summary of our results and existing results is presented in Table~\ref{tab:2}.

\section{Other Related Work}

SGD and its many variants (e.g., mini-batch SGD and stochastic momentum (SM) methods) have been analyzed for stochastic non-convex optimization~\citep{DBLP:journals/siamjo/GhadimiL13a,DBLP:journals/mp/GhadimiL16,DBLP:journals/mp/GhadimiLZ16,yangnonconvexmo}.
The iteration complexities of all these algorithms is $O(1/\epsilon^4)$ for finding a first-order stationary point (FSP) (in expectation $\E[\|\nabla f(\x)\|^2_2]\leq \epsilon^2$  or in high probability).   Recently, there are some improvements for stochastic non-convex optimization. \citet{DBLP:journals/corr/LeiJCJ17} proposed a first-order stochastic algorithm (named SCSG) using the variance-reduction technique, which enjoys an iteration complexity of $O(1/\epsilon^{-10/3})$ for finding an FSP (in expectation), i.e., $\E[\|\nabla f(\x)\|^2_2]\leq \epsilon^2$. \citet{natasha2} proposed a variant of SCSG (named Natasha1.5) with the same convergence and complexity. 
An important application of NEON is that  previous stochastic algorithms that have a first-order convergence guarantee can be strengthened to enjoy a second-order convergence guarantee by leveraging the proposed first-order NEON procedures to escape from saddle points. We will analyze several  algorithms by combining the updates of SGD, SM, mini-batch SGD and SCSG with the proposed NEON. 
 
Several recent works~\citep{DBLP:journals/corr/SNCG,natasha2,reddi2017generic} propose to strengthen existing first-order stochastic algorithms to have  second-order convergence guarantee by leveraging the second-order information. \citet{DBLP:journals/corr/SNCG} used  mini-batch SGD, \citet{reddi2017generic} used SVRG for a finite-sum problem, and \citet{natasha2} used a similar algorithm to SCSG for their first-order algorithms. The second-order methods used in these studies for computing negative curvature can be replaced by the proposed NEON procedures. It is notable although  a generic approach for stochastic non-convex optimization was proposed in \citep{reddi2017generic},  its requirement on the first-order stochastic algorithms precludes many interesting algorithms such as SGD, SM, and SCSG. 
Stronger convergence guarantee (e.g., converging to a global minimum) of stochastic algorithms has been studied in~\citep{DBLP:conf/icml/HazanLS16} for a certain family of problems, which is beyond the setting of the present work. 

We note that the proposed NEON can be considered as a noisy Power method applied to the Hessian matrix or an augmented matrix, where the Hessian-vector product is estimated  by a finite difference approximation using only the first-order information. However, we cannot use existing analysis of noisy power method~\citep{DBLP:conf/colt/BalcanDWY16,DBLP:conf/nips/HardtP14} to analyze the proposed NEON because that all existing analysis focus on gap-dependent complexity. Although~\citet{DBLP:conf/colt/BalcanDWY16} have established  a gap-independent complexity, it is not strong enough for analyzing the proposed NEON due to its strong requirement on the noise matrix. To address this challenge, we are enlightened by the connection between the perturbed gradient method proposed in~\citep{jin2017escape} and the power method, and we follow a similar route in~\citep{jin2017escape} but with some simplification and customization  to achieve our goal. 
 

\setlength{\belowdisplayskip}{6pt} \setlength{\belowdisplayshortskip}{2pt}
\setlength{\abovedisplayskip}{6pt} \setlength{\abovedisplayshortskip}{2pt}
\section{Preliminaries}
In this section, we present some notations and preliminaries. Let $\|\cdot\|$ denote the Euclidean norm of a vector and $\|\cdot\|_2$ denote the spectral norm of a matrix. 
A function $f(\x)$ has a Lipschitz continuous gradient if it is differentiable and for any $\x, \y\in\R^d$, there exists $L_1>0$ such that 
$|f(\x) - f(\y) - \nabla f(\y)^{\top}(\x - \y)|\leq \frac{L_1}{2}\|\x - \y\|^2$.
A function $f(\x)$ has a Lipschitz continuous Hessian if it is twice differentiable and for any $\x, \y\in\R^d$, there exists $L_2>0$ such that 
\begin{align}
|f(\x) - f(\y) - \nabla f(\y)^{\top}(\x - \y) - \frac{1}{2}(\x - \y)^{\top}\nabla^2 f(\y)(\x - \y)|\leq \frac{L_2}{6}\|\x - \y\|^3.
\end{align}
It also implies that 
\begin{align}\label{eqn:HessianC}
\|\nabla f(\x + \u) - \nabla f(\x) - \nabla^2 f(\x)\u\|\leq \frac{L_2\|\u\|^2}{2}.
\end{align}
We first make the following assumptions regarding the stochastic non-convex optimization problem~(\ref{eqn:opt}). 
\begin{ass}\label{ass:1}For the  problem~(\ref{eqn:opt}), we assume that 
\begin{itemize} 
\item[(i).] every random function $f(\x; \xi)$ is twice differentiable, and it has Lipschitz continuous gradient, i.e., there exists $L_1>0$ such that  $\|\nabla f(\x; \xi) - \nabla f(\y; \xi)\|\leq L_1\|\x - \y\|$, and   has a Lipschitz continuous Hessian, i.e.,  there exists $L_2>0$ such that $\|\nabla^2 f(\x; \xi) - \nabla^2 f(\y; \xi)\|_2\leq L_2\|\x - \y\|$.
\item[(ii).]  given an initial point $\x_0$, there exists $\Delta<\infty$ such that $F(\x_0) - F(\x_*)\leq \Delta$, where $\x_*$ denotes the global minimum of~(\ref{eqn:opt}).
\item[(iii).] (optional) there exists $G>0$ such that $\mathbb{E}[\exp(\|\nabla f(\x; \xi) - \nabla F(\x)\|^2/G^2)]\leq \exp(1)$ holds. 
\item[(iv).] (optional) there exists $V>0$ such that $\mathbb{E}[\|\nabla f(\x; \xi) - \nabla F(\x)\|^2]\leq V$ holds. 
\end{itemize}
\end{ass}
It is notable that the first two assumptions are standard assumptions for non-convex optimization in order to establish second-order convergence,  the third assumption is necessary for high probability analysis, and the last assumption is needed for expectational analysis. 

Next, we discuss a second-order method based on Hessian-vector products to escape from a non-degenerate saddle point $\x$ of a function $f(\x)$ that satisfies  
\begin{align}\label{eqn:saddle}
 \|\nabla f(\x)\|\leq \epsilon, \quad \lambda_{\min}(\nabla^2 f(\x))\leq -\gamma,
\end{align}
which can be found in many previous studies~\citep{clement17,DBLP:journals/corr/noisynegative,DBLP:journals/corr/CarmonDHS16}.
The method is based on a negative curvature ({\bf NC for short is used in the sequel}) direction $\v\in\R^d$ that satisfies $\|\v\|=1$ and 
\begin{align}\label{eqn:nc}
\v^{\top}\nabla^2 f(\x)\v\leq - c\gamma
\end{align}
where $c>0$ is a constant. Given such a vector $\v$, we can update the solution according to 
\begin{align}\label{eqn:ncu}
\x_+ = \x - \frac{c\gamma}{L_2}\text{sign}(\v^\top\nabla f(\x))\v,
\end{align}
or
\begin{align}\label{eqn:ncu2}
\x'_+ = \x - \frac{c\gamma}{L_2}\bar\xi\v,
\end{align}
if $\nabla f(\x)$ is not available, where $\bar\xi\in\{1, -1\}$ is a Rademacher random variable. The following lemma establishes that the objective value of $\x_+$ or $\x'_+$ is less than that of $\x$ by a sufficient amount, which makes it possible to escape from the saddle point $\x$.
\begin{lemma}\label{lemma:NCD:onestep}
For $\x$ and $\v$ satisfying~(\ref{eqn:saddle}) and~(\ref{eqn:nc}), let $\x_+$, $\x'_+$ be given in~(\ref{eqn:ncu}) and~(\ref{eqn:ncu2}), then we have
\begin{align*}
    &f(\x)- f(\x_+)\geq  \frac{c^3\gamma^3}{3L_2^2}, \quad\E[f(\x)- f(\x'_+)]\geq  \frac{c^3\gamma^3}{3L_2^2}.
\end{align*}
\end{lemma}
To compute a NC direction $\v$ that satisfies~(\ref{eqn:nc}), we can employ the Lanczos method or the Power method  for computing the maximum eigen-vector of the matrix $(I - \eta \nabla^2 f(\x))$, where $\eta L_1\leq 1$ such that $I - \eta\nabla^2 f(\x)\succeq 0$.  
The Power method starts with a random vector $\v_1\in\R^d$ (e.g., drawn from a uniform distribution over the unit
sphere) and iteratively compute 
\begin{align}\label{eqn:power}
\v_{\tau+1} = (I - \eta\nabla^2 f(\x))\v_{\tau}, \tau=1, \ldots, t
\end{align}

Following the results in~\citep{doi:10.1137/0613066}, it can be shown that if $\lambda_{\min}(\nabla^2 f(\x))\leq -\gamma$, then with at most $ \frac{\log(d/\delta^2)L_1}{\gamma}$ Hessian-vector products, the Power method~(\ref{eqn:power}) finds a vector $\hat\v_t = \v_t/\|\v_t\|$ such that $\hat\v_t^{\top}\nabla^2 f(\x)\hat\v_t\leq -\frac{\gamma}{2}$ holds with high probability $1-\delta$. Similarly, the Lanczos method (e.g., Lemma 11 in~\citep{clement17}) can find such a vector $\hat \v_t$ with a lower  number of Hessian-vector products, i.e., $ \min(d, \frac{\log(d/\delta^2)\sqrt{L_1}}{2\sqrt{2\varepsilon}})$.

In practice, Hessian-vector products may be approximated by using gradients. For example, \cite{conf/icml/Martens10,DBLP:journals/corr/CarmonDHS16} suggest to approximate a Hessian-vector by 
\begin{align*}
\nabla^2 f(\x) \v = \lim_{h\rightarrow 0}\frac{\nabla f(\x + h\v) - \nabla f(\x)}{h}.
\end{align*}
However, such a heuristic approach might be unstable when $h$ is very small. On the other hand, the convergence guarantee of the Power method and the Lanczos method with noise in approximating $\nabla^2 f(\x)\v$ is still under explored. 
\section{Extracting NC From Noise}
In this section, we will present and analyze  first-order procedures for computing a NC direction $\v$ that satisfies~(\ref{eqn:nc}) at any point $\x$ whose Hessian has a negative eigen-value less than $-\gamma$. The proposed procedures NEON in this section are applied to a twice differentiable  non-convex function $f(\x)$, whose gradient can be computed.  First, we propose a gradient descent based method for extracting NC, which achieves a similar {\it iteration complexity} to the Power method.  Second, we present an accelerated gradient method to extract the NC to match the iteration complexity of the Lanczos method. Finally, we discuss the application of these procedures for stochastic non-convex optimization using the mini-batching technique.   

The NEON is inspired by the perturbed gradient descent (PGD) method proposed in the seminal work~\citep{jin2017escape} and its connection with the Power method as discussed shortly. Around a {saddle point} $\x$ that satisfies~(\ref{eqn:saddle}), the PGD method first generates a random noise vector $\hat\e$ from the sphere of an Euclidean ball with a proper radius, then starts with a noise perturbed solution $\x_0 = \x+ \hat\e$, the PGD generates the following sequence of solutions: 
\begin{align}\label{eqn:noisyGD}
\x_\tau = \x_{\tau-1}  - \eta \nabla f(\x_{\tau-1}).
\end{align}
To establish a connection with the Power method~(\ref{eqn:power}) and motivate the proposed NEON, let us define another sequence of $\xh_\tau = \x_\tau - \x$. Then we have the following recurrence for $\xh_\tau$
\begin{align*}
\xh_\tau = \xh_{\tau-1}  - \eta \nabla f(\xh_{\tau-1}+\x), \quad\tau=1, \ldots, t
\end{align*}
It is clear that for $\tau=1, \ldots, t$,
\begin{align*}
\xh_\tau = \xh_{\tau-1}   - \eta\nabla f(\x)- \eta (\nabla f(\xh_{\tau-1}+\x) - \nabla f(\x)).
\end{align*}
To understand the above update, we adopt the following approximation: from~(\ref{eqn:saddle}) we know that $\nabla f(\x)\approx 0$, and from the Lipschitz continuous Hessian condition~(\ref{eqn:HessianC}), we can see that $\nabla f(\xh_{\tau-1}+\x) - \nabla f(\x)\approx \nabla^2 f(\x)\xh_{\tau-1}$ as long as $\|\xh_{\tau-1}\|$ is small, then for $\tau=1, \ldots, t$ 
\begin{align*}
\xh_\tau \approx \xh_{\tau-1}  - \eta \nabla^2 f(\x)\xh_{\tau-1} = (I - \eta \nabla^2 f(\x))\xh_{\tau-1}.
\end{align*}
It is obvious that the above approximated recurrence is close to the the sequence generated by the Power method with the same starting random vector $\hat\e = \v_1$. This intuitively explains that why the updated solution $\x_t = \x + \xh_t$ can decrease the objective value due to that $\xh_t$ is close to a NC of the Hessian $\nabla^2 f(\x)$. 

\setlength{\textfloatsep}{9pt}

\begin{algorithm}[t]
\caption{NEgative-curvature-Originated-from-Noise (NEON): NEON$(f, \x, t, \mathcal F, r)$}\label{alg:ncn}
\begin{algorithmic}[1]
\STATE \textbf{Input}:  $f, \x, t, \mathcal F, r$

\STATE Generate $\u_0$ randomly from the sphere of an Euclidean ball of radius $r$
\FOR{$\tau=0,\ldots, t$}
\STATE $\u_{\tau+1} = \u_\tau - \eta (\nabla f(\x+\u_{\tau}) - \nabla f(\x))$
\ENDFOR
\IF{$\min_{\|\u_{\tau}\|\leq U }\hat f_\x(\u_\tau) \leq -2.5\mathcal F$}
\STATE  let ${\tau'}= \arg\min_{\tau,\|\u_\tau\|\leq U}\hat f_\x(\u_\tau)$
\STATE {\bf return} $\u_{\tau'}$
\ELSE
\STATE {\bf return} $0$ 
\ENDIF

\end{algorithmic}
\end{algorithm}

To provide a formal analysis of this argument, we will first analyze the following recurrence for $\tau=1,\ldots$: 
\begin{align}\label{eqn:noisypower}
\u_{\tau} = \u_{\tau-1} - \eta (\nabla f(\x + \u_{\tau-1}) - \nabla f(\x)),
\end{align}
starting with  a random noise vector $\u_0$, which is drawn from the sphere of an Euclidean ball with a proper radius $r$. It is notable that the recurrence in~(\ref{eqn:noisypower}) is slightly different from that in~(\ref{eqn:noisyGD}). We emphasize that this simple change is useful  for extracting  the NC at any points whose Hessian has a negative eigen-value not just at non-degenerate saddle points, which can be used in some stochastic or deterministic algorithms~\citep{natasha2,DBLP:journals/corr/CarmonDHS16,clement17,DBLP:journals/corr/noisynegative}.  The proposed procedure NEON based on the above sequence for finding a NC direction of $\nabla^2 f(\x)$ is presented in Algorithm~\ref{alg:ncn}, where $\hat f_{\x}(\u)$ is defined in~(\ref{eqn:objn}). 
The following theorem states our key result for extracting the NC from the Hessian using noise initiated sequence~(\ref{eqn:noisypower}).
\begin{thm}\label{thm:main:GD}
For any $\gamma\in(0,1)$ and a sufficiently small $\delta\in(0,1)$, let $\x$ be a point such that $\lambda_{\min}(\nabla^2 f(\x))\leq -\gamma$. For any constant $\hat c\geq 18$, there  exists a constant $c_{\max}$ that depends on $\hat c$, such that  if NEON is called with $t = \hat c\frac{\log (dL_1 /(\gamma\delta))}{\eta \gamma}$,  $\mathcal F=\eta  \gamma^3 L_1L_2^{-2} \log^{-3}(dL_1 /(\gamma\delta))$, $r = \sqrt{\eta }\gamma^2L_1^{-1/2}L_2^{-1}\log^{-2}(dL_1 /(\gamma\delta))$, $U = 4\hat c(\sqrt{\eta L_1}\mathcal F/L_2)^{1/3}$ and a  constant $\eta\leq c_{\max}/L_1$, then  with high probability $1-\delta$ it returns a vector $\u$ such that 
$\frac{\u^{\top}\nabla^2 f(\x)\u}{\|\u\|^2}\leq -\frac{\gamma}{8\hat c^2 \log(dL_1 /(\gamma\delta))}\leq - \widetilde\Omega(\gamma)$. 
If NEON returns $\u\neq 0$, then the above inequality must hold; if NEON returns $0$, we can conclude that $\lambda_{\min}(\nabla^2 f(\x))\geq -\gamma$ with high probability $1- O(\delta)$. 
\end{thm}
{\bf Remark:} 
The above lemma shows that at any point $\x$ whose Hessian has a negative eigen-value (including non-degenerate saddle points), NEON can find a NC of $\nabla^2 f(\x)$ with high probability.

\subsection{Finding NC by Accelerated Gradient Method}
Although NEON provides a similar guarantee for extracting a NC as that provided by the Power method, but its iteration complexity $O(1/\gamma)$ is worse than that of the Lanczos method, i.e., $O(1/\sqrt{\gamma})$. In this subsection, we present a result matches $O(1/\sqrt{\gamma})$ of the Lanczos method. 

Let us recall the sequence~(\ref{eqn:noisypower}), which is essentially  an application of gradient descent (GD) method to the following objective function:
\begin{align}\label{eqn:objn}
\hat f_\x(\u) = f(\x+\u) - f(\x) - \nabla f(\x)^{\top}\u. 
\end{align}
In the sequel, we write $\hat f_\x(\u) = \hat f(\u)$, where the dependent $\x$ should be clear from the context. 
By the Lipschitz continuous Hessian condition, we have that
\begin{align*}
\frac{1}{2}\u^{\top}\nabla^2 f(\x)\u - \frac{L_3}{6}\|\u\|^3 \leq \hat f(\u).
\end{align*}
It implies that if $\hat f(\u)$ is sufficiently less than zero and $\|\u\|$ is not too large, then $\frac{\u^{\top}\nabla^2 f(\x)\u}{\|\u\|^2}$ will be sufficiently less than zero.


A natural question to ask is whether the convergence of GD update of NEON can be accelerated by accelerated gradient (AG) methods. 
It is well-known from convex optimization literature that AG methods can accelerate the convergence of GD method for smooth problems. Recently, several studies have explored AG methods for non-convex optimization~\citep{Li:2015:APG:2969239.2969282,corrACGWright,DBLP:conf/icml/CarmonDHS17,AGNON}. Notably, \citet{corrACGWright} analyzed the behavior of AG methods near strict saddle points and investigated
the rate of divergence from a strict saddle point for toy quadratic problems. \citet{AGNON} analyzed a single-loop algorithm based on  Nesterov's AG method for deterministic non-convex optimization.  However, none of these studies provide an explicit complexity guarantee on extracting NC from the Hessian matrix for a general non-convex problem.
Inspired by these studies, we will show that 
Nesterov's AG (NAG) method~\citep{opac-b1104789} when applied the function $\hat f(\u)$ can find a NC  with a complexity of $\widetilde O(1/\sqrt{\gamma})$.


The updates of NAG method applied to the function $\hat f(\u)$ at a given point $\x$ is given by 
\begin{equation}\label{eqn:neonpn}
\begin{aligned}
\y_{\tau+1}& = \u_{\tau} - \eta \nabla \hat f(\u_{\tau})  \\
\u_{\tau+1}& =\y_{\tau+1} +  \zeta(\y_{\tau+1} - \y_{\tau})
\end{aligned}
\end{equation}
where the term $\zeta(\y_{\tau+1} - \y_{\tau})$ is the momentum term, and $\zeta\in(0, 1)$ is the momentum parameter.  
The proposed algorithm based on the NAG method (referred to as NEON$^+$) for extracting NC of a Hessian matrix $\nabla^2 f(\x)$ is presented in Algorithm~\ref{alg:neonp},  where  $$\Delta_\x(\y_\tau,\u_\tau)=\hat f_\x(\y_\tau)- \hat f_\x(\u_\tau) -  \nabla \hat f_\x(\u_\tau)^{\top}(\y_\tau - \u_\tau),$$ and NCFind is a procedure that returns a NC by searching over the history $\y_{0:\tau}, \u_{0:\tau}$ shown in Algorithm~\ref{alg:ncfind}. The condition check in Step 4 is to detect easy cases such that NCFind can easily find a NC in historical solutions without continuing the update. It is notable that NCFind is similar to a procedure called Negative Curvature Exploitation (NCE) in~\citep{AGNON}. However, the difference is that NCFind is tailored to finding a negative curvature satisfying~(\ref{eqn:nc}), while NCE in~\citep{AGNON} is for ensuring a decrease on a modified objective. 

\begin{algorithm}[t]
\caption{Accelerated Gradient methods for Extracting NC from Noise: \neonp$(f, \x, t, \mathcal F, U,  \zeta, r)$}\label{alg:neonp}
\begin{algorithmic}[1]
\STATE \textbf{Input}:  $f, \x, t, \mathcal F,  U, \zeta, r$
\STATE Generate $\y_0=\u_0$ randomly from the sphere of an Euclidean ball of radius $r$
\FOR{$\tau=0,\ldots, t$}
\IF{$\Delta_\x(\y_\tau, \u_\tau)< - \frac{\gamma}{2}\|\y_\tau - \u_\tau\|^2$}
\STATE {\bf return} $\v=$NCFind($\y_{0:\tau}, \u_{0:\tau}$)
\ENDIF
\STATE compute $(\y_{\tau+1}, \u_{\tau+1})$  by~(\ref{eqn:neonpn})
\ENDFOR
\IF{$\min_{\|\y_{\tau}\|\leq U }\hat f_\x(\y_\tau)  \leq -2\mathcal F$ }
\STATE  let ${\tau'}= \arg\min_{\tau, \|\y_\tau\|\leq  U}\hat f_\x(\y_\tau)$
\STATE  {\bf return} $\y_{\tau'}$
\ELSE 
\STATE {\bf return} $0$ 
\ENDIF
\end{algorithmic}
\end{algorithm}
\begin{algorithm}[t]
\caption{NCFind $(\y_{0:\tau}, \u_{0:\tau})$}\label{alg:ncfind}
\begin{algorithmic}[1]
\IF{$\min_{j=0,\ldots, \tau}\|\y_{j} - \u_j\|\geq \zeta\sqrt{6\eta\mathcal{F}}$}
\STATE {\bf return} $\y_j$, \\where $j = \min\{j': \|\y_{j'}-\u_{j'}\|\geq \zeta\sqrt{6\eta\mathcal{F}}\}$
\ELSE 
\STATE {\bf return} $\y_\tau- \u_{\tau}$
\ENDIF
\end{algorithmic}
\end{algorithm}

Before presenting our main result, we first present an informal analysis about why \neonp~is faster than \neon~based on an eigen-gap argument.  This analysis is from a perspective of Power method for computing dominating eigen-vectors of a matrix in light of our goal is to compute a NC of a Hessian matrix corresponding to negative eigen-values. In particular, by ignoring the error of $\nabla\hat f(\u_\tau) = \nabla f(\x+\u_\tau) - \nabla f(\x)$ for approximating $H=\nabla^2 f(\x)$, the update in~(\ref{eqn:neonpn}) can be written as 
\begin{align*}
\uh_{\tau+1} = \left[\begin{array}{cc}(1 + \zeta)(I - \eta H)& -\zeta (I - \eta H)\\ I & 0 \end{array}\right] \uh_{\tau-1}. 
\end{align*}
where
$\uh_{\tau+1} = (\u_{\tau+1}^{\top}, \u_\tau^{\top})^{\top}$. 
The above sequence can be considered as an application of the Power method to an augmented matrix:
$A =  \left[\begin{array}{cc}(1 + \zeta)(I - \eta H)& -\zeta (I - \eta H)\\ I & 0 \end{array}\right]$.
According to existing analysis of the Power method for computing top eigen-vectors in the top-$k$ eigen-space of the matrix $A$~(e.g., \cite{DBLP:conf/nips/HardtP14}), the iteration complexity depends on the eigen-gap $\Delta_k$ of $A$  in an order of $\widetilde O(1/\Delta_k)$, where the eigen-gap is defined as the difference between the $k$-th largest  eigen-value of $A$ and the $(k+1)$-th largest eigen-value of $A$. 
We can show that if the eigen-values of $H$ satisfy $\lambda_1\leq\lambda_2\ldots\leq\lambda_k\leq -\gamma<0\leq \lambda_{k+1}\leq\lambda_d$, then by choosing $\zeta = 1  - \sqrt{\eta\gamma}\in(0,1)$, the eigen-gap $\Delta_k$ of $A$ corresponding to its top-$k$  eigen-space is at least $\sqrt{\eta\gamma}/2$. This informal analysis helps understand that why NEON$^+$ with NAG has a complexity of $\widetilde O(1/\sqrt{\gamma})$. 
Nevertheless, a formal analysis yields the following result. 
\begin{thm}\label{thm:main:AGD}
For any $\gamma\in(0,1)$ and a sufficiently small $\delta\in(0,1)$, let $\x$ be a point such that $\lambda_{\min}(\nabla^2 f(\x))\leq -\gamma$. For any constant $\hat c\geq 43$, there  exists a constant $c_{\max}$ that depends on $\hat c$, such that  if \neonp is called with $t = \sqrt{\frac{\hat c\log (dL_1 /(\gamma\delta))}{\eta \gamma}}$,  $\mathcal F=\eta  \gamma^3 L_1L_2^{-2} \log^{-3}(dL_1 /(\gamma\delta))$, $r = \sqrt{\eta }\gamma^2L_1^{-1/2}L_2^{-1}\log^{-2}(dL_1 /(\gamma\delta))$, $U=12\hat c(\sqrt{\eta L_1}\mathcal F/L_2)^{1/3}$, a small constant $\eta\leq c_{\max}/L_1$, and a momentum parameter  $\zeta = 1 - \sqrt{\eta \gamma}$, then  with high probability $1-\delta$ it returns a vector $\u$ such that 
$\frac{\u^{\top}\nabla^2 f(\x)\u}{\|\u\|^2} \leq -\frac{\gamma}{72\hat c^2 \log(dL_1 /(\gamma\delta))} \leq  -\widetilde\Omega(\gamma)$. 
If \neonp returns $\u\neq 0$, then the above inequality must hold; if \neonp returns $0$, we can conclude that $\lambda_{\min}(\nabla^2 f(\x))\geq -\gamma$ with high probability $1- O(\delta)$. 
\end{thm}

\subsection{Stochastic Approach for Extracting NC}
In this subsection, we will present a stochastic approach for extracting NC. For simplicity, we refer to both NEON and NEON$^+$ as NEON. The challenge in employing NEON for finding a NC for the original function $F(\x)$ in~(\ref{eqn:opt}) is that we cannot evaluate the gradient of $F(\x)$ exactly. To address this issue, we resort to the mini-batching technique. 

Let $\S = \{\xi_1, \ldots, \xi_m\}$ denote a set of random samples and define: 
$F_{\S}(\x) = \frac{1}{m}\sum_{i=1}^m f(\x; \xi_i)$. 
Then we apply NEON to $F_\S(\x)$ for finding an approximate  NC $\u_\S$ of $\nabla^2 F_\S(\x)$. 
Below, we show that as long as $m$ is sufficiently large, $\u_\S$ is also an approximate NC of $\nabla^2 F(\x)$. 

\begin{lemma}\label{lem:NCNmini}
For a sufficiently small $\delta\in(0,1)$ and $\hat c \geq 43$, let $m\geq \frac{16L_1^2\log(6d/\delta)}{s^2\gamma^2}$, where $s=\frac{\log^{-1}(3dL_1/(2\gamma\delta))}{(12\hat c)^2}$ is a proper small constant. If $\lambda_{\min}(\nabla^2 F(\x))\leq -\gamma$, there exists $c>0$ such that with probability $1-\delta$, NEON($F_\S, \x, t, \mathcal F, r)$  returns  a vector $\u_\S$ such that 
 $\frac{\u_\S^{\top}\nabla^2F(\x)\u_\S}{\|\u_\S\|^2}\leq - c\gamma$, 
where $c = (12 \hat c)^{-2}\log^{-1}(3dL_1/(2\gamma\delta))$. If  NEON($F_\S, \x, t, \mathcal F, r)$   returns $0$, then with high probability $1-O(\delta)$ we have $\lambda_{\min}(\nabla^2 F(\x))\geq - 2\gamma$. In either case, NEON terminates with an IFO complexity of $\widetilde O(1/\gamma^3)$ or  $\widetilde O(1/\gamma^{2.5})$ corresponding to Algorithm~\ref{alg:ncn} and Algorithm~\ref{alg:neonp}, respectively.
\end{lemma}
\paragraph{Simulation.}Before ending this section, we present some simulations to verify  the proposed NEON procedures for extracting NC. To this end, we consider minimizing non-linear least square loss with a non-convex regularizer for classification, i.e., 
\begin{align*}
F(\x) = \sum_{i=1}^d\frac{x_i^2}{1+x_i^2} + \frac{\lambda}{n}\sum_{i=1}^n(b_i - \sigma(\x^{\top}\a_i))^2
\end{align*}
where $b_i\in\{0, 1\}$ denotes the label and $\a_i\in\R^d$ denotes the feature vector of the $i$-th data, $\lambda>0$ is a trade-off parameter, and $\sigma(\cdot)$ is a sigmoid function. We generate a random vector $\x\sim \mathcal N(0, I)$ as the target point to construct $\hat F_\x(\u)$ and compute a NC of $\nabla^2 F(\x)$. We use a binary classification data named gisette from the libsvm data website that has $n=6000$ examples and $d=5000$ features, and set $\lambda=3$ in our simulation to ensure there is significant NC from the non-linear least-square loss. The step size $\eta$ and initial radius in NEON procedures are set to be 0.01 and the momentum parameter in NEON$^+$ is set to be 0.9. 

First, we compare different NEON procedures (NEON, NEON$^+$) with second-order methods that use Hessian-vector products, namely the Power method and the Lanczos method, where the Hessian-vector products are calculated exactly. The result is shown in Figure~\ref{fig:a} whose $y$-axis denotes the value of $\uh^{\top}H\uh$,  where $\uh$ represents the found normalized NC vector and $H=\nabla^2 F(\x)$ is the Hessian matrix. Second, we compare different NEON procedures and second-order methods with the stochastic versions of NEON (denoted by NEONst and NEON$^+$st in the figure) that run on a sub-sampled data with a sample size of 100 in Figure~\ref{fig:b}, where the $x$-axis represents the \#IFO calls. Please note that the solid red curve corresponding to NEON$^+$st in Figure~\ref{fig:b} terminates earlier due to that NCFind is executed.   Several observations follow: (i) NEON performs similarly to the Power method (the two curves overlap in the figure); (ii) NEON$^+$ has a faster convergence than NEON; 
 (iv) the stochastic versions of NEON and NEON$^+$ can quickly find a good NC directions than their full versions in terms of IFO complexity and are even competitive with the Lanczos method. Finally, we note that the curves may look different for different target vector $\x$ but the overall trend looks similarly. We include several more results in the supplement.

\begin{figure}[t] 
\centering
	\subfigure[NEON on full data\label{fig:a}]{\includegraphics[scale=0.37]{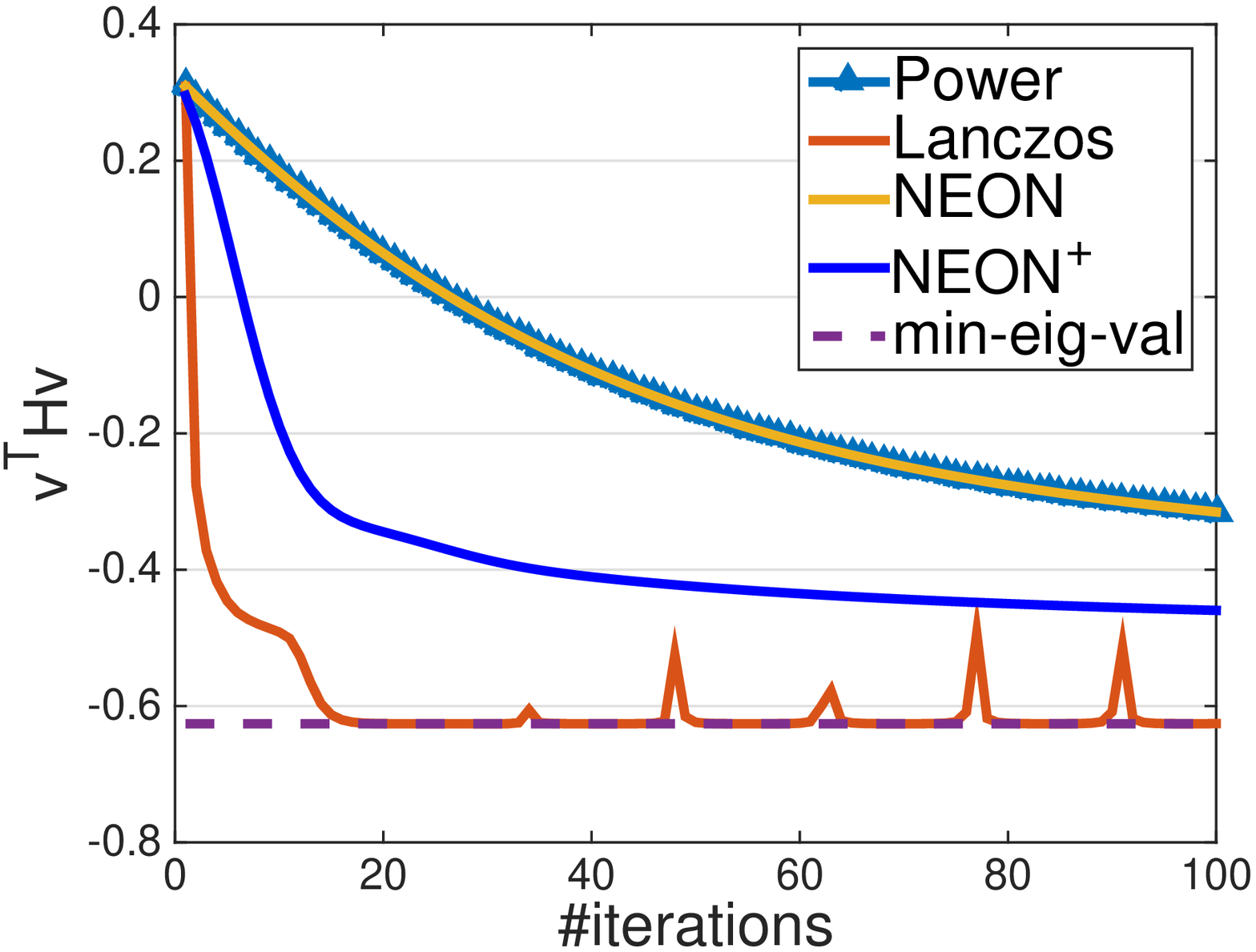}}
    \subfigure[NEON on sub-sampled data\label{fig:b}]{\includegraphics[scale=0.37]{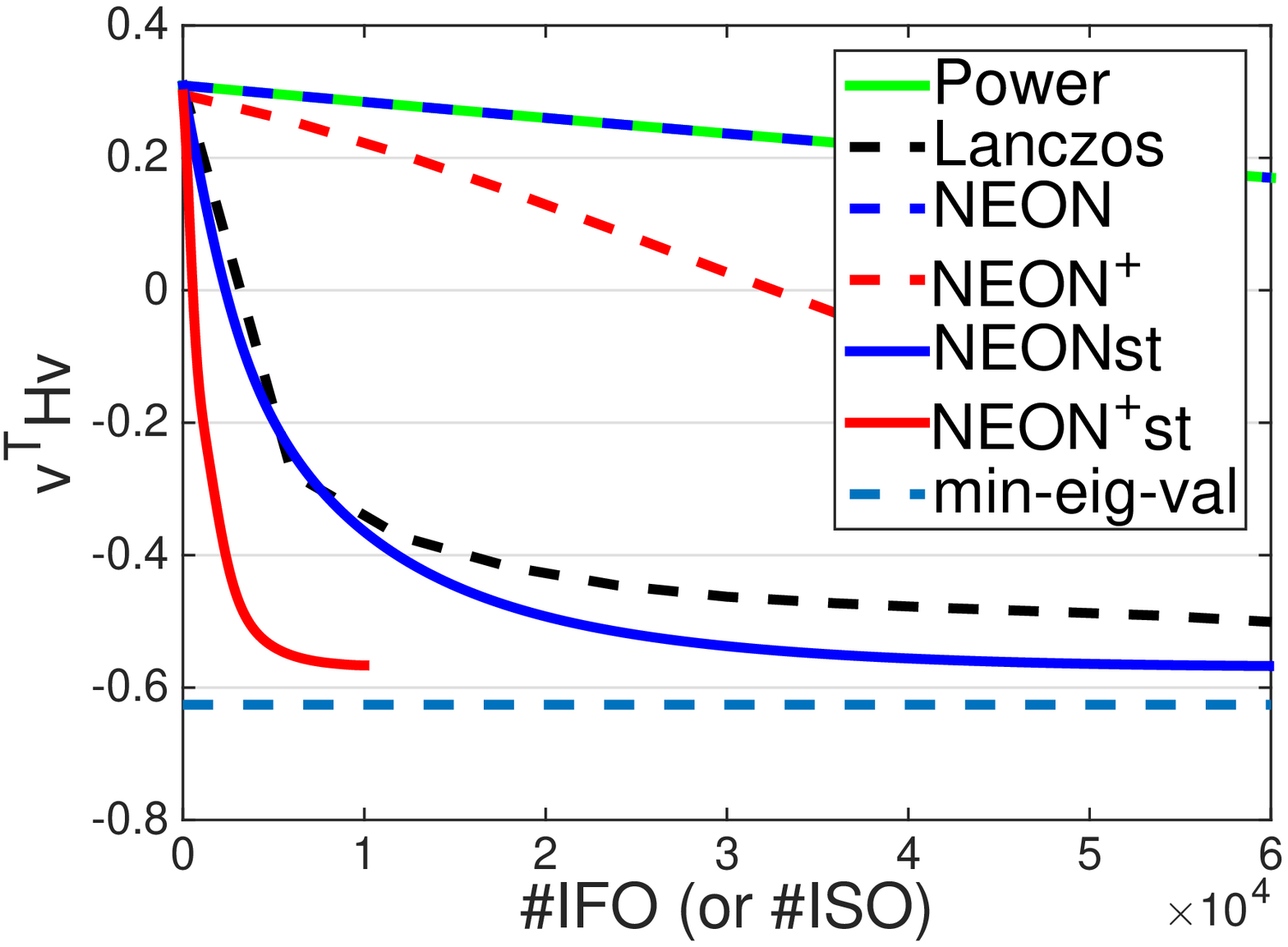}}  
\caption{Comparison between different NEON procedures and Second-order Methods}
\label{fig:neon}
\end{figure}

\section{First-order Algorithms for Stochastic Non-Convex Optimization}
In this section, we will first present a general framework for promoting existing first-order stochastic algorithms denoted by $\mathcal A$ to have a second-order convergence, which is shown in Algorithm~\ref{alg:general}. The proposed NEON is used for escaping from a saddle point. It should be noted that Algorithm~\ref{alg:general} is abstract depending on how to implement Step 3, how to check the first-order condition, and how to set the step size parameter $\bar\xi$ in Step 9. In practice, one might use some heuristic approaches to check the first-order condition or postpone checking the first-order condition until a sufficiently large number of Step 3 is executed. For $\bar\xi$, one may use the mini-batch stochastic gradient on $\y_j$ denoted by $\g(\y_j)$ to set $\bar\xi = \text{sign}(\u_j^{\top}\g(\y_j))$ similar to that in~(\ref{eqn:ncu}). Of course other parameters in Step 9 should be tuned as well in practice.

\begin{algorithm}[t]
\caption{NEON-$\mathcal A$
}\label{alg:general}
\begin{algorithmic}[1]
\STATE \textbf{Input}:  $\x_1$ and other parameters of algorithm $\mathcal A$
\FOR{$j=1,2,\ldots,$}
\STATE Compute $(\y_j, \z_j) = \mathcal A(\x_j)$
\IF{first-order condition of $\y_j$ not met}
\STATE let $\x_{j+1} = \z_j$
\ELSE
\STATE Compute $\u_j = \text{NEON}(F_{\S_2}, \y_j, t, \mathcal F, r)$
\STATE {\bf if} $\u_j=0$ {\bf return} $\y_j$
\STATE {\bf else} let $\x_{j+1} = \y_j - \frac{c\gamma\bar\xi}{L_2}\frac{\u_j}{\|\u_j\|}$ 
\ENDIF
\ENDFOR
\end{algorithmic}
\end{algorithm}

For theoretical interest, we will analyze Algorithm~\ref{alg:general} with a Rademacher random variable $\bar\xi\in\{1,-1\}$ and its three main components satisfying the following properties. 
\begin{prop}\label{property:1}
(1) Step 7 - Step 9 guarantees that if $\lambda_{\min}(\nabla^2 F(\y_j))\leq - \gamma$, there exists $C>0$ such that  $\E[F(\x_{j+1}) - F(\y_j)]\leq - C\gamma^3$. Let the total IFO complexity of Step 7 - Step 9 be $T_n$. (2) There exists  a first-order stochastic algorithm $\mathcal A(\x_j)$ that satisfies:
\begin{equation}\label{eqn:A}
\begin{aligned}
\|\nabla F(\y_j)\|\geq \epsilon\rightarrow \E[F(\z_j) - F(\x_j)]\leq - \varepsilon(\epsilon, \alpha)\\
\|\nabla F(\y_j)\|\leq \epsilon \rightarrow \E[F(\y_j) - F(\x_j)]\leq C\gamma^3/2
\end{aligned}
\end{equation}
where $\varepsilon(\epsilon, \alpha)$ is a function of $\epsilon$ and a parameter $\alpha>0$. Let the total IFO complexity of $\mathcal A(\x)$ be $T_a$. (3) the check of first-order condition can be implemented by using a mini-batch of samples $\mathbf S$, i.e., $\|\nabla F_{\mathbf S}(\y_j)\|\leq \epsilon$, where $\mathbf S$ is independent of $\y_j$ such that $\|\nabla F(\y_j) - \nabla F_{\mathbf S}(\y_j)\|\leq \epsilon/2$. Let the IFO complexity of checking the first-order condition be $T_c$.
\end{prop}
{\bf Remark:} Property (1)  can be guaranteed by Lemma~\ref{lem:NCNmini} and Lemma~\ref{lemma:NCD:onestep}. When using NEON, $T_n=\widetilde O(1/\gamma^3)$ and when using NEON$^+$, $T_n=\widetilde O(1/\gamma^{2.5})$. For Property (2), we will analyze several interesting algorithms.  Property (3) can be guaranteed by Lemma 3 in the supplement under Assumption~(\ref{ass:1}) (iii) with $T_c = \widetilde O(\frac{1}{\epsilon^2})$.

\begin{algorithm}[t]
\caption{SM: $(\x_0, \eta, \beta, s, t)$
}\label{alg:sm}
\begin{algorithmic}[1]
\FOR{$\tau=0, 1, 2,\ldots,t$}
\STATE Compute $\x_{\tau+1}$ according to~(\ref{eqn:um})
\STATE Compute $\x^+_{\tau+1}$ according to~(\ref{eqn:um2})
\ENDFOR
\STATE {\bf return} $(\x^+_{\tau'}, \x^+_{t+1})$, where $\tau'\in\{0, \ldots, t\}$ is a randomly generated.
\end{algorithmic}
\end{algorithm}

Based on the above properties, we have the following convergence guarantee of Algorithm~\ref{alg:general}. 
\begin{thm}\label{thm:gen}
Assume Properties~\ref{property:1} hold. Then with high probability $1-\delta$, NEON-$\mathcal A$ terminates with 
a total IFO complexity of $\widetilde O(\max(\frac{1}{\varepsilon(\epsilon, \alpha)}, \frac{1}{\gamma^3})(T_n + T_a + T_c))$. Upon termination, with high probability $\|\nabla F(\y_j)\|\leq O(\epsilon)$ and $\lambda_{\min}(\nabla^2 F(\y_j))\geq - 2\gamma$, where $\widetilde O(\cdot)$ hides logarithmic factors of $d$ and $1/\delta$, and problem's other constant parameters.
\end{thm}
Next, we present corollaries of Theorem~\ref{thm:gen} for several instances of $\mathcal A$, including stochastic gradient descent (SGD) method, stochastic momentum (SM) methods,  mini-batch SGD (MSGD), and SCSG. 
SGD and its momentum variants (including stochastic heavy-ball (SHB) method and stochastic Nesterov's accelerated gradient (SNAG) method) are popular stochastic algorithms for solving a stochastic non-convex optimization problem. We will consider them in a unified framework as established in~\citep{yangnonconvexmo}. The updates are 
\begin{equation}\label{eqn:um}
\begin{aligned}
\xh_{\tau+1} &  = \x_\tau - \eta \nabla f(\x_\tau; \xi_{\tau})\\
\xh^s_{\tau+1} & = \x_\tau - s\eta \nabla f(\x_\tau; \xi_{\tau})\\
\x_{\tau+1} & = \xh_{\tau+1} + \beta(\xh^s_{\tau+1} - \xh^s_\tau)
\end{aligned}
\end{equation}
for $\tau=0,\ldots, t$ and $\xh^s_0= \x_0$, where $\beta\in(0,1)$ is a momentum constant, $\eta$ is a step size, $s=0, 1, 1/(1-\beta)$ corresponds to SHB, SNAG and SGD. Define another sequence $\x^+_\tau$ with $\x^+_0=\x_0$ and  
\begin{equation}\label{eqn:um2}
\begin{aligned}
\x^+_{\tau} &  = \x_\tau + \p_\tau, \tau\geq 1,\\
\p_{\tau} & = \frac{\beta}{1-\beta}(\x_\tau - \x_{\tau-1} - s\eta\nabla f(\x_{\tau-1}; \xi_{\tau-1})).
\end{aligned}
\end{equation}
We can implement $\A$ by Algorithm~\ref{alg:sm} and have the following result. 
\begin{cor}\label{thm:SM}
Let $\mathcal A(\x_j)$ be implemented by~Algorithm~\ref{alg:sm} with 
$t=  \Theta(1/\epsilon^{2}), \eta=\Theta(\epsilon^2), \beta\in(0,1), s\in(0, 1/(1-\beta))$. Then $T_a =  O(1/\epsilon^2)$ and $\varepsilon(\epsilon, \alpha)=\Theta(\epsilon^2)$.  
Suppose that $\gamma\geq \epsilon^{2/3}$ and $\E[\|\nabla f(\x; \xi)\|^2]$ is bounded for $s\neq 1/(1-\beta)$. Then with high probability, NEON-SM finds an $(\epsilon, \gamma)$-SPP with a total IFO complexity of $\widetilde O(\max(\frac{1}{\epsilon^2}, \frac{1}{\gamma^3})(T_n + \frac{1}{\epsilon^2}))$. 
\end{cor}
{\bf Remark:} When $\gamma=\epsilon^{1/2}$, NEON-SM has an IFO complexity of  $\widetilde O(\frac{1}{\epsilon^4})$.

MSGD computes $(\y_j, \z_j)$ by 
\begin{equation}\label{eqn:SGD}
\begin{aligned}
\z_j &= \x_j - L_1^{-1}\nabla F_{\mathbf S_1}(\x_j),\quad \y_j  =\x_j
\end{aligned}
\end{equation}
 where $\mathbf S_1$ is a set of samples independent of $\x_j$.  
\begin{cor}\label{thm:SGM}
Let $\mathcal A(\x_j)$ be implemented by~(\ref{eqn:SGD}) with 
$|\mathbf S_1|= \widetilde O(1/\epsilon^{2})$. Then $T_a = \widetilde O(1/\epsilon^2)$ and $\varepsilon(\epsilon, \alpha)=\frac{\epsilon^2}{4L_1}$.  
With high probability, NEON-MSGD finds an $(\epsilon, \gamma)$-SPP with a total IFO complexity of $\widetilde O(\max(\frac{1}{\epsilon^2}, \frac{1}{\gamma^3})(T_n + 1/\epsilon^2))$. 
\end{cor}
{\bf Remark:} Compared to Corollary~\ref{thm:SM}, there is no requirement on $\gamma\geq\epsilon^{2/3}$, which is due to that MSGD can guarantee  that $\E[F(\y_j) - F(\x_j)]\leq 0$. 


SCSG was proposed in~\citep{DBLP:journals/corr/LeiJCJ17}, which only provides a first-order convergence guarantee. SCSG runs with multiple epochs, and each epoch uses similar updates as SVRG with three distinct features: (i) it was applied to a sub-sampled function $F_{\mathbf S_1}$; (ii) it allows for using a mini-batch samples of size $b$ independent of $\mathbf S_1$ to compute stochastic gradients;  (ii) the number of updates of each epoch is a random  number following a geometric distribution dependent on $b$ and $|\mathbf S_1|$. These features make each SGCG epoch denoted by SCSG-epoch$(\x, \mathbf S_1, b)$ have an expected IFO complexity of $T_a=O(|\mathbf S_1|)$. Due to limitation of space, we present  SCSG-epoch$(\x, \mathbf S_1, b)$ in the supplement. For using SCSG, $\y_j$ and $\z_j$ are 
\begin{equation}\label{eqn:SCSG}
\begin{aligned}
\y_j &=\text{SCSG-epoch}(\x_j, \mathbf S_1, b),\quad\z_j  =\y_j
\end{aligned}
\end{equation}

\begin{cor}\label{thm:SCSG}
Let $\mathcal A(\x_j)$ be implemented by~(\ref{eqn:SCSG}) 
with $|\mathbf S_1| =  \widetilde O\left(\max(1/\epsilon^{2}, 1/(\gamma^{9/2}b^{1/2}))\right)$. Then $\varepsilon(\epsilon, \alpha)=\Omega(\epsilon^{4/3}/b^{1/3})$ and $\E[T_a] = \widetilde O\left(\max(1/\epsilon^{2}, 1/(\gamma^{9/2}b^{1/2}))\right)$.  
With high probability, NEON-SCSG finds an $(\epsilon, \gamma)$-SSP with an expected total IFO complexity of $\widetilde O(\max(\frac{b^{1/3}}{\epsilon^{4/3}}, \frac{1}{\gamma^3})(T_n + 1/\epsilon^{2} +  1/(\gamma^{9/2}b^{1/2})))$. 
\end{cor}
{\bf Remark:} When $\gamma=\epsilon^{1/2}, b=1/\sqrt{\epsilon}$, NEON-SCSG has an expected IFO complexity of  $\widetilde O(\frac{1}{\epsilon^{3.5}})$. When $\gamma\geq \epsilon^{4/9}, b=1$,  NEON-SCSG has an expected IFO complexity of  $\widetilde O(1/\epsilon^{3.33})$.

Finally, we mention that the proposed NEON can be used in existing second-order stochastic algorithms that require a NC direction as a substitute of second-order methods~\citep{natasha2, reddi2017generic}. \citet{natasha2} developed Natasha2, which uses second-order online Oja's algorithm for finding the NC.  \citet{reddi2017generic} developed a stochastic algorithm for solving a finite-sum problem by using SVRG and a second-order stochastic algorithm for computing the NC.   
We can replace the second-order methods for computing a NC in these algorithms by the proposed NEON or NEON$^+$, with the resulting algorithms referred to as NEON-Natasha and NEON-SVRG. 
It is a simple exercise to derive the convergence results in Table~\ref{tab:2}, which is left to interested readers. 

\section{Conclusions}
We have proposed novel first-order procedures to extract  NC from a Hessian matrix by using a noise-initiated sequence, which are of independent interest. A general framework for promoting a first-order stochastic algorithm to enjoy a second-order convergence is also proposed. Based on the proposed general framework, we designed several first-order stochastic algorithms with state-of-the-art second-order convergence guarantee.  

\bibliography{all,ref}
\bibliographystyle{icml2018}

\appendix
\section{Proof of Lemma~\ref{lemma:NCD:onestep}}
\begin{proof}
Let $\eta = \frac{c\gamma}{L_2} \text{sign}(\v^\top\nabla f(\x))$ be the step size, so that $\x_+ = \x - \eta \v$. By the $L_2$-Lipschitz continuous Hessian of $f(\x)$, we have
\begin{align*}
|f(\x_+)-f(\x)  +  \eta\v^\top\nabla f(\x)-\frac{1}{2}\eta^2\v^\top\nabla^2f(\x)\v| \leq\frac{L_2}{6}\|\eta\v\|^3.
\end{align*}
By noting that $\eta\v^\top\nabla f(\x)\geq 0$, we have
\begin{align*}
    f(\x)- f(\x_+) \geq & \eta\v^\top\nabla f(\x) -\frac{1}{2}\eta^2\v^\top\nabla^2f(\x)\v - \frac{L_2}{6}\|\eta\v\|^3 \\
    \geq & \frac{c^3\gamma^3}{2L_2^2} - \frac{c^3\gamma^3}{6L_2^2} = \frac{c^3\gamma^3}{3L_2^2},
\end{align*}
where the last inequality uses that $\v^{\top}\nabla^2 f(\x)\v\leq - c\gamma$, $\|\v\| = 1$, and the definition of $\eta$. 

For the case of $\x'_+$, we have $\x_+' = \x - \eta \v$, where $\eta =\bar\xi\frac{c\gamma}{L_2}$. 
Similarly, we can prove the expectation result of $\x_+'$. Then 
\begin{align*}
    \E[f(\x)- f(\x_+)]
    \geq & \E[\eta\v^\top\nabla f(\x) -\frac{1}{2}\eta^2\v^\top\nabla^2f(\x)\v - \frac{L_2}{6}\|\eta\v\|^3] \\
    \geq& \frac{c^3\gamma^3}{2L_2^2} - \frac{c^3\gamma^3}{6L_2^2} = \frac{c^3\gamma^3}{3L_2^2},
\end{align*}
where we use $\E[\eta]=0$, $\E[\eta^2]=\frac{c^2\gamma^2}{L_2^2}$,  $\E[|\eta|^3]=\frac{c^3\gamma^3}{L_2^3}$, and $\|\v\| = 1$. 
\end{proof}

\section{Concentration inequalities}
We first present some concentration inequalities of random vectors and random matrices. 
Below, we let $\S_1$ and $\S_2$ to denote a set of random samples that are generated independently of $\x$. 
\begin{lemma}[\citep{DBLP:journals/mp/GhadimiLZ16}, Lemma 4]\label{lemma:concen:grad}
Suppose Assumption~\ref{ass:1}(iii) holds. Let $\nabla F_{\S_1}(\x)= \frac{1}{|\S_1|} \sum_{\z_i\in \S_1}\nabla f(\x; \z_i)$. For any $\epsilon, \delta \in (0,1)$, 
$\x\in\R^d$ when $|\S_1 |\geq \frac{2G^2(1+8\log(1/\delta))}{\epsilon^2}$, we have $\Pr(\|\nabla F_{\S_1}(\x) - \nabla F(\x)\|\leq \epsilon ) \geq 1-\delta$.
\end{lemma}
\begin{lemma}[\citep{peng16inexacthessian}, Lemma 4]\label{lemma:concen:hessian}
Suppose Assumption~\ref{ass:1}(i) holds. Let $\nabla^2F_{\S_2}(\x) = \frac{1}{|\S_2|} \sum_{\z_i\in \S_2}\nabla^2 f(\x; \z_i)$. For any $\epsilon, \delta \in (0,1), \x\in\R^d$,  when $|\S_2 |\geq  \frac{16L_1^2 \log(2d/\delta)}{\epsilon^2}$, we have $\Pr(\|\nabla^2F_{\S_2}(\x) - \nabla^2 F(\x)\|_2\leq \epsilon ) \geq 1-\delta$.
\end{lemma}

{\bf Claim}. In the following analysis, when we say high probability, it means there is a probability $1-\delta$ with a small enough $\delta<0$.  In many cases, we prove an inequality for one iteration with high probability, which implies the final result with high probability using union bound with a finite number of iterations. Instead of repeating this argument, we will simply assume this is done. We can always set the $\delta$ in the  involved parameters (in the logarithmic part) small enough to make our argument fly.

\section{Proof of Lemma~\ref{lem:NCNmini}}
Due to that the proof of Theorem \ref{thm:main:GD} and Theorem \ref{thm:main:AGD} is lengthy, we postpone them into the end of the supplement.
\begin{proof}
By using a matrix concentration inequality in Lemma~\ref{lemma:concen:hessian}, if $m \geq  \frac{16L_1^2 \log(6d/\delta)}{s^2\gamma^2}$,  with probability $1-\delta/3$ we have
\begin{align*}
\|\nabla^2 F(\x) - \nabla^2 F_\S(\x)\|_2\leq s\gamma.
\end{align*}
Since $\|L_1I - \nabla^2 F(\x)\|_2= L_1 - \lambda_{\min}(\nabla^2 F(\x))$ and $\|L_1I - \nabla^2 F_\S(\x)\|_2= L_1 - \lambda_{\min}(\nabla^2 F_\S(\x))$, if $\lambda_{\min}(\nabla^2 F(\x))\leq -\gamma$, with probability $1-\delta/3$,
\begin{align*}
(L_1 - \lambda_{\min}(\nabla^2 F(\x))) - (L_1 - \lambda_{\min}(\nabla^2 F_\S(\x)))
\leq& \|(L_1I - \nabla^2 F(\x)) -(L_1I - \nabla^2 F_\S(\x)) \|_2 \leq s\gamma.
\end{align*}
As a result, with probability $1 -\delta/3$, $\lambda_{\min}(\nabla^2 F_\S(\x))\leq -\gamma + s\gamma\leq  - 3\gamma/4$ with $s = \frac{\log^{-1}(3dL_1/(2\gamma\delta))}{(12
\hat c)^2}\leq 1/4$. 

(i) The NEON applied to $F_\S$ can generate $\u_\S$ with probability $1-2\delta/3$ (over randomness in $\S$ and NEON) such that 
\begin{align*}
 \frac{\u_\S^{\top}\nabla^2 F_\S(\x)\u_\S}{\|\u_\S\|^2}\leq &\frac{-2\mathcal F'}{(4\hat c\mathcal P')^2} =  \frac{-\gamma}{8\hat c^2 \log(3dL_1 /(2\gamma\delta))}
\leq  \frac{-\gamma}{72\hat c^2 \log(3dL_1 /(2\gamma\delta))} \leq -\widetilde \Omega(\gamma),
\end{align*}
where $\mathcal F'=\eta  L_1 \gamma^3 L_2^{-2} \log^{-3}(3dL_1 /(2\gamma\delta))$ and $\mathcal P'=\sqrt{\eta L_1}  \gamma L_2^{-1} \log^{-1}(3dL_1 /(2\gamma\delta))$. 

(ii) Similarly, the NEON$^+$ applied to $F_\S$ can generate $\u_\S$ with probability $1-2\delta/3$ (over randomness in $\S$ and NEON$^+$) such that 
\begin{align*}
& \frac{\u_\S^{\top}\nabla^2 F_\S(\x)\u_\S}{\|\u_\S\|^2} \leq \frac{ -\gamma}{72\hat c^2 \log(3dL_1 /(2\gamma\delta))} \leq -\widetilde \Omega(\gamma). 
\end{align*}

As a result, both for NEON and NEON$^+$, with  probability $1-\delta$ (over randomness in $\S$ and NEON or NEON$^+$), 
\begin{align*}
 \left|\frac{\u_\S^{\top}\nabla^2 F(\x)\u_\S}{\|\u_\S\|^2} - \frac{\u_\S^{\top}\nabla^2 F_\S(\x)\u_\S}{\|\u_\S\|^2}\right|
\leq \|\nabla^2 F(\x) - \nabla^2F_\S(\x)\|_2 \leq s\gamma.
\end{align*}
Hence,
\begin{align*}
 \frac{\u_\S^{\top}\nabla^2 F(\x)\u_\S}{\|\u_\S\|^2} \leq - \frac{\gamma}{72\hat c^2 \log(3dL_1 /(2\gamma\delta))} + s\gamma = - c\gamma ,
\end{align*}
where $s=\frac{\log^{-1}(3dL_1/(2\gamma\delta))}{(12
\hat c)^2}$. 
If NEON or NEON$^+$ returns $0$, we then terminate the algorithm, which guarantees that $\lambda_{\min}(F_\S(\x))\geq -\gamma$ with high probability and therefore $-\lambda_{\min}(\nabla^2 F(\x))\leq \gamma + s\gamma \leq  2\gamma$ with high probability.
\end{proof}

\section{Proof of Theorem~\ref{thm:gen}}
\begin{proof}
To prove the convergence of the generic algorithm, we need to prove the total number of iterations NEON-$\mathcal A$ before termination. Upon  termination, it then holds that $\|\nabla F_{\S_1}(\y_j)\|\leq \epsilon$, $\lambda_{\min}(\nabla^2 F_{\S_2}(\x_j))\geq -\gamma$. By concentration inequalities, we have $\lambda_{\min}(\nabla^2 F(\y_j))\geq - 2\gamma$ and $\|\nabla F(\y_j)\|\leq O(\epsilon)$ hold with high probability. Before termination, let us consider 
two cases based on the first-order condition at point $\y_j$: (1) $\|\nabla F(\y_j)\|\geq \epsilon$ and (2) $\|\nabla F(\y_j)\| < \epsilon$. 

For the first case, by (\ref{eqn:A}) we have 
\begin{align*}
 \E[F(\z_j) - F(\x_j)]\leq - \varepsilon(\epsilon, \alpha).
\end{align*}
Since the the first-oder condition of $\y_j$ not met in this case, then $\x_{j+1} = \z_j$, so that
\begin{align}\label{inqe:case:1}
 \E[F(\x_{j+1}) - F(\x_j)]\leq - \varepsilon(\epsilon, \alpha).
\end{align}
For the second case, by (\ref{eqn:A}) we have
\begin{align*}
\E[F(\y_j) - F(\x_j)]\leq \frac{C\gamma^3}{2}
\end{align*}
Before termination, NEON returns $\u_j \neq 0$, then it satisfies $\frac{\u_j^{\top}\nabla^2F(\y_j)\u_j}{\|\u_j\|^2}\leq - c\gamma $ with high probability according to Lemma~\ref{lemma:NCD:onestep} and Lemma~\ref{lem:NCNmini}, i.e., $\Pr\left(\frac{\u_j^{\top}\nabla^2 F(\y_j)\u_j}{\|\u_j\|^2}\leq -c\gamma\right)\geq 1 - \delta$.  Similar to the analysis in Lemma~\ref{lemma:NCD:onestep}, we have
\begin{align*}
\E[F(\y_j  + c\gamma \bar \xi  \hat\u/L_2) - F(\y_j)] 
\leq \E\left[\frac{1}{2}\eta^2\hat\u^\top\nabla^2 F(\y_j)\hat\u  +  \frac{L_2}{6}\|\eta\hat \u\|^3\right],
\end{align*}
where $\hat\u = \u_j/\|\u_j\|$ and $\eta= c\gamma \bar \xi/L_2$. 
\begin{align*}
& \E[\hat\u^\top\nabla^2 F(\y_j)\hat\u]\\
=& \E[\hat\u^\top\nabla^2 F(\y_j)\hat\u|\hat\u^\top\nabla^2 F(\y_j)\hat\u\leq -c\gamma] \cdot \Pr(\hat\u^\top\nabla^2 F(\y_j)\hat\u\leq -c\gamma) \\
&+ \E[\hat\u^\top\nabla^2 F(\y_j)\hat\u|\hat\u^\top\nabla^2 F(\y_j)\hat\u>-c\gamma]\cdot \Pr(\hat\u^\top\nabla^2 F(\y_j)\hat\u>-c\gamma)\\
\leq& -c\gamma\Pr(\hat\u^\top\nabla^2 F(\y_j)\hat\u\leq -c\gamma) + L_1\delta \\
\leq &- c\gamma (1-\delta) + L_1\delta = -c\gamma + (c\gamma + L_1)\delta.
\end{align*} 
With $\delta\leq c\gamma/2(c\gamma + L_1)$, we have $\E[\hat\u^\top\nabla^2 F(\y_j)\hat\u] \leq - c\gamma/2$. As a result, 
\begin{align*}
\E[F(\x_{j+1}) - F(\y_j)] \leq \E\left[\frac{1}{2}\eta^2\hat\u^\top\nabla^2 F(\y_j)\hat\u  +  \frac{L_2}{6}\|\eta\hat \u\|^3\right] \leq - \frac{c^3\gamma^3}{12L_2^2}
\end{align*}
By selecting $C$ such that $C < \frac{c^3}{6L_2^2}$, we will have
\begin{align}\label{inqe:case:2}
\E[F(\x_{j+1}) - F(\x_j)]\leq - C\gamma^3/2.
\end{align}
By (\ref{inqe:case:1}) and (\ref{inqe:case:2}) we get
\begin{align}\label{inqe:case:1:2}
\E[F(\x_{j+1}) - F(\x_j)]\leq - \underbrace{\min(\varepsilon(\epsilon,\alpha), C\gamma^3/2)}\limits_{\theta}.
\end{align}

Next, we will show within $\widetilde O\left(\max\left(\frac{1}{\varepsilon(\epsilon,\alpha)}, \frac{1}{\gamma^3}\right)\log(1/\zeta)\right)$ outer iterations  of NEON-$\mathcal A$, there exists at least on  $\y_j$ such that $\lambda_{\min}(\nabla^2 F(\y_j))\geq -\gamma$ and $\|\nabla F(\y_j)\|\leq \epsilon$ with high probability. This analysis is similar to that of Theorem 14 in~\cite{pmlr-v40-Ge15}. As a result, at such a $\y_j$ NEON-$\mathcal A$ terminates with a high probability. 
Let us consider three cases:
\begin{align*}
\mathcal C_1 =\{& \y_j | \|\nabla F(\y_j)\|\geq\epsilon \text{ for some } j>0\}\\
\mathcal C_2 =\{& \y_j |\|\nabla F(\y_j)\|\leq \epsilon \text{ and } \lambda_{\min}(\nabla^2 F(\y_j))\leq -\gamma  \text{ for some } j>0\}\\
\mathcal C_3 =\{& \y_j |\|\nabla F(\y_j)\|\leq \epsilon \text{ and } \lambda_{\min}(\nabla^2 F(\y_j))\geq -\gamma  \text{ for some } j>0\}
\end{align*}
Clearly, Case 3 is our favorable case, thus we need to carefully study the occurrences of Cases 1 and 2. 
Let us define an event $\mathcal E_j =\{ \exists i\leq j, \y_i \in \mathcal C_3 \}$, and then $\bar{\mathcal E}_j =\{ \forall i\leq j, \y_i \notin \mathcal C_3 \}$. It is easy to show that $Pr({\bar{\mathcal E}_{j}})\leq Pr({\bar{\mathcal E}_{j-1}})$. Then we have 
\begin{align*}
&\E [F(\x_{j+1}) I_{\bar{\mathcal E}_{j}}] - \E [F(\x_j) I_{\bar{\mathcal E}_{j-1}}] 
=\E [F(\x_{j+1}) - F(\x_j) | {\bar{\mathcal E}_{j}}] Pr(\bar{\mathcal E}_{j}) 
+  \E [F(\x_j) I_{\bar{\mathcal E}_{j}}] - \E [F(\x_j) I_{\bar{\mathcal E}_{j-1}}]
\end{align*}
It is easy to bound the first term in R.H.S by $\E [F(\x_{j+1}) - F(\x_j) | {\bar{\mathcal E}_{j}}] Pr(\bar{\mathcal E}_{j})\leq  -\theta Pr(\bar{\mathcal E}_{j}).$ To bound the second term, we need following two results.

(a) Let consider different situations of $ I_{\bar{\mathcal E}_{j}}$ and $ I_{\bar{\mathcal E}_{j-1}}$.
\begin{itemize}
\item{ If $ I_{\bar{\mathcal E}_{j}} = 1$, then $ I_{\bar{\mathcal E}_{j-1}}=1$, so that
\begin{align*}
&\E [F(\x_j) I_{\bar{\mathcal E}_{j}}] - \E [F(\x_j) I_{\bar{\mathcal E}_{j-1}}] 
= \E [F(\x_j)] - \E [F(\x_j)] = 0.
\end{align*}
}
\item{
If $ I_{\bar{\mathcal E}_{j}} = 0$, then $ I_{\bar{\mathcal E}_{j-1}}$ could be either $1$ or $0$.
When $ I_{\bar{\mathcal E}_{j-1}}=0$, then
\begin{align*}
\E [F(\x_j) I_{\bar{\mathcal E}_{j}}] - \E [F(\x_j) I_{\bar{\mathcal E}_{j-1}}] = 0.
\end{align*}
When $ I_{\bar{\mathcal E}_{j-1}}=1$, then
\begin{align*}
&\E [F(\x_j) I_{\bar{\mathcal E}_{j}}] - \E [F(\x_j) I_{\bar{\mathcal E}_{j-1}}] 
= - \E [F(\x_j)] \text{Pr}({\bar{\mathcal E}_{j-1}}-{\bar{\mathcal E}_{j}}).
\end{align*}
Please note that here $\text{Pr}({\bar{\mathcal E}_{j-1}}-{\bar{\mathcal E}_{j}})$ means the probability of ${\bar{\mathcal E}_{j-1}}$ happens (i.e., $I_{\bar{\mathcal E}_{j-1}}=1$ ) and ${\bar{\mathcal E}_{j}}$ doesn't happen (i.e., $I_{\bar{\mathcal E}_{j}}=0$). 
}
\end{itemize}
(b) According to (\ref{inqe:case:1:2}) we can see that under the event $\bar{\mathcal E}_j$
\begin{align*}
\E[F(\x_{j})] \leq F(\x_0) - j \theta.
\end{align*}
As a result,  $F(\x_*)\leq \E[F(\x_{j})] \leq F(\x_0) $, and
\begin{align}\label{bound:exp:obj}
 |\E[F(\x_{j})]| \leq B := \max\{ |F(\x_*)|, |F(\x_0)|\}. 
\end{align}
Thus, 
\begin{align*}
&\E [F(\x_{j+1}) I_{\bar{\mathcal E}_{j}}] - \E [F(\x_j) I_{\bar{\mathcal E}_{j-1}}] \leq -\theta \text{Pr}(\bar{\mathcal E}_{j}) + B\text{Pr}({\bar{\mathcal E}_{j-1}}-{\bar{\mathcal E}_{j}}).
\end{align*}
By summing up over $k = 0,\dots, j$ we have
\begin{align*}
& \E [F(\x_{j+1}) I_{\bar{\mathcal E}_j}] - \E[F(\x_0)] \leq - \theta \sum_{k=0}^{j}\text{Pr}(\bar{\mathcal E}_{k}) + B \sum_{k=0}^{j}\text{Pr}({\bar{\mathcal E}_{k-1}}-{\bar{\mathcal E}_{k}})\\
\leq &- \theta \sum_{k=0}^{j}\text{Pr}(\bar{\mathcal E}_{k}) + B \text{Pr}({-\bar{\mathcal E}_{j}})
\leq - \theta (j+1)\text{Pr}(\bar{\mathcal E}_{j}) + B \text{Pr}(-{\bar{\mathcal E}_{j}}),
\end{align*}
which implies
\begin{align*}
\theta (j+1)\text{Pr}(\bar{\mathcal E}_{j}) \leq & B \text{Pr}(-{\bar{\mathcal E}_{j}}) -E [F(\x_{j+1}) I_{\bar{\mathcal E}_j}] + E[F(\x_0)]\\
\leq& B + [F(\x_0) - F(\x_*)] \leq B +\Delta.
\end{align*}
As $(j+1)$ grows to as large as $\frac{2(B+\Delta)}{\theta}$, we will have $Pr(\bar{\mathcal E}_{j}) \leq \frac{1}{2}$.
Therefore, after $\widetilde{O}(\frac{2(B+\Delta)}{\theta})$ steps, $\y_j\in\mathcal C_3$ must occur at least once with probability at least $\frac{1}{2}$. If we repeat this $\log(1/\zeta)$ times, then after $\widetilde{O}(\log(1/\zeta)\frac{1}{\theta})$ steps, with probability at least $1-\zeta/2$, $\y_j\in\mathcal C_3$ must occur at least once.
Therefore, with high probability, NEON-$\mathcal A$ terminates with a total IFO complexity of $\widetilde O(\max\left(\frac{1}{\varepsilon(\epsilon, \alpha)}, \frac{1}{\gamma^3}\right)(T_n + T_a + T_c))$. 
\end{proof}

\section{Proof of Corollary~\ref{thm:SM}}
\begin{proof}
Let us first consider SM($\x_0, \eta, \beta, s, t)$. According to the analysis of Theorem 3 in~\cite{yangnonconvexmo}, when $\eta\leq (1-\beta)/(2L_1)$, we have
\begin{align*}
&\frac{1}{t+1}\sum_{\tau=0}^t\E[\|\nabla F(\x_{\tau})\|^2]\leq \frac{2\E[F(\x_0) - F(\x^+_{t+1})]}{\eta (t+1)/(1-\beta)} + D \eta,
\end{align*}
where $D = \left[\frac{L_1\beta^2((1-\beta)s-1)^2\sigma^2}{(1-\beta)^3} + \frac{L_1 \sigma^2}{1-\beta}\right]$, and $\sigma^2$ is the upper bound of $\E[\|\nabla f(\x; \xi)\|^2]$. 
On the other hand, by Lemma 3 of~\cite{yangnonconvexmo}, we know for any $\tau \geq 0$,
\begin{align*}
 \E[\|\nabla F(\x^+_{\tau}) - \nabla F(\x_{\tau})\|^2] \leq \frac{L_1\eta^2}{1-\beta} \left[D -  \frac{L_1 \sigma^2}{1-\beta}\right] 
\leq  \frac{L_1\eta^2}{1-\beta} D \leq \frac{D\eta}{2},
\end{align*}
where the last inequality is due to $\eta\leq (1-\beta)/(2L_1)$. Since $ \E[\|\nabla F(\x^+_{\tau})\|^2] \leq  2 \E[\|\nabla F(\x^+_{\tau}) - \nabla F(\x_{\tau})\|^2  + \|\nabla F(\x_{\tau})\|^2]$, then
\begin{align}\label{eqn:sm}
\frac{1}{t+1}\sum_{\tau=0}^t\E[\|\nabla F(\x^+_{\tau})\|^2]
\leq \frac{4\E[F(\x_0) - F(\x^+_{t+1})]}{\eta (t+1)/(1-\beta)} + 3D \eta,
\end{align}
Next, consider $(\y_j, \z_j) = \text{SM}(\x_j, \eta, \beta, s, t)$, we have $\y_j = \x^+_{\tau'}$, $\z_j = \x^+_{t+1}$, and 
\begin{align*}
&\E[\|\nabla F(\y_j)\|^2]\leq \frac{4\E[F(\x_j) - F(\z_j)]}{\eta (t+1)/(1-\beta)} + 3D \eta.
\end{align*}
When $\|\nabla F(\y_j)\|\geq \epsilon$ we have 
\begin{align*}
&\E[F(\z_j) - F(\x_j)]  \leq  - \frac{\eta (t+1)}{4(1-\beta)} (\epsilon^2 - 3 D \eta) \leq - \frac{C'\epsilon^2}{48D(1-\beta)},
\end{align*}
where the last inequality holds by choosing $\eta = \frac{\epsilon^2}{6D}$ and $t+1 \geq \frac{C'}{\epsilon^2}$, where $C'$ is a constant. Therefore, we have $\varepsilon(\epsilon, \alpha) = \frac{C'\epsilon^2}{48D(1-\beta)}$.
Similarly, from~(\ref{eqn:sm}) we have
\begin{align*}
&\E[F(\y_j) - F(\x_j)]  \leq   \frac{3\eta^2 (t+1) D}{4(1-\beta)} \leq  \frac{C\epsilon^2}{2},
\end{align*}
where the last inequality holds with $t+1\leq \frac{24DC(1-\beta)}{\epsilon^2}$.  Therefore, when  $\|\nabla F(\y_j)\|\leq \epsilon$, we have $\E[F(\y_j) -F(\x_j)]\leq \epsilon^2$. By assuming that $\gamma\geq \epsilon^{2/3}$,  the two inequalities in~(\ref{eqn:A}) hold. 
\end{proof}

\section{Proof of Corollary~\ref{thm:SGM}}
\begin{proof}
Based on Theorem~\ref{thm:gen}, we only need to show (\ref{eqn:A}) holds for NEON-MSGD. By the updates in (\ref{eqn:SGD}), we know $\y_j = \x_j$. It implies that the second inequality of (\ref{eqn:A}) holds, i.e. $\E[F(\y_j) - F(\x_j)]\leq \frac{C\gamma^3}{2}$.
Next, we will show the first inequalty holds when $\|\nabla F(\y_j)\|\geq \epsilon$, i.e. $\|\nabla F(\x_j)\|\geq \epsilon$.
Recall that the update of MSGD is $\z_j = \x_j - \frac{1}{L_1}\nabla F_{\mathbf S_1}(\x_j)$, then by the smoothness of $F(\cdot)$ we have
\begin{align*}
 \E[ F(\z_{j}) - F(\x_j)] 
\leq &\E\left[ (\z_{j} -\x_j)^\top\nabla F(\x_j) + \frac{L_1}{2}\|\z_{j} - \x_j\|^2\right] \\
= &  -\frac{1}{L_1} \|\nabla F(\x_j) \|^2 + \E\left[\frac{1}{2L_1}\|\nabla F_{\mathbf S_1}(\x_j)\|^2\right] \\
= &  -\frac{1}{2L_1} \|\nabla F(\x_j) \|^2 + \E\left[\frac{1}{2L_1}\|\nabla F(\x_j) - \nabla F_{\mathbf S_1}(\x_j)\|^2\right]\\
\leq & -\frac{1}{2L_1} \|\nabla F(\x_j) \|^2 + \frac{1}{2L_1} \frac{V}{|\S_1|},
\end{align*}
where the last inequality uses Assumption~\ref{ass:1} (iv). 
By choosing $|\S_1| \ge \frac{2V}{\epsilon^2}$ and using the condition of $\|\nabla F(\x_j)\|\geq \epsilon$, we have
\begin{align*}
 \E[ F(\z_{j}) - F(\x_j)]\leq  -\frac{\epsilon^2}{4L_1} = - \varepsilon(\epsilon, \alpha).
\end{align*}
Additionally, we have $\y_j=\x_j$ indicating that the second inequality in (\ref{eqn:A}) holds. 
\end{proof}

\section{Proof of Corollary~\ref{thm:SCSG}}
\begin{proof}
Based on Theorem~\ref{thm:gen}, we only need to show (\ref{eqn:A}) holds for NEON-SCSG. By the updates in (\ref{eqn:SCSG}), we know $\z_j = \y_j$. 
Each epoch of SCGS guarantees~\citep{DBLP:journals/corr/LeiJCJ17} that 
\begin{align*}
\E[\|\nabla F(\y_j)\|^2]\leq \frac{5L_1 b^{1/3}}{c' |\S_1|^{1/3}} \E[F(\x_{j}) - F(\y_j)] + \frac{6 V}{|\S_1|}.
\end{align*}
As a result,  when $\|\nabla F(\y_j)\|\geq \epsilon$, we have
\begin{align*}
\E[F(\y_j) - F(\x_j)]\leq  -\frac{c' |\S_1|^{1/3}}{5L_1b^{1/3}}(\epsilon^2-\frac{6V}{|\S_1|})
\end{align*}

\begin{algorithm}[t]
\caption{SCSG-epoch: $(\x, \mathbf S_1, b)$}\label{alg:SCSG-epoch}
\begin{algorithmic}[1]
\STATE \textbf{Input}:  $\x$, an independent set of samples $\mathbf S_1$ and $b\leq |\mathbf S_1|$
\STATE Set $m_1 = |\S_1|$, $\eta = c'(m_1/b)^{-2/3}$ with $c'\leq 1/6$
\STATE Compute $\nabla F_{\S}(\x_{j-1})$
\STATE Let $\x_0 = \x$ and generate $N\sim \text{Geom}(m_1/(m_1+b))$
\FOR{$k=1,2,\ldots, N$}
\STATE Sample samples $\S_k$ of size $b$
\STATE Compute $\v_{k} = \nabla F_{\S_k}(\x_{k-1}) - \nabla F_{\S_k}(\x_0) + \nabla F_{\S}(\x_0)$
\STATE $\x_k = \x_{k-1} - \eta \v_{k}$ 
\ENDFOR
\STATE RETURN $\x_{N}$
\end{algorithmic}
\end{algorithm}

By setting $|\S_1| \geq  12V/\epsilon^2$, we have
\begin{align*}
\E[F(\y_j) - F(\x_j)]\leq -\frac{c' |\S_1|^{1/3}}{10L_1b^{1/3}} \epsilon^2 \leq - \frac{c' (12V)^{1/3}}{10L_1}\frac{\epsilon^{4/3}}{b^{1/3}}
\end{align*}
On the other hand, when  $\|\nabla F(\y_j)\|\leq \epsilon$, we have
\begin{align*}
0\leq \E[\|\nabla F(\y_j)\|^2]\leq \frac{5L_1 b^{1/3}}{c' |\S_1|^{1/3}} \E[F(\x_{j}) - F(\y_j)] + \frac{6 V}{|\S_1|},
\end{align*}
i.e., 
\begin{align*}
 \E[F(\y_{j}) - F(\x_j)]\leq \frac{6c' V}{5L_1|\S_1|^{2/3}b^{1/3}}.
\end{align*}
By choosing $|\S_1| \geq  \left(\frac{12Vc'}{5CL_1}\right)^{3/2} \frac{1}{b^{1/2}\gamma^{9/2}}$,
\begin{align*}
 \E[F(\y_{j}) - F(\x_j)]\leq \frac{C\gamma}{2}.
\end{align*}

In summary, the sample sizes must satisfy $|\S_1|\geq \widetilde O(\max(1/\epsilon^{2}, 1/(\gamma^{9/2}b^{1/2}))$. The total iterations (i.e., the number of calls of SCSG-Epoch and NEON) is $\widetilde O\left(\max(\frac{b^{1/3}}{\epsilon^{4/3}}, \frac{1}{\gamma^3})\right)$. The expected IFO complexity of each SCSG-Epoch is $2|\S_1|\geq \widetilde O(\max(1/\epsilon^{2}, 1/(\gamma^{9/2}b^{1/2}))$ regardless the value of $b$ due to the geometric distribution of $N$ in SCSG-Epoch. 
\end{proof}

\section{Proof of Theorem~\ref{thm:main:GD}}
We first prove the following lemma. 
\begin{lemma}\label{lemma:key:noisypower2n}
Under the same conditions and settings as in Theorem~\ref{thm:main:GD},   If NEON is called  at  a point $\x$ such that $\lambda_{\min}(\nabla^2 f(\x))\leq -\gamma$, then  with high probability $1-\delta$ there exists $\tau\leq t$ such that
\begin{align*}
&f(\x  + \u_\tau) - f(\x) - \nabla f(\x)^{\top}\u_\tau \leq - 2\mathcal F, \\
 &\|\u_\tau\|\leq U = 4\hat c(\sqrt{\eta L_1}\mathcal F/L_2)^{1/3}
\end{align*}
\end{lemma}

For simplicity, we recall and define some notations. Let $H = \nabla^2 f(\x)$ be the Hessian matrix and $\e_1$ be the unit minimum eigenvector of $H$. We also use the following notations in the proofs: 
\begin{align*}
&\mathcal F := \eta L_1 \frac{\gamma^3}{L_2^2}\cdot \log^{-3}(d\kappa /\delta), \\
&\mathcal  P  :=\sqrt{\eta L_1}\frac{\gamma}{L_2} \cdot \log^{-1}(d\kappa /\delta), \\
&\mathcal  J  := \frac{\log (d\kappa /\delta)}{\eta \gamma},
\end{align*}
where $\kappa = L_1/ \gamma \ge 1$ is the condition number. 
From above notations, we know $r=\frac{\mathcal P}{\kappa}\log^{-1}(d\kappa /\delta) = \mathcal P\gamma L_1^{-1}\log^{-1}(d\kappa /\delta)\leq\mathcal P$. 
Let us define 
\begin{align}\label{approx:qua:n}
\hat f_{\x}(\u) = f(\x + \u ) - f(\x) -\u^\top\nabla f(\x)
\end{align}
Also, we need the following two lemmas for our proof. It is notable that the following two lemmas are similar to Lemma 16 and Lemma 17 in\cite{jin2017escape}. The analysis of Lemma~\ref{lemma:key2n} is similar to that of Lemma 17 in~\cite{jin2017escape}. However, we provide a much simpler analysis for Lemma~\ref{lemma:key1n} than that of Lemma 16 in~\cite{jin2017escape}.

\begin{lemma}\label{lemma:key1n}
For any constant $\hat c\geq 8$, there exists $c_{\max}$: for $\x$ satisfies the condition that  $\lambda_{\min}(\nabla^2 f(\x))\leq - \gamma$
and any initial point $\u_0$ with $\|\u_0\|\leq 2r$, 
define
\begin{align*}
T = \min\{\inf_\tau \{\tau| \hat f_\x(\u_\tau) - \hat f_\x(\u_0)  \leq -3\mathcal F\}, \hat c\mathcal J\}.
\end{align*}
Then for any $\eta\leq c_{\max}/L_1$, we have for all $\tau<T$ that $\|\u_\tau\|\leq 2\hat c\mathcal P$. 
\end{lemma}
\begin{lemma}\label{lemma:key2n}
There exist constant $c_{\max}, \hat c$ such that: for $\x$ satisfies the condition that $\lambda_{\min}(\nabla^2 f(\x))\leq - \gamma$
define a sequence $\w_t$ similar to $\u_t$ except $\w_0  =\u_0+ \mu r \e_1$, where $\mu\in[\delta/2\sqrt{d}, 1]$ and $\e_1$ is a unit eigen-vector corresponding to the minimum eigen-value of $\nabla^2 f(\x)$, and let  $\v_t = \w_t - \u_t$
\begin{align*}
T = \min\{\inf_\tau \{\tau| \hat f_\x(\w_\tau) - \hat f_\x(\w_0)\leq -3\mathcal F\}, \hat c\mathcal J\}.
\end{align*}
Then for any $\eta\leq c_{\max}/L_1$, if $\|\u_\tau\|\leq 2\hat c \mathcal P$ for all $\tau<T$, then $T<\hat c\mathcal J$.
\end{lemma}

\begin{proof}{\bf (of Lemma~\ref{lemma:key:noisypower2n})}
We define $T_*=\hat c \mathcal J$ and $T' = \inf_{\tau}\{\tau| \hat f_{\x}(\u_\tau) - \hat f_{\x}(\u_0) \leq -3\mathcal F\}$. Let's consider following two scenarios:\\ \\
{\bf (1) $T'\leq T_*$:} Since $\u_{T'} = \u_{T'-1} - \eta(\nabla f(\x+\u_{T'-1}) - \nabla f(\x))$, we can see that $\|\u_{T'}\|\leq \|\u_{T'-1}\| + \eta L_1\|\u_{T'-1}\|\leq 4 \hat c \mathcal P\triangleq U $ by employing Lemma~\ref{lemma:key1n}. Then we have
\begin{align*}
f(\x + \u_{T'}) - f(\x)- \u_{T'}^{\top} \nabla f(\x)
\leq & f(\x + \u_{0}) - f(\x)- \u_{0}^{\top} \nabla f(\x)  -3\mathcal F\\
\leq & \frac{L_1}{2}\|\u_0\|^2 -3\mathcal F \leq \mathcal F - 3\mathcal F = -2 \mathcal F.
\end{align*}
Therefore we have
\begin{align*}
& \min_{1\leq \tau'\leq T_*, \|\u_{\tau'}\|\leq U}f(\x + \u_{\tau'}) - f(\x) -\u_{\tau'}^{\top} \nabla f(\x)
 \leq  f(\x + \u_{T'}) - f(\x)-\u_{T'}^{\top} \nabla f(\x) \leq  -2\mathcal F.
\end{align*}
{\bf (2) $T' > T_*$:} By Lemma~\ref{lemma:key1n}, we have $\|\u_{\tau}\| \leq 2\hat c\mathcal P$ for all $\tau \leq T_*$. Let define $T'' = \inf_\tau \{\tau| \hat f_\x(\w_\tau) - \hat f_\x(\w_0) \leq -3\mathcal F\}$. By Lemma~\ref{lemma:key2n}, we konw $T''\leq T_*$.  From the proof of Lemma~\ref{lemma:key2n}, we also know that $\|\w_{T''-1}\|\leq 2\hat c \mathcal P$. Hence $\|\w_{T''}\|\leq \|\w_{T'-1}\| + \eta L_1\|\w_{T''-1}\|\leq 4 \hat c \mathcal P$.  Similar to case (1), we have 
\begin{align*}
&\min_{1\leq \tau'\leq T_*, \|\w_{\tau'}\|\leq U} f(\x + \w_{\tau'}) - f(\x) -\w_{\tau'}^{\top} \nabla f(\x)\leq -2\mathcal F.
\end{align*}
Therefore,
\begin{align*}
\min&\left\{ \min_{1\leq\tau'\leq T_*, \|\u_{\tau'}\|\leq U}f(\x + \u_{\tau'}) - f(\x) -\u_{\tau'}^{\top} \nabla f(\x),\right.\\
&\left.\min_{1\leq\tau'\leq T_*, \|\w_{\tau'}\|\leq U}f(\x + \w_{\tau'}) - f(\x) -\w_{\tau'}^{\top} \nabla f(\x) \right\} \leq -2\mathcal F.
\end{align*}
We know $\u_0$ follows an uniform distribution over $\mathbb B_0(r)$ with radius $r = \mathcal P/(\kappa \cdot \log(d\kappa/\delta))$. Let denote by $\mathcal X_s \subset \mathbb B_0(r)$ the set of bad initial points such that $\min_{1\leq \tau'\leq T_*, \|\u_{\tau}'\|\leq U}f(\x + \u_{\tau'}) - f(\x) -\u_{\tau'}^{\top} \nabla f(\x) > -2\mathcal F$ when $\u_0 \in \mathcal X_s$; otherwise $\min_{1\leq \tau'\leq T_*, \|\u_{\tau}'\|\leq U}f(\x + \u_{\tau'}) - f(\x) -\u_{\tau'}^{\top} \nabla f(\x) \le -2\mathcal F$ when $\u_0 \in  \mathbb B_0(r) - \mathcal X_s$.

By above analysis, for any $\u_0\in \mathcal X_s$, we have $(\u_0 \pm \mu r \e_1) \not \in \mathcal X_s $ where $\mu \in [\frac{\delta}{2\sqrt{d}}, 1]$. Let denote by $I_{\mathcal X_s}(\cdot)$ the indicator function of being inside set $\mathcal X_s$. We set $u^{(1)}$ as the component along $\e_1$ direction and $\u^{(-1)}$ as the remaining $d-1$ dimensional vector, then the vector $\u = (u^{(1)}, \u^{(-1)})$. It is easy to have an upper bound of $\mathcal X_s$'s volumn:
\begin{align*}
& \text{Vol}(\mathcal X_s) =  \int_{\mathbb B_0^{(d)}(r)}  \mathrm{d}\u \cdot I_{\mathcal X_s}(\u)
= \int_{\mathbb B_0^{(d-1)}(r)}  \mathrm{d}\u^{(-1)} \int_{-\sqrt{r^2 - \|\u^{(-1)}\|^2}}^{\sqrt{r^2 - \|\u^{(-1)}\|^2}} \mathrm{d} u^{(1)} \cdot I_{\mathcal X_s}(\u)\\
\leq & \int_{\mathbb B_0^{(d-1)}(r)}  \mathrm{d}\u^{(-1)} \cdot 2\mu r 
\leq \int_{\mathbb B_0^{(d-1)}(r)}  \mathrm{d}\u^{(-1)}  \cdot 2 \frac{\delta}{2\sqrt{d}}r
 =   \text{Vol}(\mathbb B_0^{(d-1)}(r)) \frac{\delta r}{\sqrt{d}} 
\end{align*}
Then, 
\begin{align*}
&\frac{\text{Vol}(\mathcal X_s)}{\text{Vol}(\mathbb B_0^{(d)}(r)) }
\le \frac{ \text{Vol}(\mathbb B_0^{(d-1)}(r)) \frac{\delta r}{\sqrt{d}}}{\text{Vol}(\mathbb B_0^{(d)}(r)) }
= \frac{\delta}{\sqrt{\pi d}}\frac{\Gamma(\frac{d}{2}+1)}{\Gamma(\frac{d}{2}+\frac{1}{2})}
\le \frac{\delta}{\sqrt{\pi d}} \cdot \sqrt{\frac{d}{2}+\frac{1}{2}} \le \delta
\end{align*}
where the second inequality is due to $\frac{\Gamma(x+1)}{\Gamma(x+1/2)}<\sqrt{x+\frac{1}{2}}$ for all $x\ge 0$.
Thus, we have $\u_0 \not \in \mathcal X_s$ with at least probability $1-\delta$. Therefore,
\begin{align}\label{ineq:keylemma:1:n}
\min_{1\leq \tau'\leq T_*, \|\u_{\tau'}\|\leq U}f(\x + \u_{\tau'}) - f(\x) -\u_{\tau'}^{\top} \nabla f(\x) \leq -2\mathcal F.
\end{align}

To finish the proof of Theorem~\ref{thm:main:GD}, by the Lipschitz continuity of Hessian, we have 
\begin{align*}
&\left| f(\x+\u_{\tau}) - f(\x) - \nabla f(\x)^\top\u_\tau - \frac{1}{2}\u_{\tau}^\top\nabla^2 f(\x)\u_{\tau}\right| 
\leq \frac{L_2}{6}\|\u_{\tau}\|^3.
\end{align*}
Then by inequality (\ref{ineq:keylemma:1:n}), we get
\begin{align*}
   &\frac{1}{2}\u_\tau^\top\nabla^2 f(\x)\u_\tau
   \leq  f(\x+\u_\tau) - f(\x) - \nabla f(\x)^\top\u_\tau + \frac{L_2}{6}\|\u_\tau\|^3 
   \leq   -2\mathcal F + \mathcal F  \leq  -\mathcal F.
\end{align*}
That is, 
\begin{align}
   \frac{\u_\tau^\top\nabla^2 f(\x)\u_\tau}{\|\u_\tau\|^2}  \leq \frac{-2\mathcal F}{(4\hat c \mathcal P)^2}\leq-\frac{\gamma}{8\hat c^2 \log(dL_1 /(\gamma\delta))} .
\end{align}
If NEON returns $0$ following Bayes theorem, it is not difficult to show that $\lambda_{\min}(\nabla^2 f(\x))\geq - \gamma$ with high probability $1-O(\delta)$ for a sufficiently small $\delta$.
\end{proof}
\subsection{Proof of Lemma~\ref{lemma:key1n}}
\begin{proof}
By using the smoothness of $\hat f_\x(\u)$, we have
\begin{align*}
 \hat f(\u_{\tau+1}) 
\leq & \hat  f(\u_\tau) + \nabla \hat f(\u_\tau)^\top (\u_{\tau+1} - \u_\tau) + \frac{L_1}{2}\|\u_{\tau+1} - \u_\tau\|^2\\
= & \hat  f(\u_\tau) - \frac{1}{\eta} \|\u_{\tau+1} - \u_\tau\|^2 + \frac{L_1}{2}\|\u_{\tau+1} - \u_\tau\|^2 \\
\leq & \hat  f(\u_\tau) -  \frac{1}{2\eta}\|\u_{\tau+1} - \u_\tau\|^2,
\end{align*}
where the first equality uses the update of $\u_{\tau+1}$ in Algorithm~\ref{alg:ncn};  the last inequality uses the fact of $\eta L_1 < 1$.
By summing up $\tau$ from $0$ to $t-1$ where $t<T$, we have
\begin{align*}
&\hat f(\u_{t}) \leq  \hat f(\u_{0}) - \frac{1}{2\eta}\sum_{\tau=0}^{t-1}\|\u_{\tau} - \u_{\tau-1}\|^2,
\end{align*}
which inplies
\begin{align*}
\frac{1}{2\eta}\sum_{\tau=0}^{t-1}\|\u_{\tau} - \u_{\tau-1}\|^2  \leq & \hat f(\u_{0}) - \hat f(\u_{t}).
\end{align*}
Hence, if $t\leq T$ then $\hat f(\u_t) - \hat f(\u_0)\geq -3\mathcal F$, i.e., $\hat f(\u_0) - \hat f(\u_t) \leq 3\mathcal F$. Then we have
\begin{align*}
&\sum_{\tau=0}^{t-1}\|\u_{\tau} - \u_{\tau-1}\| \leq \sqrt{t\sum_{\tau=0}^{t-1}\|\u_{\tau} - \u_{\tau-1}\|^2}
\leq  \sqrt{\frac{3\hat c \log(\frac{d\kappa}{\delta})}{\eta\gamma}  \frac{\eta L_1\gamma^3}{L_2^2 \log^3(\frac{d\kappa}{\delta})} 2\eta}\\
= &\sqrt{6\hat c}\sqrt{\eta L_1} \frac{\gamma}{L_2}\log^{-1}(d\kappa/\delta) = \sqrt{6\hat c}  \mathcal P.
\end{align*}
Then for all $\tau\leq t-1$, $\|\u_\tau\|\leq \sum_{k=1}^{\tau}\|\u_{k} -\u_{k-1}\| + \|\u_0\|\leq \sqrt{6\hat c} \mathcal P + \mathcal P \leq \hat c \mathcal P$, where the last inequality is due to $\hat c \geq 18$. Additionally, by the update of $\u_{\tau+1}$ in Algorithm~\ref{alg:ncn}, we have
\begin{align*}
\|\u_{\tau+1}\|\leq& (1+\eta L_1)\|\u_{\tau}\| \leq 2\hat c \mathcal P.
\end{align*}
\end{proof}

\subsection{Proof of Lemma~\ref{lemma:key2n}}
The proof is almost the same as the proof of Lemma 17 in~\citep{jin2017escape}. For completeness, we include it in this subsection.
By the way $\u_0$ is constructed,  we have $\|\u_0\|\leq r \leq \hat c\mathcal P$.
Let us define $\v_\tau = \w_\tau - \u_\tau$, then  { $\v_0 = \mu r \e_1$ (i.e., $r=\frac{\mathcal P}{\kappa}\log^{-1}(d\kappa/\delta)$)}, $\mu \in [\delta/(2\sqrt{d}), 1]$. By the update equation of $\w_{\tau+1}$ (similar to $\u_{\tau+1}$), we have
\begin{align*}
&\u_{\tau+1} + \v_{\tau+1} =\w_{\tau+1} 
=  \w_\tau - \eta (\nabla f(\x+ \w_\tau) -\nabla  f(\x)) \\
= &\u_\tau + \v_\tau - \eta (\nabla  f(\x+\u_\tau + \v_\tau ) -\nabla  f(\x)) \\
= &\u_\tau + \v_\tau - \eta (\nabla  f(\x+\u_\tau + \v_\tau ) -\nabla  f(\x+\u_\tau) + \nabla  f(\x+\u_\tau ) -\nabla  f(\x)) \\
= &\u_\tau  - \eta ( \nabla  f(\x+\u_\tau ) -\nabla  f(\x)) + \v_\tau -\eta H \v_{\tau} - \eta \left[\int_{0}^1 \nabla^2 f(\x+\u_t + \theta\v_t) \mathrm{d}\theta -H\right] \v_\tau  \\
= &\u_\tau  - \eta ( \nabla  f(\x+\u_\tau ) -\nabla f(\x)) + [I -\eta H -\eta \Delta'_\tau]\v_{\tau}
\end{align*}
where $\Delta'_\tau := \int_{0}^1 \nabla^2  f(\x+\u_t + \theta\v_t) \mathrm{d}\theta -H$. We know $\|\Delta'_\tau\| \leq L_2( \|\u_\tau\| + \|\v_\tau\|/2)$ according to the Lipschitz continuity of Hessian. 
Then we know the update of $\v_\tau$ is
\begin{align}\label{eq:v_dynamic:n}
\v_{\tau+1} = (I - \eta H - \eta \Delta'_\tau) \v_\tau
\end{align}

We first show that $\v_{\tau}, \tau<T$ is upper bounded. Due to $\|\w_0 \| \leq \|\u_0\| +\|\v_0\| \leq r + r =2r$,  following the result in Lemma \ref{lemma:key1n}, we have $\|\w_\tau\| \leq 2\hat c \mathcal P$ for all $\tau< T$. According to the condition in Lemma \ref{lemma:key2n}, we have $\|{\u_\tau}\| \leq 2\hat c \mathcal P $ for all $\tau<T$.
Thus
\begin{align} \label{eq:bound_v:n}
\|{\v_\tau}\| \leq \|{\u_\tau}\| +\|{\w_\tau}\| \leq 4\hat c \mathcal P \text{~for all~} \tau<T
\end{align}
Then we have for $\tau<T$
\begin{align*}
\|\Delta'_\tau\| \leq L_2( \|\u_\tau\| + \|\v_\tau\|/2) \leq L_2  \cdot 2\hat c \mathcal P
\end{align*}

Next, we show that $\v_{\tau}, \tau<T$ is lower bounded. To proceed the proof,  denote by $\psi_\tau$ the norm of $\v_\tau$ projected onto $\e_1$ direction and denote by $\varphi_\tau$ the norm of $\v_\tau$ projected onto the remaining subspace. The update equation of $\v_{\tau +1}$ (\ref{eq:v_dynamic:n}) implies
\begin{align*}
\psi_{\tau+1} \ge& (1+\gamma \eta)\psi_\tau - \zeta\sqrt{\psi_\tau^2 + \varphi_\tau^2},\\
\varphi_{\tau+1} \le &(1+\gamma\eta)\varphi_\tau + \zeta\sqrt{\psi_\tau^2 + \varphi_\tau^2},
\end{align*}
where $\zeta = \eta L_2  \mathcal P (4 \hat c )$. We then prove the following inequality holds for all $\tau < T$ by induction,
\begin{equation}
\varphi_\tau \le 4 \zeta \tau \cdot \psi_\tau\label{eqn:phit:n}
\end{equation}
According to the definition of $\v_0$, we have $\varphi_0 = 0$, indicating that the inequality (\ref{eqn:phit:n}) holds for $\tau=0$. Assume the inequality (\ref{eqn:phit:n}) holds for all $\tau\leq t$. We need to show that the inequality (\ref{eqn:phit:n}) holds for $t+1 \leq T$. It is easy to have the following inequalities
\begin{align*}
4\zeta(t+1)\psi_{t+1} 
\geq & 4\zeta (t+1) \left( (1+\gamma \eta)\psi_t - \zeta \sqrt{\psi_t^2 + \varphi_t^2}\right),\\
\varphi_{t+1} \leq &4 \zeta  t(1+\gamma\eta) \psi_t + \zeta \sqrt{\psi_t^2 + \varphi_t^2}.
\end{align*} 
Then we only need to show
\begin{equation*}
 \left(1+4\zeta (t+1)\right)\sqrt{\psi_t^2 + \varphi_t^2}
 \le 4 (1+\gamma \eta)\psi_t.
\end{equation*}
If we choose $\sqrt{c_{\max}}\le \frac{1}{(4\hat c)^2}$, and $\eta \le c_{\max}/L_1$, then 
\begin{align*}
4\zeta (t+1) \leq& 4\zeta T \leq 4\eta L_2  \mathcal P (4 \hat c)\cdot \hat c \mathcal J =4\sqrt{\eta L_1}(4 \hat c )\hat c\le 1.
\end{align*}
This implies
\begin{align*}
4 (1+\gamma \eta)\psi_t \geq& 4\psi_t \geq 2\sqrt{2\psi_t^2} 
\geq \left(1+4\zeta (t+1)\right)\sqrt{\psi_t^2 + \varphi_t^2}.
\end{align*}
We then finish the induction. According to (\ref{eqn:phit:n}), we get $\varphi_\tau \le 4  \zeta \tau \cdot \psi_\tau \le \psi_\tau$ so that
\begin{equation}\label{eq:growth_v:n}
\psi_{\tau+1} \geq (1+\gamma \eta)\psi_\tau - \sqrt{2}\zeta\psi_\tau \geq (1+\frac{\gamma \eta}{2})\psi_\tau, 
\end{equation}
where the last inequality is due to $\zeta = \eta L_2 \mathcal P(4 \hat c ) \leq  \sqrt{c_{\max}}(4 \hat c ) \gamma \eta \cdot\log^{-1}(d\kappa/\delta)  < \frac{\gamma \eta}{4\hat c} < \frac{\gamma \eta}{2\sqrt{2}}$ with $\eta \le c_{\max}/ L_1$, $\log(d\kappa/\delta) \geq 1$, and $\hat c > 1$.

By the inequalities (\ref{eq:bound_v:n}) and (\ref{eq:growth_v:n}), for all $\tau<T$, we get
\begin{align*}
4( \mathcal P \cdot \hat c)
\ge &\|\v_\tau\| \ge \psi_\tau \ge (1+\frac{\gamma \eta}{2})^\tau \psi_0=(1+\frac{\gamma \eta}{2})^\tau \mu r \\
=&(1+\frac{\gamma \eta}{2})^\tau \mu\frac{\mathcal P}{\kappa}\log^{-1}(d\kappa/\delta) 
\ge (1+\frac{\gamma \eta}{2})^\tau \frac{\delta}{2\sqrt{d}}\frac{\mathcal P}{\kappa}\log^{-1}(d\kappa/\delta).
\end{align*}
Then
\begin{align*}
&T < \frac{\log (8 \frac{\kappa\sqrt{d}}{\delta}\cdot \hat c \log(d\kappa/\delta))}{\log (1+\frac{\gamma \eta}{2})}
\leq \frac{5}{2} \frac{\log (8 \frac{\kappa\sqrt{d}}{\delta}\cdot  \hat c \log(d\kappa/\delta))}{\gamma\eta}
\leq \frac{5}{2}(2 + \log (8 \hat c))\mathcal J,
\end{align*}
where the last inequality uses the facts that $\delta\in (0, \frac{d\kappa}{e}]$ and $\log (d\kappa/\delta) \ge 1$.
By choosing a large enough constant $\hat c$ satisfying $ \frac{5}{2}(2 + \log (8 \hat c)) \leq \hat c$ ($\hat c \geq 18$ works), we have
$T <  \hat c \mathcal J $ and complete the proof.

\section{Proof of Theorem~\ref{thm:main:AGD}}
We first consider the case when NCFind is executed, i.e., there exists $\tau>0$
\begin{align}\label{cond:largeNC}
\hat f(\y_{\tau}) < \hat f(\u_{\tau}) + \nabla \hat f(\u_\tau)^\top (\y_{\tau} - \u_\tau) - \frac{\gamma}{2}\|\y_{\tau} - \u_\tau\|^2,
\end{align}
and NCfind returns a non-zero vector $\v$. We will prove that $\v^{\top}\nabla^2 f(\x)\v/\|\v\|_2^2\leq -\widetilde\Omega(\gamma)$. There are two scenarios we need to consider. 

{\bf Scenario (1)} there exists $j \leq \tau$ such that 
\begin{align*}
j =\arg \min_{0\leq k \leq \tau} \{ \|\y_{k} - \u_k\|\geq \zeta \sqrt{6\eta\mathcal{F}} \}.
\end{align*}
Since $\u_k - \y_k = \zeta(\y_k - \y_{k-1})$, then it is equivalent to
\begin{align*}
j =\arg \min_{0\leq k \leq \tau} \{ \|\y_{k} - \y_{k-1}\|\geq \sqrt{6\eta\mathcal{F}} \}.
\end{align*}
According to the definition of $j$, we know
for all $0\leq k \leq j-1$, $\|\y_{k} - \y_{k-1}\|\leq \sqrt{6\eta\mathcal{F}}$ and $\|\y_{j} - \y_{j-1}\|\geq \sqrt{6\eta\mathcal{F}}$.
Thus, 
\begin{align*}
\|\y_{j-1}\| \leq &\sum_{k=1}^{j-1}\|\y_{k} - \y_{k-1}\| + \|\y_0\| 
\leq  j \sqrt{6\eta\mathcal{F}} + 2\mathcal P
\leq  \hat c \mathcal J \sqrt{6\eta\mathcal{F}} + 2\mathcal P  = (\sqrt{6}\hat c+2)\mathcal P .
\end{align*}
Similarly, for all $0\leq k \leq j-1$, we can show
\begin{align*}
\|\y_{k}\| \leq (\sqrt{6}\hat c+2)\mathcal P.
\end{align*}
By the update of NAG, we have $\|\u_{j-1} - \y_{j-1}\| = \zeta \|\y_{j-1} - \y_{j-2}\|$. Then
\begin{align*}
\|\u_{j-1}\| \leq&  \|\u_{j-1} - \y_{j-1}\| + \|\y_{j-1}\|= \zeta \|\y_{j-1} - \y_{j-2}\| + \|\y_{j-1}\|\\
 \leq & \|\y_{j-1} - \y_{j-2}\| + \|\y_{j-1}\| 
 \leq  \sqrt{6\eta\mathcal{F}} +(\sqrt{6}\hat c+2)\mathcal P  \leq  (\sqrt{6}\hat c+3)\mathcal P.
\end{align*}
Similarly, for all $0\leq k \leq j-1$, we can show
\begin{align*}
\|\u_{k}\| \leq  (\sqrt{6}\hat c+3)\mathcal P.
\end{align*}
Next we bound $\|\y_j\|$. Since
$\|\y_j-\y_{j-1}\| =  \|\u_{\tau-1} - \u_{\tau-2}  + \eta(\nabla \hat f(\u_{\tau-2}) - \nabla\hat f(\u_{\tau-1}))\|\leq (1 + \eta L_1)\|\u_{j-1} - \u_{j-2}\|$, then
\begin{align*}
\|\y_j\| \leq &(1 + \eta L_1)\|\u_{j-1} - \u_{j-2}\| + \|\y_{j-1}\| 
\leq 5/4\|\u_{j-1} - \u_{j-2}\| + \|\y_{j-1}\| \\
\leq&5/4\|\u_{j-1}\| +5/4\| \u_{j-2}\| + \|\y_{j-1}\| \leq (7\sqrt{6}\hat c/2+19/2)\mathcal P,
\end{align*}
where the first inequality uses $\eta L_1 \leq \frac{1}{4}$.
By using the relationship $\|\u_{j} - \y_{j}\| = \zeta \|\y_{j} - \y_{j-1}\|$ in NAG, we get
\begin{align*}
\|\u_{j}\| \leq&  \zeta \|\y_{j} - \y_{j-1}\|+ \|\y_{j}\|
\leq\|\y_{j} - \y_{j-1}\|+ \|\y_{j}\|
\leq \|\y_{j-1}\|+ 2\|\y_{j}\| \leq (8\sqrt{6}\hat c+21)\mathcal P.
\end{align*}
Therefore, by choosing $\hat c \geq 15$, we have shown
\begin{align}\label{lem1:AGD:upp:case1}
\|\y_{k}\| \leq 10\hat c\mathcal P~\text{and}~\|\u_{k}\| \leq 21\hat c\mathcal P,~0\leq k \leq j.
\end{align}
Next, we will show that $\y_j$ is a NC, i.e., $\y_j^{\top}\nabla^2 f(\x)\y_j\leq - \widetilde\Omega(\gamma^3)$.
We know that the inequality (\ref{cond:largeNC}) doesn't hold for all $0\leq k < \tau$, 
then by the analysis of Lemma~\ref{lemma:key1:AGD} (in particular inequality~(\ref{eqn:lm8key}) due to that~(\ref{cond:largeNC}) does not hold any  $k=0,\ldots, \tau-1$ ), we have
\begin{align*}
& \hat f(\y_{k+1}) + \frac{1}{2\eta} \| \y_{k+1} - \y_{k}\|^2  
\leq \hat f(\y_k)  + \frac{1}{2\eta} \| \y_k - \y_{k-1}\|^2.
\end{align*}
By summing up over $k$ from $0$ to $j-1$, we get
\begin{align*}
& \hat f(\y_j) + \frac{1}{2\eta} \| \y_j - \y_{j-1}\|^2 \leq \hat f(\y_0).
\end{align*}
Combining with $\|\y_{j} - \y_{j-1}\|\geq \sqrt{6\eta\mathcal{F}}$, we have
\begin{align*}
f(\x + \y_j) - f(\x)- \y_j^{\top} \nabla f(\x)
\leq& f(\x + \y_{0}) - f(\x)- \y_{0}^{\top}\nabla f(\x) - 3\mathcal F\\
\leq& -3\mathcal F + \frac{L_1}{2}\|\y_0\|^2\leq -3\mathcal F + \mathcal F\leq -2\mathcal F,\end{align*}
By the Lipschitz continuity of Hessian, we have 
\begin{align*}
&\left| f(\x+\y_j) - f(\x) - \nabla f(\x)^\top\y_j - \frac{1}{2}\y_j^\top\nabla^2 f(\x)\y_j\right|
 \leq \frac{L_2}{6}\|\y_j\|^3.
\end{align*}
Then
\begin{align*}
  & \frac{1}{2}\y_j^\top\nabla^2 f(\x)\y_j 
  \leq f(\x+\y_j) - f(\x) - \nabla f(\x)^\top\y_j + \frac{L_2}{6}\|\y_j\|^3 
   \leq  -2\mathcal F + \mathcal F  \leq  -\mathcal F.
\end{align*}
Therefore, combining with (\ref{lem1:AGD:upp:case1}) we get
\begin{align}\label{neon2_upp_1}
\frac{\y_j^\top\nabla^2 f(\x)\y_j}{\|\y_j\|^2}\leq -\frac{2\mathcal F}{(10\hat c\mathcal P)^2}
   = - \frac{\gamma}{50\hat c^2 \log(dL_1 /(\gamma\delta))}\leq - \frac{\gamma}{72\hat c^2 \log(dL_1 /(\gamma\delta))}.
\end{align}

{\bf Scenario (2)} 
For all $0\leq k \leq \tau $, $\|\y_{k} - \y_{k-1}\|\leq \sqrt{6\eta\mathcal{F}} $. 
Thus, 
\begin{align*}
\|\y_{\tau}\| \leq &\sum_{k=1}^{\tau}\|\y_{k} - \y_{k-1}\| + \|\y_0\| \leq  j \sqrt{6\eta\mathcal{F}} + 2\mathcal P
\leq  \hat c \mathcal J \sqrt{6\eta\mathcal{F}} + 2\mathcal P  = (\sqrt{6}\hat c+2)\mathcal P \leq  3\hat c\mathcal P.
\end{align*}
Similarly, we can show
\begin{align}\label{lem1:AGD:upp:case21}
\|\y_{k}\| \leq 3\hat c\mathcal P,~0\leq k \leq \tau,
\end{align}
where $\hat c \geq 15$.
By the update of NAG, we have $\|\u_{j-1} - \y_{j-1}\| = \zeta \|\y_{j-1} - \y_{j-2}\|$. Then
\begin{align*}
\|\u_{\tau}\| \leq&  \|\u_{\tau} - \y_{\tau}\| + \|\y_{\tau}\|
 = \zeta \|\y_{\tau} - \y_{\tau-1}\| + \|\y_{\tau}\|\\
 \leq & \|\y_{\tau} - \y_{\tau-1}\| + \|\y_{\tau}\|
 \leq  \sqrt{6\eta\mathcal{F}} +(\sqrt{6}\hat c+2)\mathcal P  \leq  (\sqrt{6}\hat c+3)\mathcal P\leq  3\hat c\mathcal P.
\end{align*}

Similarly, we can show
\begin{align}\label{lem1:AGD:upp:case2}
\|\u_{k}\| \leq 3\hat c\mathcal P,~0\leq k \leq \tau.
\end{align}
In addition, we also have
\begin{align}\label{lem1:AGD:upp:case22}
\|\y_{k}-\u_k\| \leq 6\hat c\mathcal P,~0\leq k \leq \tau.
\end{align}
 By expanding $\hat f(\y_{\tau})$ in a Taylor series, we get
\begin{align*}
    \hat f(\y_\tau) =&\hat f(\u_{\tau}) + \nabla \hat f(\u_\tau)^\top (\y_{\tau}- \u_\tau)+ \frac{1}{2}(\y_{\tau}- \u_\tau)^\top\nabla^2\hat f(\pi_\tau) (\y_{\tau}- \u_\tau),
\end{align*}
where $\pi_\tau=\theta' \u_\tau + (1-\theta')\y_\tau$ and $0\leq\theta'\leq 1$. By inequality (\ref{cond:largeNC}), we then have
\begin{align*}
    &\frac{1}{2}(\y_{\tau}- \u_\tau)^\top\nabla^2\hat f(\pi_\tau) (\y_{\tau}- \u_\tau)
    < -\frac{\gamma}{2}\|\y_\tau-\u_\tau\|^2.
\end{align*}
It is clear that $\y_{\tau}-\u_{\tau}\neq0$.  Then we have
\begin{align*}
    (\y_{\tau}- \u_\tau)^\top\nabla^2 f(\x) (\y_{\tau}- \u_\tau)
    \leq&  -\gamma\|\y_\tau-\u_\tau\|^2+ (\y_{\tau}- \u_\tau)^\top[\nabla^2f(\x) - \nabla^2\hat f(\pi_\tau)] (\y_{\tau}- \u_\tau)\\
    \leq& -\gamma\|\y_\tau-\u_\tau\|^2 + L_2\|\theta' \u_\tau + (1-\theta')\y_\tau\|\|\y_{\tau}- \u_\tau\|^2\\
    \leq& [-\gamma+L_2(\|\u_\tau\| + \|\y_\tau\|)]\|\y_\tau-\u_\tau\|^2\\
    \leq& -\frac{ \gamma}{2}\|\y_\tau-\u_\tau\|^2.
\end{align*}
Therefore,
\begin{align}\label{neon2_upp_2}
\frac{(\y_{\tau}- \u_\tau)^\top\nabla^2 f(x) (\y_{\tau}- \u_\tau)}{\|\y_{\tau}- \u_\tau\|^2}\leq -\frac{\gamma}{2} \leq - \frac{\gamma}{72\hat c^2 \log(dL_1 /(\gamma\delta))}. 
\end{align}

Next, we will prove Theorem~\ref{thm:main:AGD} under the condition that NCFind is not excuted, i.e., (\ref{cond:largeNC}) does not happen. 
\begin{lemma}\label{lemma:key1:AGD}
Define $ \mathcal F=\eta L_1 \frac{\gamma^3}{L_2^2}\cdot \log^{-3}(d\kappa /\delta)$,
$\mathcal P=\sqrt{\eta L_1}\frac{\gamma}{L_2} \cdot \log^{-1}(d\kappa /\delta)$, 
$\mathcal J=\sqrt{\frac{\log (d\kappa /\delta)}{\eta \gamma}}$.
{Suppose the inequality (\ref{cond:largeNC}) doesn't hold.} There exist a universal constant $c_{\max}$, $\hat c$: for $\x$ satisfies the condition that $\lambda_{\min}(\nabla^2 f(\x))\leq - \gamma$,  any $\|\u_0\| \leq 2r$, where
$r=\frac{\mathcal P}{\kappa}\log^{-1}(d\kappa /\delta)$,
and  
\begin{align*}
T = \min\left\{ \inf_\tau\{\tau|\hat  f(\y_\tau)-\hat f(\y_0)\leq -3\mathcal F\}, \hat c \mathcal J \right\},
\end{align*}
If $\zeta$ satisfies $1 - \frac{\sqrt{1 + 3\eta L_1}}{1-(1-s)\eta L_1} \zeta \geq \sqrt{\eta\gamma}$,
then for any $\eta\leq c_{\max}/L_1$, we have for all $\tau<T$ that $\|\u_\tau\|\leq 6\hat c \mathcal P$. 
\end{lemma}

\begin{lemma}\label{lemma:key2:AGD}
There exist constant $c_{\max}, \hat c$ and a vector $\e$ with $\|\e\|\leq 1$ such that: for $\x$ satisfies the condition that $\lambda_{\min}(\nabla^2 f(\x))\leq - \gamma$
define a sequence $\w_t$ similar to $\u_t$ except $\w_0  =\u_0+ \mu r \e$, where $\mu\in[\delta/2\sqrt{d}, 1]$, and define a difference sequence  $\v_t = \w_t - \u_t$. Let 
\begin{align*}
T = \min\{\inf_\tau \{\tau| \hat f_\x(\y_\tau) - \hat f_\x(\y_0)\leq -3\mathcal F\}, \hat c\mathcal J\}.
\end{align*}
where $\y_\tau$ is the sequence generated by~(\ref{eqn:neonpn}) starting with $\w_0$. 
Then for any $\eta\leq c_{\max}/L_1$, if $\|\u_\tau\|\leq 6\hat c\mathcal P$ for all $\tau<T$, then $T<\hat c\mathcal J$.
\end{lemma}

\subsection{Continuing the Proof of Theorem~\ref{thm:main:AGD}}
\begin{proof}
When NCFind doesn't  terminate the algorithm, then we use Lemma~\ref{lemma:key1:AGD} and Lemma~\ref{lemma:key2:AGD} to prove that NEON$^+$ finds NC with a high probability.
Let us define $T_*=\hat c \mathcal J$ and $T' = \inf_{\tau}\{\tau| \hat f_{\x}(\y_\tau) - \hat f_{\x}(\y_0) \leq -3\mathcal F\}$, and consider following two scenarios:\\ \\
{\bf Scenario (1) $T'\leq T_*$:}  By the unified update in~(\ref{eqn:neonpn}), we have
\begin{align*}
    \y_{\tau+1} = \u_\tau - \eta\nabla \hat f_\x(\u_\tau)
\end{align*}
Then  
\[
\|\y_{T'}\|\leq (1+\eta L_1)\|\u_{T'-1}\|\leq 12\hat c \mathcal P\triangleq U
\]
By employing Lemma~\ref{lemma:key1:AGD}, we have
\begin{align}\label{eqn:ut:AGD}
f(\x + \y_{T'}) - f(\x)- \y_{T'}^{\top} \nabla f(\x) \leq& f(\x + \y_{0}) - f(\x)- \y_{0}^{\top}\nabla f(\x) - 3\mathcal F\nonumber\\
\leq& -3\mathcal F + \frac{L_1}{2}\|\y_0\|^2\leq -3\mathcal F + \mathcal F\leq -2\mathcal F,
\end{align}
Therefore we have
\begin{align*}
 \min_{1\leq \tau'\leq T_* \atop \|\y_{\tau'}\|\leq U, \y_\tau\sim \u_0}f(\x + \y_{\tau'}) - f(\x) -\y_{\tau'}^{\top} \nabla f(\x)
 \leq  f(\x + \y_{T'}) - f(\x)-\y_{T'}^{\top} \nabla f(\x) \leq  -2\mathcal F.
\end{align*}
where $\y_\tau\sim\u_0$ means that the generated sequence is starting from $\u_0$. 

{\bf Scenario (2) $T' > T_*$:} By Lemma~\ref{lemma:key1:AGD}, we have $\|\u_{\tau}\| \leq 6\mathcal P\hat c$ for all $\tau \leq T_*$. Let define $T'' = \inf_\tau \{\tau| \hat f_\x(\y_\tau) - \hat f_\x(\y_0) \leq -3\mathcal F\}$. Here, we abuse the same notation $\y_\tau$ to denote the generated sequence starting from $\w_0$.  By Lemma~\ref{lemma:key2:AGD}, we konw $T''\leq T_*$.  From the proof of Lemma~\ref{lemma:key2:AGD}, we also know that $\|\w_{T''-1}\|\leq 6\mathcal P \hat c$. Hence $\|\y_{T''}\|\leq 12\hat c \mathcal P$.  Similar to~(\ref{eqn:ut:AGD}), we have 
\begin{align*}
\min_{1\leq \tau'\leq T_* \atop  \|\y_{\tau'}\|\leq U, \y_\tau\sim \w_0} f(\x + \y_{\tau'}) - f(\x) -\y_{\tau'}^{\top} \nabla f(\x) \leq -2\mathcal F.
\end{align*}
Therefore,
\begin{align*}
\min&\left\{ \min_{1\leq\tau'\leq T_*\atop \|\y_{\tau'}\|\leq U, \y_\tau\sim \u_0}f(\x + \y_{\tau'}) - f(\x) -\y_{\tau'}^{\top} \nabla f(\x),\right.\\
&\left.\min_{1\leq\tau'\leq T_*\atop \|\y_{\tau'}\|\leq U, \y_\tau\sim\w_0}f(\x + \y_{\tau'}) - f(\x) -\y_{\tau'}^{\top} \nabla f(\x) \right\} 
\leq -2\mathcal F.
\end{align*}
Please recall that $\u_0$ is random vector following an uniform distribution over $\mathbb B_0(r)$ with radius $r = \mathcal P/(\kappa \cdot \log(d\kappa/\delta))$. Let $\mathcal X_s \subset \mathbb B_0(r)$ be the set of bad initial points such that $\min_{1\leq \tau'\leq T_* , \|\y_{\tau}'\|\leq U}f(\x + \y_{\tau'}) - f(\x) -\y_{\tau'}^{\top} \nabla f(\x) > -2\mathcal F$ when $\u_0 \in \mathcal X_s$; otherwise $\min_{1\leq \tau'\leq T_*, \|\y_{\tau}'\|\leq U}f(\x + \y_{\tau'}) - f(\x) -\y_{\tau'}^{\top} \nabla f(\x) \le -2\mathcal F$ when $\u_0 \in  \mathbb B_0(r) - \mathcal X_s$.

From our analysis, for any $\u_0\in \mathcal X_s$, we have $(\u_0 \pm \mu r \e) \not \in \mathcal X_s $ where $\mu \in [\frac{\delta}{2\sqrt{d}}, 1]$.  Denote by $I_{\mathcal X_s}(\cdot)$ the indicator function of being inside set $\mathcal X_s$. We set $u^{(1)}$ as the component along $\e$ direction and $\u^{(-1)}$ as the remaining $d-1$ dimensional vector, then the vector $\u = (u^{(1)}, \u^{(-1)})$. The volume of $\mathcal X_s$'s can be upper bounded as:
\begin{align*}
& \text{Vol}(\mathcal X_s) = \int_{\mathbb B_0^{(d)}(r)}  \mathrm{d}\u \cdot I_{\mathcal X_s}(\u) 
=  \int_{\mathbb B_0^{(d-1)}(r)}  \mathrm{d}\u^{(-1)} \int_{-\sqrt{r^2 - \|\u^{(-1)}\|^2}}^{\sqrt{r^2 - \|\u^{(-1)}\|^2}} \mathrm{d} u^{(1)} \cdot I_{\mathcal X_s}(\u)\\
\leq & \int_{\mathbb B_0^{(d-1)}(r)}  \mathrm{d}\u^{(-1)} \cdot 2\mu r 
 \leq \int_{\mathbb B_0^{(d-1)}(r)}  \mathrm{d}\u^{(-1)}  \cdot 2 \frac{\delta}{2\sqrt{d}}r =   \text{Vol}(\mathbb B_0^{(d-1)}(r)) \frac{\delta r}{\sqrt{d}} 
\end{align*}
Then, 
\begin{align*}
&\frac{\text{Vol}(\mathcal X_s)}{\text{Vol}(\mathbb B_0^{(d)}(r)) }
\le \frac{ \text{Vol}(\mathbb B_0^{(d-1)}(r)) \frac{\delta r}{\sqrt{d}}}{\text{Vol}(\mathbb B_0^{(d)}(r)) }
= \frac{\delta}{\sqrt{\pi d}}\frac{\Gamma(\frac{d}{2}+1)}{\Gamma(\frac{d}{2}+\frac{1}{2})}
\le  \frac{\delta}{\sqrt{\pi d}} \cdot \sqrt{\frac{d}{2}+\frac{1}{2}} \le \delta,
\end{align*}
where the second inequality is due to $\frac{\Gamma(x+1)}{\Gamma(x+1/2)}<\sqrt{x+\frac{1}{2}}$ for all $x\ge 0$.
Thus, we have $\u_0 \not \in \mathcal X_s$ with at least probability $1-\delta$. Therefore, with probability at least $1 - \delta$
%
\begin{align}\label{ineq:keylemma:1:AGD}
\min_{1\leq \tau'\leq T_* \atop \|\y_{\tau'}\|\leq U, \y_\tau\sim\u_0}f(\x + \y_{\tau'}) - f(\x) -\y_{\tau'}^{\top} \nabla f(\x) \leq -2\mathcal F.
\end{align}
Let us complete the proof.  By the Lipschitz continuity of Hessian, we have 
\begin{align*}
\left| f(\x+\y_{\tau}) - f(\x) - \nabla f(\x)^\top\y_\tau - \frac{1}{2}\y_{\tau}^\top\nabla^2 f(\x)\y_{\tau}\right| \leq \frac{L_2}{6}\|\y_{\tau}\|^3.
\end{align*}
Then by inequality (\ref{ineq:keylemma:1:AGD}), we get
\begin{align*}
  \frac{1}{2}\y_\tau^\top\nabla^2 f(\x)\y_\tau 
  \leq  f(\x+\y_\tau) - f(\x) - \nabla f(\x)^\top\y_\tau + \frac{L_2}{6}\|\y_\tau\|^3 
   \leq  -2\mathcal F + \mathcal F  \leq  -\mathcal F.
\end{align*}
Therefore, 
\begin{align}\label{neon2_upp_3}
\frac{\y_\tau^\top\nabla^2 f(\x)\y_\tau}{\|\y_\tau\|^2}\leq -\frac{2\mathcal F}{(12\hat c\mathcal P)^2}
    = - \frac{\gamma}{72\hat c^2 \log(dL_1 /(\gamma\delta))},
\end{align}
If \neonp returns $0$ it is not difficult to prove that $\lambda_{\min}(\nabla^2 f(\x))\geq - \gamma$ holds with high probability $1-O(\delta)$ for a sufficiently small $\delta$ by Bayes theorem.
\end{proof}

\subsection{Proof of Lemma~\ref{lemma:key1:AGD}}
\begin{proof}
 By using the smoothness of $\hat f(\u)$, we have
\begin{align*}
\hat f(\y_{\tau+1}) 
\leq & \hat  f(\u_\tau) + \nabla \hat f(\u_\tau)^\top (\y_{\tau+1} - \u_\tau) + \frac{L_1}{2}\|\y_{\tau+1} - \u_\tau\|^2 \\
= & \hat f(\u_\tau) - \eta \|\nabla \hat f(\u_\tau)\|^2 + \frac{1}{2}\eta^2 L_1\|\nabla \hat f(\u_\tau)\|^2,
\end{align*}
where the last equality uses the update (\ref{eqn:neonpn}). 
Since the NCFind never happens during $t$ iterations, then the inequality (\ref{cond:largeNC}) doesn't hold for any $\tau$,
thus we have
\begin{align*}
\hat f(\y_{\tau+1})  + \nabla \hat f(\u_\tau)^\top (\y_{\tau} - \u_\tau) - \frac{\gamma}{2}\|\y_{\tau} - \u_\tau\|^2 
\leq \hat f(\y_{\tau})  - \eta \|\nabla \hat f(\u_\tau)\|^2 + \frac{1}{2}\eta^2 L_1\|\nabla \hat f(\u_\tau)\|^2.
\end{align*}
By the update (\ref{eqn:neonpn}), we have
\begin{align*}
\| \y_{\tau+1} - \y_\tau\|^2 =&  \| \u_{\tau} - \y_\tau - \eta \nabla \hat f(\u_\tau)\|^2 \\
= & \| \u_{\tau} - \y_\tau \|^2 -2 \eta \nabla \hat f(\u_\tau)^\top (\u_{\tau} - \y_\tau) + \eta^2 \|\nabla \hat f(\u_\tau)\|^2
\end{align*}
and $\|\u_\tau - \y_\tau\| = \zeta \|\y_\tau - \y_{\tau-1}\|$. Thus,
\begin{align}\label{eqn:lm8key}
 \hat f(\y_{\tau+1}) + \frac{1}{2\eta} \| \y_{\tau+1} - \y_\tau\|^2  \leq &\hat f(\y_{\tau})  - \frac{\eta(1-\eta L_1)}{2} \|\nabla \hat f(\u_\tau)\|^2 +\frac{1+2\eta\gamma}{2\eta} \| \u_{\tau} - \y_\tau \|^2\notag\\
\leq &\hat f(\y_{\tau}) +\frac{1+2\eta\gamma}{2\eta} \| \u_{\tau} - \y_\tau \|^2\notag\\
= &\hat f(\y_{\tau})  +\frac{(1+2\eta\gamma)\zeta^2}{2\eta} \| \y_\tau - \y_{\tau-1}\|^2\notag \\
= &\hat f(\y_{\tau})  + \frac{1}{2\eta} \| \y_{\tau} - \y_{\tau-1}\|^2 -\frac{1-(1+2\eta\gamma)\zeta^2}{2\eta} \| \y_\tau - \y_{\tau-1}\|^2\notag\\
\leq &\hat f(\y_{\tau})  + \frac{1}{2\eta} \| \y_{\tau} - \y_{\tau-1}\|^2 -\frac{\sqrt{\eta\gamma}}{2\eta} \| \y_\tau - \y_{\tau-1}\|^2.
\end{align}
where the second inequality uses $\eta L_1 <1$; and the last inequality uses $1- (1+\eta \gamma)\zeta^2 \geq 1 - \zeta$ by choosing $\zeta = 1 - \sqrt{\eta\gamma}$ with $\sqrt{\eta\gamma} \leq 1/2$.

By summing up $\tau$ from $0$ to $t-1$ where $t<T$, we have
\begin{align*}
\hat f(\y_{t})+\frac{1}{2\eta}\|\y_{t} - \y_{t-1}\|^2
 \leq  \hat f(\y_{0}) - \frac{\sqrt{\eta\gamma}}{2\eta}\sum_{\tau=0}^{t-1}\|\y_{\tau} - \y_{\tau-1}\|^2,
\end{align*}
which inplies
\begin{align*}
\frac{\sqrt{\eta\gamma}}{2\eta}\sum_{\tau=0}^{t-1}\|\y_{\tau} - \y_{\tau-1}\|^2  \leq & \hat f(\y_{0}) - \hat f(\y_{t}).
\end{align*}
Hence, if $t<T$ then $\hat f(\y_t) - \hat f(\y_0)\geq -3\mathcal F$, i.e., $\hat f(\y_0) - \hat f(\y_t) \leq 3\mathcal F$. Then we have
\begin{align*}
&\sum_{\tau=0}^{t-1}\|\y_{\tau} - \y_{\tau-1}\| \leq \sqrt{t\sum_{\tau=0}^{t-1}\|\y_{\tau} - \y_{\tau-1}\|^2}
\leq  \sqrt{\frac{3\hat c \log^{1/2}(\frac{d\kappa}{\delta})}{\sqrt{\eta\gamma}}  \frac{\eta L_1\gamma^3}{L_2^2 \log^3(\frac{d\kappa}{\delta})}\frac{2\eta }{\sqrt{\eta\gamma}}}\\
\leq & \sqrt{\frac{3\hat c}{\sqrt{\eta\gamma}}\frac{ \eta L_1 \gamma^3}{L_2^2\log^2(\frac{d\kappa}{\delta})}\frac{2\eta }{\sqrt{\eta\gamma}}} 
= \sqrt{6\hat c}\sqrt{\eta L_1} \frac{\gamma}{L_2}\log^{-1}(d\kappa/\delta) = \sqrt{6\hat c}  \mathcal P.
\end{align*}
Since $\sum_{\tau=1}^{t-1}\|\u_{\tau} - \u_{\tau-1}\|\leq \sum_{\tau=0}^{t-1}\|\y_{\tau+1} - \y_{\tau}\|$, then $\|\u_\tau\|\leq \sum_{t=1}^{\tau}\|\u_{t} -\u_{t-1}\| + \|\u_0\|\leq \sqrt{6\hat c} \mathcal P + \mathcal P \leq \hat c \mathcal P$, for all $\tau\leq t-1$, where the last inequality is due to $\hat c \geq 43$. Additionally, by the unified update in~(\ref{eqn:neonpn}), we have
\begin{align*}
\|\u_{\tau+1}\|\leq& (1+\eta L_1)\|\u_{\tau}\| + \zeta(1+\eta L_1))\|\u_{\tau}\|  +  \zeta(1+\eta L_1))\|\u_{\tau-1}\|\leq 6\hat c \mathcal P.
\end{align*}
\end{proof}

\subsection{Proof of Lemma~\ref{lemma:key2:AGD}}
\begin{proof}
The techniques for proving this lemma are largely borrowed from~\citep{AGNON}. In particular,  the proof is almost the same as the proof of Lemma 18 in~\citep{AGNON}. For completeness, we include it in this subsection. First, the update of NAG can be written as 
\begin{align*}
&\u_{\tau+1} + \v_{\tau+1} = \w_{\tau+1} \\
=&(1+\zeta)[\w_\tau - \eta( \nabla f(\x+\w_\tau) - \nabla f(\x))] - \zeta[\w_{\tau-1} - \eta( \nabla f(\x+\w_{\tau-1}) - \nabla f(\x))] \\
=&(1+\zeta)[ \u_\tau + \v_\tau - \eta( \nabla f(\x+ \u_\tau + \v_\tau) - \nabla f(\x))] \\ &- \zeta[ \u_{\tau-1} + \v_{\tau-1} - \eta( \nabla f(\x+\u_{\tau-1} + \v_{\tau-1}) - \nabla f(\x))]\\
=&(1+\zeta)[ \u_\tau  - \eta( \nabla f(\x+ \u_\tau ) - \nabla f(\x))] - \zeta[ \u_{\tau-1} - \eta( \nabla f(\x+\u_{\tau-1} ) - \nabla f(\x))] \\
&+(1+\zeta)[  \v_\tau  - \eta H\v_\tau - \eta\Delta_\tau \v_\tau ] - \zeta[  \v_{\tau-1} - \eta H\v_{\tau-1} - \eta\Delta_{\tau-1} \v_{\tau-1}],
\end{align*}
where $\Delta_\tau = \int_{0}^1 \nabla^2 f(\x+\u_\tau+\theta \v_\tau)) \mathrm{d}\theta - H$.
Then
\begin{align*}
\v_{\tau+1} =& (1+\zeta)[  \v_\tau  - \eta H\v_\tau - \eta\Delta_\tau \v_\tau ] - \zeta[  \v_{\tau-1} - \eta H\v_{\tau-1} - \eta\Delta_{\tau-1} \v_{\tau-1}].
\end{align*}
and
\begin{align*}
\left[\begin{array}{cc}\v_{\tau+1}\\ \v_\tau\end{array}\right] 
=  \underbrace{\left[\begin{array}{cc}(1 + \zeta)(I - \eta H)& -\zeta(I - \eta H)\\ I & 0 \end{array}\right]}\limits_{A}\left[\begin{array}{cc}\v_\tau\\ \v_{\tau-1}\end{array}\right] -\eta \left[\begin{array}{c}(1 + \zeta)\Delta_\tau \v_\tau - \zeta\Delta_{\tau-1}\v_{\tau-1} \\ 0 \end{array}\right],
\end{align*}
where $\Delta_\tau = \int_{0}^1 \nabla^2 f(\x+\u_\tau+\theta \v_\tau)) \mathrm{d}\theta - H$. It is easy to show that $\|\Delta_\tau\|\leq L_2 (\|\u_\tau\| + \|\v_\tau\|) \leq 18 L_2  \mathcal P \hat c $. Let $\delta_{\tau} = (1 + \zeta)\Delta_\tau \v_\tau - \zeta\Delta_{\tau-1}\v_{\tau-1} $, then
\begin{align}\label{update:NAG:1}
 \left[\begin{array}{cc}\v_{\tau+1}\\ \v_\tau\end{array}\right] = A \left[\begin{array}{cc}\v_\tau\\ \v_{\tau-1}\end{array}\right] -\eta \left[\begin{array}{c}\delta_\tau \\ 0 \end{array}\right] 
= A^{\tau+1} \left[\begin{array}{cc}\v_0\\ \v_{-1}\end{array}\right] -\eta \sum_{i=0}^{\tau} A^{\tau-i} \left[\begin{array}{c}\delta_{i} \\ 0 \end{array}\right].
\end{align}
We can write $\v_\tau$ by (\ref{update:NAG:1}) as follows :
\begin{align*}
\v_\tau = \left[ \begin{array}{cc} I & 0 \end{array} \right]  A^{\tau} \left[\begin{array}{cc}\v_0\\ \v_{0}\end{array}\right] -\eta \left[ \begin{array}{cc} I & 0 \end{array} \right]  \sum_{i=0}^{\tau-1} A^{\tau-1-i} \left[\begin{array}{c}\delta_{i} \\ 0 \end{array}\right],
\end{align*}
where uses the fact that $\v_{-1} = \v_0$. Next we will show the following inequality by induction:
\begin{align}\label{ineq:key1}
 \frac{1}{2} \left\| \left[ \begin{array}{cc} I & 0 \end{array} \right]  A^{\tau} \left[\begin{array}{cc}\v_0\\ \v_{0}\end{array}\right] \right\| 
\geq \left \| \eta \left[ \begin{array}{cc} I & 0 \end{array} \right]  \sum_{i=0}^{\tau-1} A^{\tau-1-i} \left[\begin{array}{c}\delta_{i} \\ 0 \end{array}\right] \right\|
\end{align}
It is easy to show that the above inequality  (\ref{ineq:key1}) holds for $\tau = 0$. We assume the inequality (\ref{ineq:key1}) holds for all $\tau$. Then,
\begin{align*}
\|\v_\tau\| \leq & \left\| \left[ \begin{array}{cc} I & 0 \end{array} \right]  A^{\tau} \left[\begin{array}{cc}\v_0\\ \v_{0}\end{array}\right] \right\| + \left \| \eta \left[ \begin{array}{cc} I & 0 \end{array} \right]  \sum_{i=0}^{\tau-1} A^{\tau-1-i} \left[\begin{array}{c}\delta_i \\ 0 \end{array}\right] \right\| \\
\leq & \frac{3}{2}\left\| \left[ \begin{array}{cc} I & 0 \end{array} \right]  A^{\tau} \left[\begin{array}{cc}\v_0\\ \v_{0}\end{array}\right] \right\|.
\end{align*}
On the other hand, we have
$\|\delta_{\tau}\| \leq (1 + \zeta)\|\Delta_\tau\| \|\v_\tau\| + \zeta\|\Delta_{\tau-1}\|\|\v_{\tau-1}\| \leq 54 L_2 \mathcal P \hat c (\|\v_\tau\| + \|\v_{\tau-1}\|)$, 
then we get
\begin{align*}
\|\delta_{\tau}\| \leq& 54 L_2 \mathcal P \hat c (\|\v_\tau\| + \|\v_{\tau-1}\|) \\
\leq & 81 L_2 \mathcal P \hat c \left(\left\| \left[ \begin{array}{cc} I & 0 \end{array} \right]  A^{\tau} \left[\begin{array}{cc}\v_0\\ \v_{0}\end{array}\right] \right\| +\left\| \left[ \begin{array}{cc} I & 0 \end{array} \right]  A^{\tau-1} \left[\begin{array}{cc}\v_0\\ \v_{0}\end{array}\right] \right\|\right) \\
\leq & 162 L_2 \mathcal P \hat c \left\| \left[ \begin{array}{cc} I & 0 \end{array} \right]  A^{\tau} \left[\begin{array}{cc}\v_0\\ \v_{0}\end{array}\right] \right\|,
\end{align*}
where the last inequality uses the monotonic property in terms of $\tau$ in Lemma~33 in \citep{AGNON} . 
We then need to prove the inequality (\ref{ineq:key1}) holds for $\tau + 1$.
We consider the following term for $\tau + 1$:
\begin{align}\label{ineq:key2}
\nonumber & \left \| \eta \left[ \begin{array}{cc} I & 0 \end{array} \right]  \sum_{i=0}^{\tau} A^{\tau-i} \left[\begin{array}{c}\delta_{i} \\ 0 \end{array}\right] \right\| 
\leq   \eta \sum_{i=0}^{\tau}  \left \| \left[ \begin{array}{cc} I & 0 \end{array} \right]  A^{\tau-i}  \left[\begin{array}{c}I\\ 0 \end{array}\right]  \right\| \left \| \delta_{i} \right\|   \\
\leq &    162 \eta L_2 \mathcal P \hat c  \sum_{i=0}^{\tau} \left \{ \left \| \left[ \begin{array}{cc} I & 0 \end{array} \right]  A^{\tau-i}  \left[\begin{array}{c}I\\ 0 \end{array}\right]  \right\| \left\| \left[ \begin{array}{cc} I & 0 \end{array} \right]  A^{i} \left[\begin{array}{cc}\v_0\\ \v_{0}\end{array}\right] \right\| \right \}.
\end{align} 
We know from the preconditions that $\lambda_{\min}(H)\leq - \gamma$ and the coordinate $\e_1$ is along the minimum eigenvector direction of Hessian matrix $H$, then we let the corresponding $2\times 2$ matrix as $A_1$ and
\begin{align*}
\left[ \begin{array}{cc} a_\tau^{(1)} & -b_\tau^{(1)} \end{array} \right] = \left[ \begin{array}{cc} 1 & 0 \end{array} \right] A_1^\tau.
\end{align*}
Since $\v_0 = \w_0 -\u_0 = \mu r \e_1$, $\v_0$ is along the $\e_1$ direction. By the analysis of Lemma 32 in \citep{AGNON}, the matrix $\left[ \begin{array}{cc} I & 0 \end{array} \right]  A^{\tau-i}  \left[\begin{array}{c}I\\ 0 \end{array}\right] $ is diagonal, and thus the spectral norm is obtained along $\e_1$ which corresponding to $\lambda_{\min}(H)$.
Thus, inequality (\ref{ineq:key2}) can be written as
\begin{align*}
& \left \| \eta \left[ \begin{array}{cc} I & 0 \end{array} \right]  \sum_{i=0}^{\tau} A^{\tau-i} \left[\begin{array}{c}\delta_i \\ 0 \end{array}\right] \right\| 
\leq  \chi \sum_{i=0}^{\tau} \left \| \left[ \begin{array}{cc} I & 0 \end{array} \right]  A^{\tau-i}  \left[\begin{array}{c}I\\ 0 \end{array}\right]  \right\| \left\| \left[ \begin{array}{cc} I & 0 \end{array} \right]  A^{i} \left[\begin{array}{cc}\v_0\\ \v_{0}\end{array}\right] \right\| \\
\leq &  \chi  \sum_{i=0}^{\tau} a_{\tau-i}^{(1)} (a_{i}^{(1)} -b_{i}^{(1)} ) \|\v_0\| 
\leq   \chi  \sum_{i=0}^{\tau} \left(\frac{2}{1-\zeta}+(\tau+1)\right) (a_{\tau+1}^{(1)} -b_{\tau+1}^{(1)} ) \|\v_0\| \\
\leq &  \chi (\tau+1)\left(\frac{2}{1-\zeta}+(\tau+1)\right) \left\| \left[ \begin{array}{cc} I & 0 \end{array} \right]  A^{\tau+1} \left[\begin{array}{cc}\v_0\\ \v_{0}\end{array}\right] \right\|\\
\leq & \chi (\hat c\mathcal J) \left(\frac{2}{1-\zeta}+\hat c\mathcal J\right) \left\| \left[ \begin{array}{cc} I & 0 \end{array} \right]  A^{\tau+1} \left[\begin{array}{cc}\v_0\\ \v_{0}\end{array}\right] \right\|,
\end{align*} 
where $\chi = 162 \eta L_2 \mathcal P \hat c$.
By choosing $2/(1-\zeta)\leq \hat c\mathcal J$, i.e. $\zeta \leq 1 - \frac{2\sqrt{\eta\gamma}}{\hat c \log(\kappa d/\delta)}$, then $162 \eta L_2 \mathcal P \hat c (\hat c\mathcal J) (2/(1-\zeta)+\hat c\mathcal J) \leq 312 \eta L_2 \mathcal P \hat c (\hat c\mathcal J)^2 = 312 \hat c^3 \sqrt{\eta L_1} = \frac{1}{2}$ if we select $\eta = \frac{c_{\max}}{L_1}$ with $c_{\max} = \frac{1}{624^2 \hat c^6}$. Therefore, we have shown that the inequality (\ref{ineq:key1}) holds.
Further, we have
\begin{align*}
\|\v_\tau\| \geq & \left\| \left[ \begin{array}{cc} I & 0 \end{array} \right]  A^{\tau} \left[\begin{array}{cc}\v_0\\ \v_{0}\end{array}\right] \right\| - \left \| \eta \left[ \begin{array}{cc} I & 0 \end{array} \right]  \sum_{i=0}^{\tau-1} A^{\tau-1-i} \left[\begin{array}{c}\delta_{i} \\ 0 \end{array}\right] \right\| \\
\geq & \frac{1}{2}\left\| \left[ \begin{array}{cc} I & 0 \end{array} \right]  A^{\tau} \left[\begin{array}{cc}\v_0\\ \v_{0}\end{array}\right] \right\|
\end{align*}
Noting that $\lambda_{\min}(H)\leq - \gamma$, by Lemmas~23, 30, 33 in \citep{AGNON}, we have
\begin{align*}
&\frac{1}{2}\left\| \left[ \begin{array}{cc} I & 0 \end{array} \right]  A^{\tau} \left[\begin{array}{cc}\v_0\\ \v_{0}\end{array}\right] \right\| \geq \left\{ \begin{array}{rcl}
       \frac{\sqrt{|x|}}{4} (1 + \frac{\sqrt{|x|}}{2})^\tau \|\v_0\|, && x\in\left[-\frac{1}{4}, -(1-\zeta)^2\right],\\ \\
        \frac{1-\zeta}{4}(1+\frac{|x|}{2(1-\zeta)})^\tau \|\v_0\|, &&  x\in\left[-(1-\zeta)^2, 0\right],
                \end{array}\right.
\end{align*}
where $|x| = |\eta \lambda_{\min}(H) | \geq  \eta \gamma$. By choosing $1-\zeta = \sqrt{\eta\gamma}$, we have
\begin{align*}
\frac{1}{2}\left\| \left[ \begin{array}{cc} I & 0 \end{array} \right]  A_n^{\tau} \left[\begin{array}{cc}\v_0\\ \v_{0}\end{array}\right] \right\| \geq 
       \frac{\sqrt{\eta \gamma}}{4} \left(1 + \frac{\sqrt{\eta \gamma}}{2}\right)^\tau \mu r
\end{align*}
Combining $\|\v_\tau\| \leq 12\mathcal P \hat c$, for all $\tau < T$, we have
\begin{align*}
12\mathcal P \hat c \geq & \frac{\sqrt{\eta \gamma}}{4}  \left(1 + \frac{\sqrt{\eta \gamma}}{2}\right)^\tau \mu r  
= \frac{\sqrt{\eta \gamma}}{4}  \left(1 + \frac{\sqrt{\eta \gamma}}{2}\right)^\tau\mu\frac{\mathcal P}{\kappa}\log^{-1}(d\kappa/\delta) \\
\geq & \frac{\sqrt{\eta \gamma}}{4}  \left(1 + \frac{\sqrt{\eta \gamma}}{2}\right)^\tau \frac{\delta}{2\sqrt{d}} \frac{\mathcal P}{\kappa}\log^{-1}(d\kappa/\delta) \\
= & \frac{\sqrt{c_{\max}\gamma/L_1}}{4}  \left(1 + \frac{\sqrt{\eta \gamma}}{2}\right)^\tau \frac{\delta}{2\sqrt{d}} \frac{\mathcal P}{\kappa}\log^{-1}(d\kappa/\delta) 
\end{align*}
Then
\begin{align*}
T < & \frac{\log (96 \frac{\kappa^{3/2}\sqrt{d}}{ \sqrt{c_{\max}}\delta}\cdot \hat c \log(d\kappa/\delta))}{\log (1+\sqrt{\gamma \eta}/2)}
\leq \frac{\log (96\frac{\kappa^{3/2}\sqrt{d}}{ \sqrt{c_{\max}}\delta}\cdot \hat c \log(d\kappa/\delta))}{2\sqrt{\gamma \eta}/3} \\
\leq&  1.5(2.5 + \log (96 \hat c/\sqrt{c_{\max}}))\mathcal J,
\end{align*}
where the last inequality holds because of $\delta\in (0, \frac{d\kappa}{e}]$ and $\log (d\kappa/\delta) \ge 1$.
By choosing $\hat c \geq 43$, we have $1.5(2.5 + \log (96 \hat c/\sqrt{c_{\max}})) \leq \hat c$, then
$T <  \hat c \mathcal J $ and complete the proof.
\end{proof}

\section{Additional simulation results}
\begin{figure}[H] 
\centering
 \subfigure[NEON on full data\label{fig:a1}]{\includegraphics[scale=0.37]{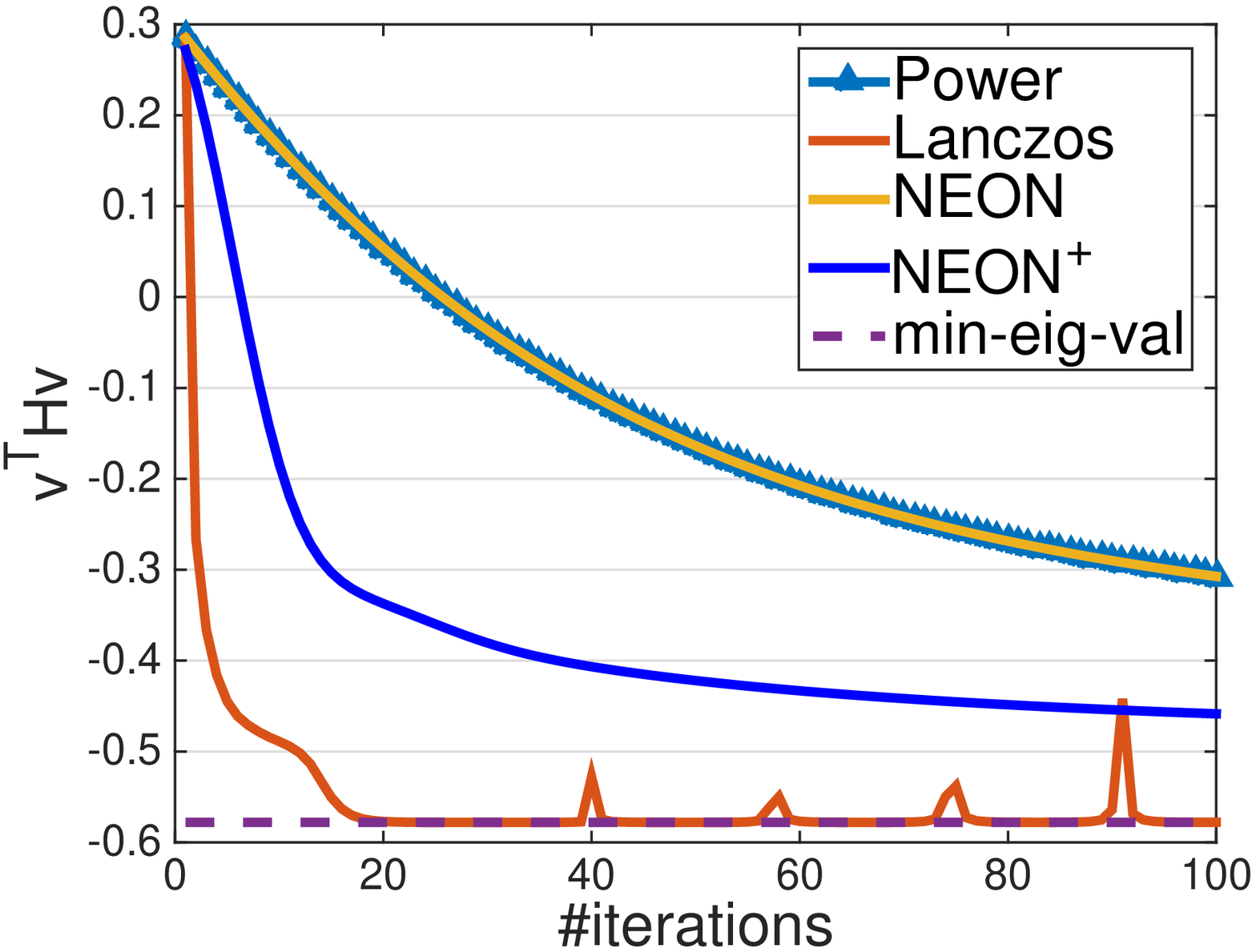}} 
 \subfigure[NEON on sub-sampled data\label{fig:b1}]{\includegraphics[scale=0.37]{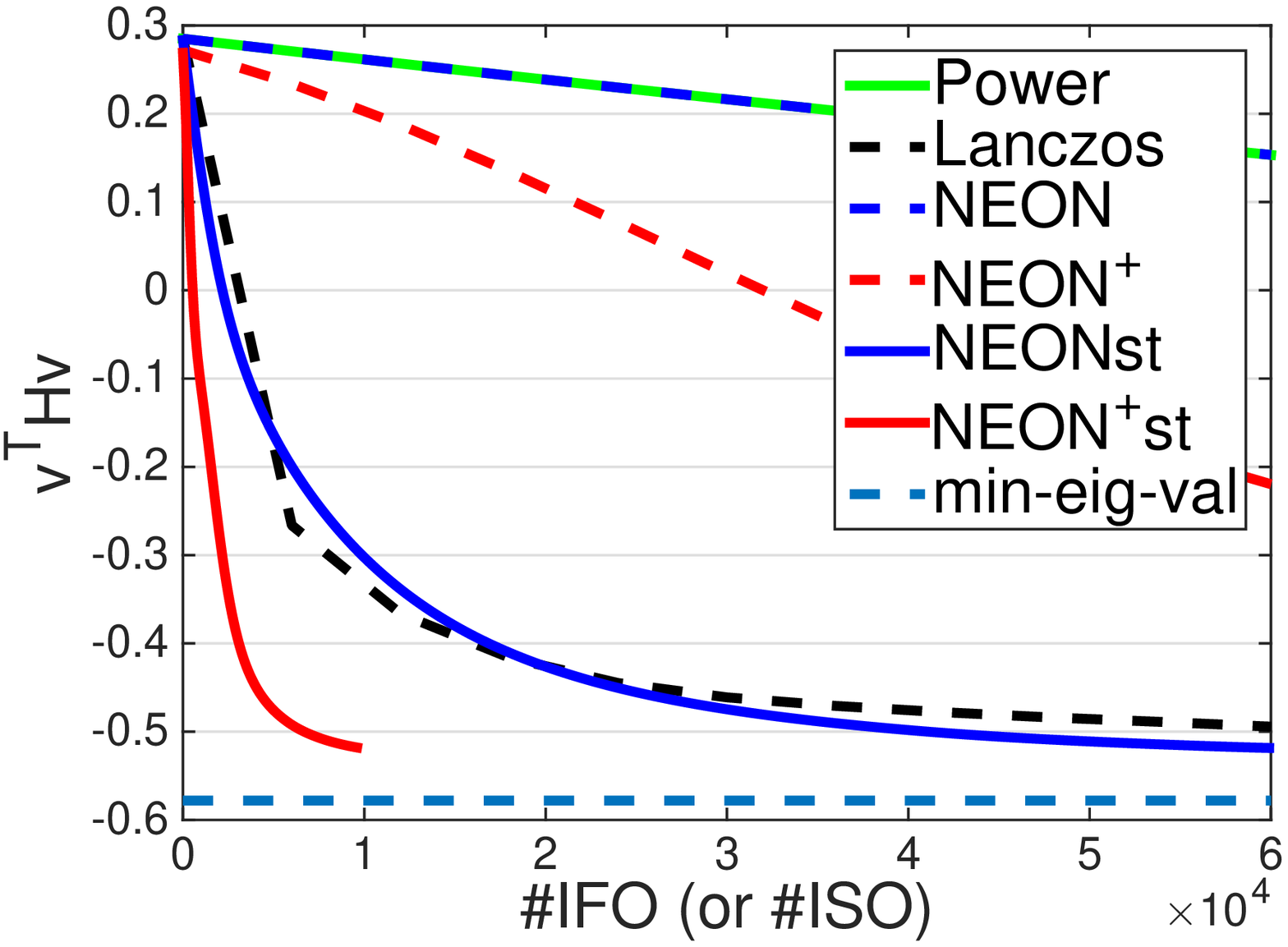}} 
\caption{Comparison between different NEON procedures and Second-order Methods}
\label{fig:neon:supp1}
\end{figure}
\begin{figure}[H] 
\centering
  \subfigure[NEON on full data\label{fig:a2}]{\includegraphics[scale=0.37]{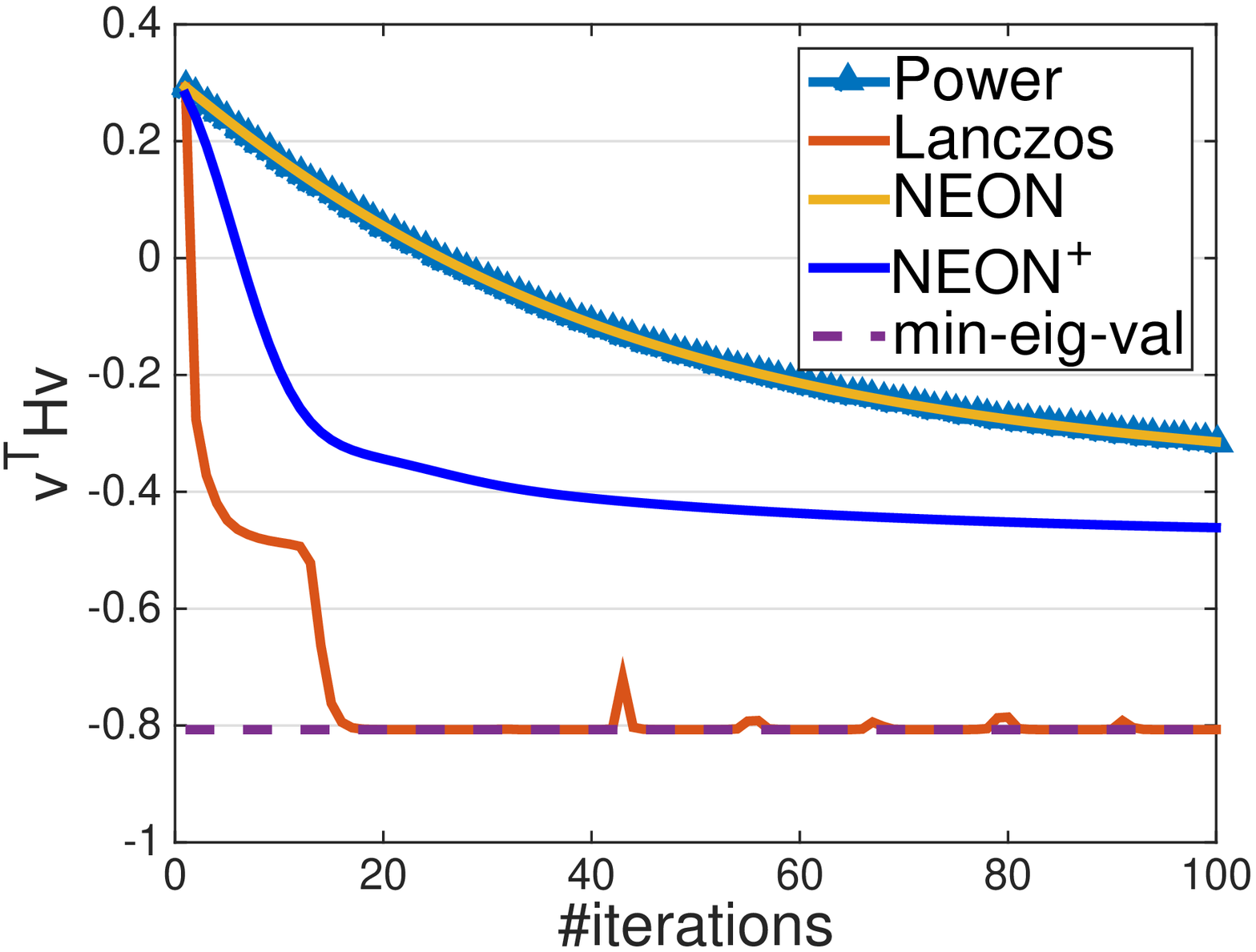}} 
    \subfigure[NEON on sub-sampled data\label{fig:b2}]{\includegraphics[scale=0.37]{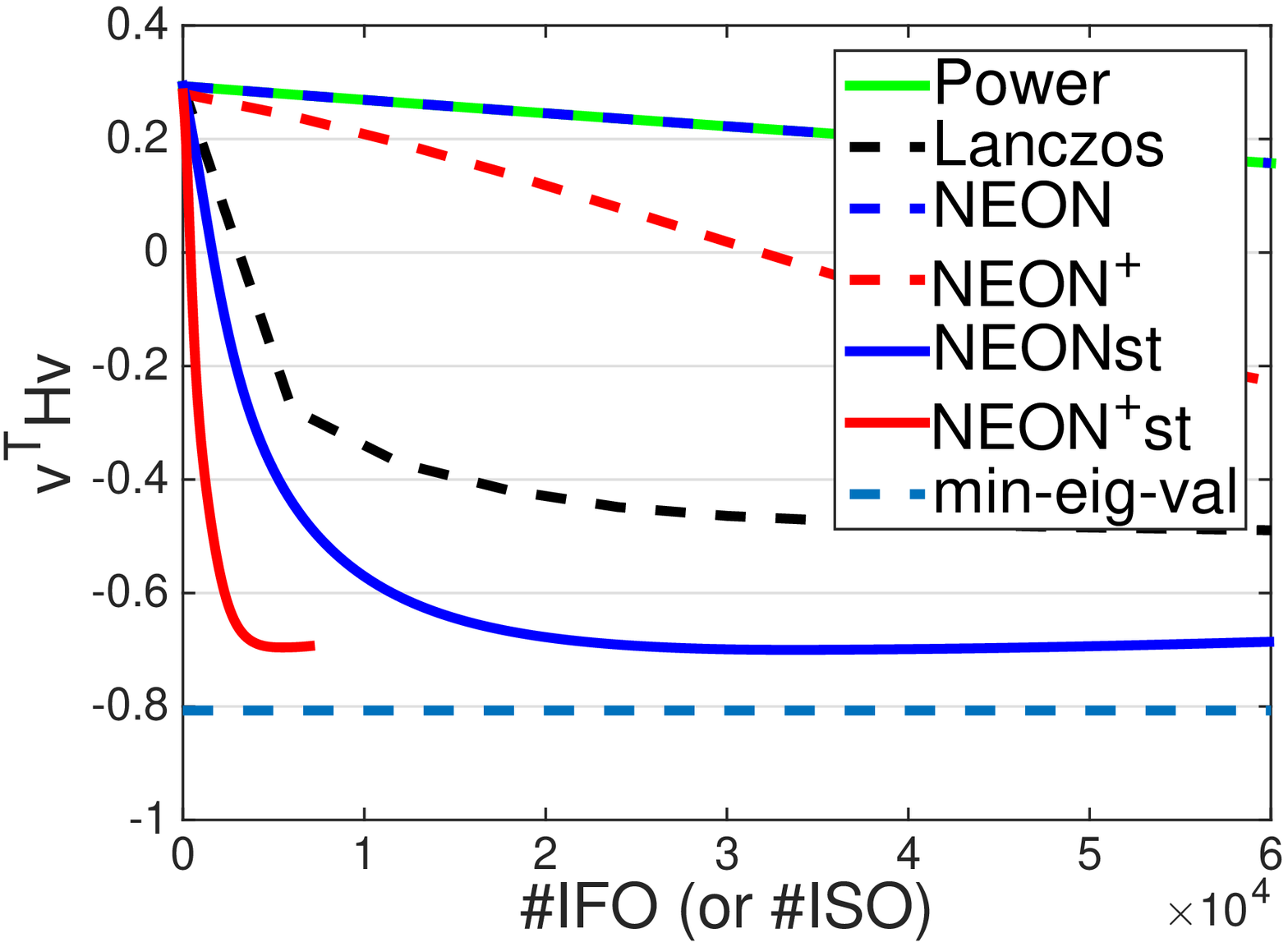}} 
\caption{Comparison between different NEON procedures and Second-order Methods}
\label{fig:neon:supp2}
\end{figure}
\begin{figure}[H] 
\centering
  \subfigure[NEON on full data\label{fig:a3}]{\includegraphics[scale=0.37]{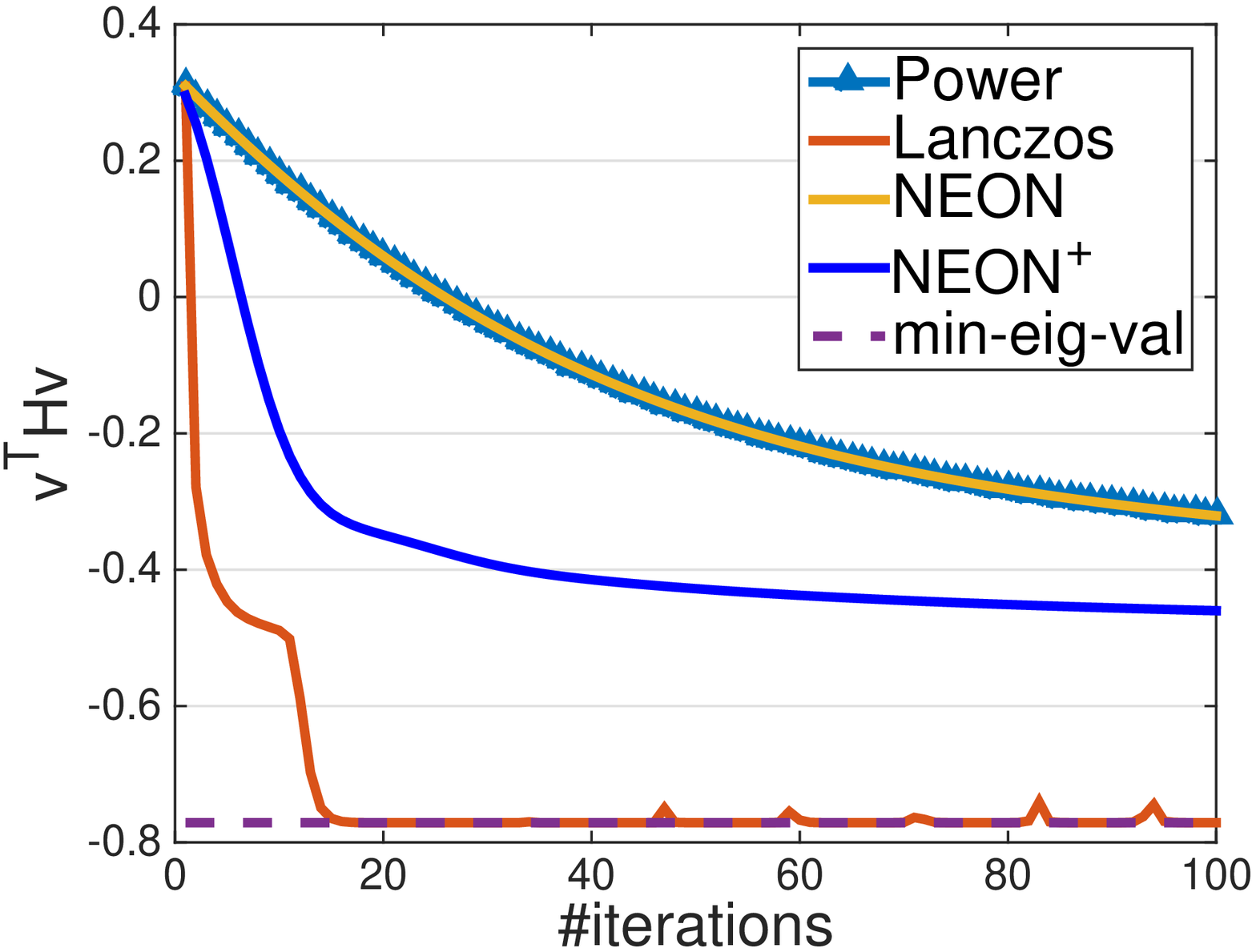}} 
    \subfigure[NEON on sub-sampled data\label{fig:b3}]{\includegraphics[scale=0.37]{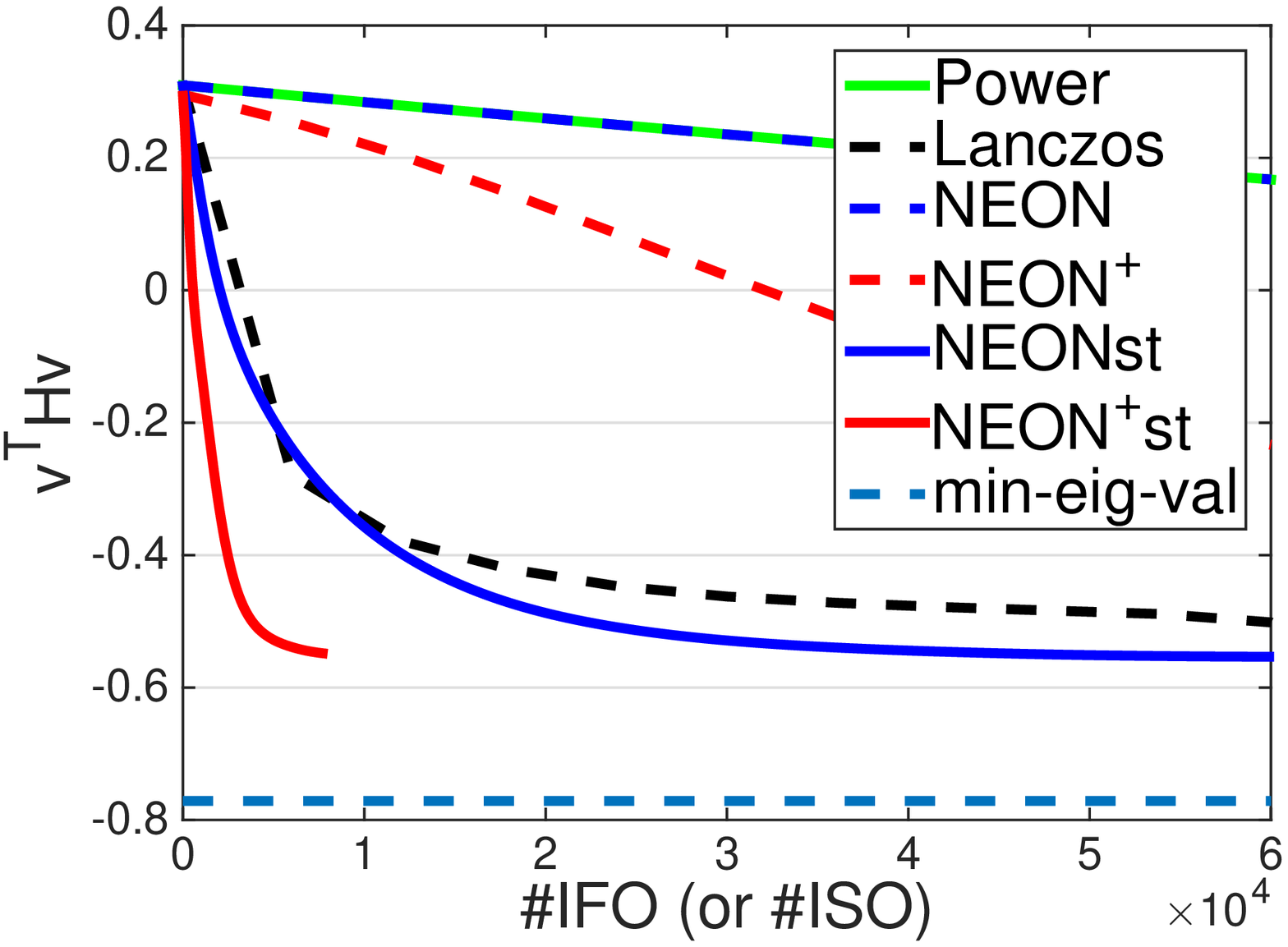}} 
\caption{Comparison between different NEON procedures and Second-order Methods}
\label{fig:neon:supp3}
\end{figure}
\begin{figure}[H] 
\centering
  \subfigure[NEON on full data\label{fig:a4}]{\includegraphics[scale=0.37]{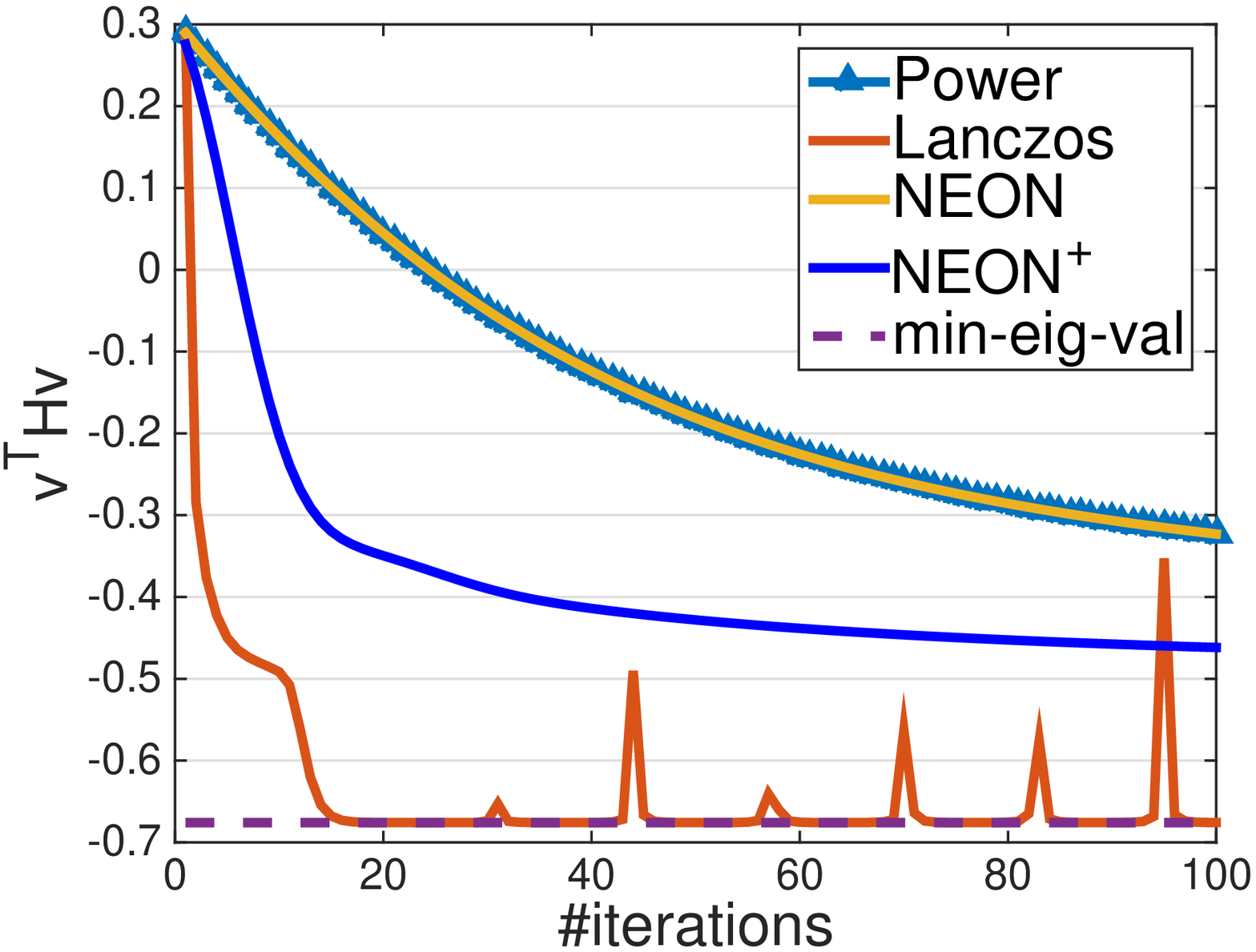}} 
    \subfigure[NEON on sub-sampled data\label{fig:b4}]{\includegraphics[scale=0.37]{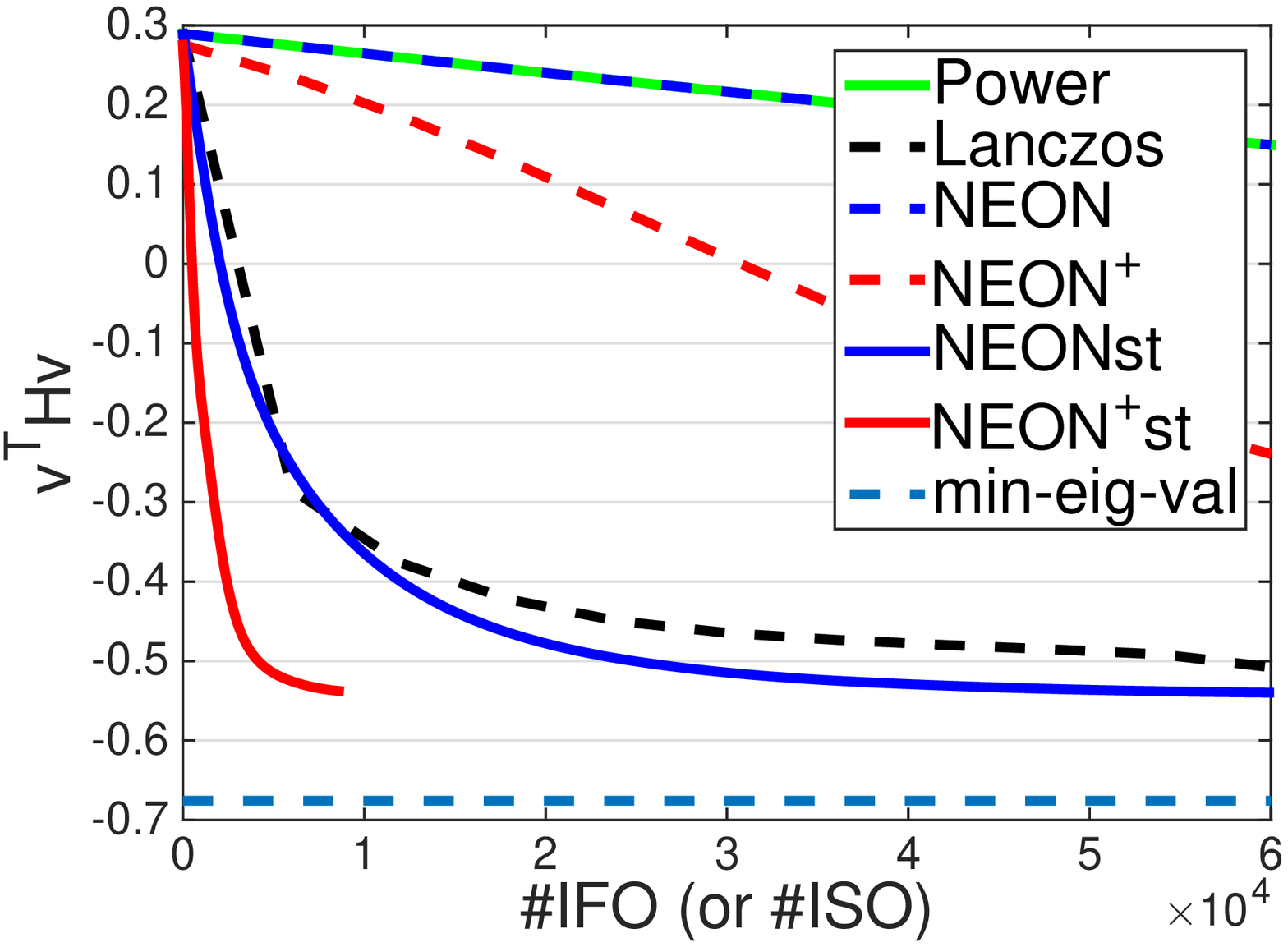}} 
\caption{Comparison between different NEON procedures and Second-order Methods}
\label{fig:neon:supp4}
\end{figure}

\section{Analysis of eigen-gap}
In subsection ``4.1 Finding NC by Accelerated Gradient Method", we mentioned that under some assumptions, the eigen-gap $\Delta_k$ of $A$ corresponding to its top$-k$ eigen-space is at least $\sqrt{\eta\gamma}/2$ by choosing $\zeta = 1-\sqrt{\eta\gamma}$. We state the result in the following lemma.
\begin{lemma}
Assume the eigen-values of $H$ satisfy $\lambda_1\leq\lambda_2\ldots\leq\lambda_k\leq -\gamma<0\leq \lambda_{k+1}\leq\lambda_d$ and $\eta\leq 1/L_1$ is sufficiently small.  Let $\e_1, \ldots, \e_d$ denote the corresponding eigen-values of $H$. Then the top-$k$ eigen-pairs of $A$ are $(\lambda^h_i(A), \widehat\e_i), i=1,\ldots, k$, where $\widehat\e_i=\left[\begin{array}{cc}\e_i\\ (1/\lambda^h_i(A))\e_i\end{array}\right]$ and
\begin{align*}
\lambda_i^{h}(A) =& \frac{1}{2} [(1+\zeta)(1-\eta\lambda_i) +\sqrt{(1+\zeta)^2(1-\eta\lambda_i)^2-4\zeta(1-\eta\lambda_i)} ].
\end{align*}
By choosing $\zeta = 1 -\sqrt{\eta\gamma}\in(0,1)$, the eigen-gap $\Delta_k$ of $A$ corresponding to its top-$k$  eigen-space is at least $\sqrt{\eta\gamma}/2$.
\end{lemma}
\begin{proof}
Following the analysis in~~\citep{corrACGWright}, we know $A$ is similar to $\text{diag}(M_1, \dots, M_d)$, where
\begin{align}\label{mat:M:An}
M_i = \left[\begin{array}{cc}(1+\zeta)(1  - \eta \lambda_i) & -\zeta(1  - \eta \lambda_i)  \\ 1 & 0\end{array}\right] 
\end{align}
Then the eigenvalues of $M_i$ satisfy
\begin{align}\label{M:eigenval}
 h(\mu) := \text{det}(A-\mu I) =& \mu^2 - (1+\zeta)(1-\eta\lambda_i) \mu + \zeta (1-\eta\lambda_i) =0,
\end{align}
for which the roots are
\begin{align}\label{M:eigenval:roots}
 \mu_i^{h,l}(A) = \frac{1}{2} [(1+\zeta)(1-\eta\lambda_i) \pm \sqrt{(1+\zeta)^2(1-\eta\lambda_i)^2-4\zeta(1-\eta\lambda_i)} ],
\end{align}
and the discriminant is 
\begin{align*}
\Delta = [(1-\zeta)^2 - (1+\zeta)^2 \eta\lambda_i] (1-\eta\lambda_i).
\end{align*}
When $\lambda_i \leq 0$, we have $\Delta \geq 0$, which implies both roots in (\ref{M:eigenval:roots}) are real.
Furthermore, we have $h(0) = \zeta (1-\eta\lambda_i) > 0$, $h(1) = 1 - (1+\zeta)(1-\eta\lambda_i) + \zeta (1-\eta\lambda_i) = \eta\lambda_i < 0$,
and $h((1+\zeta)(1-\eta\lambda_i)) =  \zeta (1-\eta\lambda_i) > 0 $ with $ (1+\zeta)(1-\eta\lambda_i) > 1$, then we know one root is in $(0,1)$ and another is in $(1, \infty)$.

When $\lambda_i > 0$, we want to show the magnitudes of roots are less than $1$. If $\Delta < 0$, $\mu_i^{h,l}(A)$ are complex, and their magnitude is
\begin{align*}
&\frac{1}{2}\sqrt{(1+\zeta)^2(1-\eta\lambda_i)^2 + 4\zeta(1-\eta\lambda_i) - (1+\zeta)^2(1-\eta\lambda_i)^2}= \sqrt{\zeta(1-\eta\lambda_i)} < 1.
\end{align*}
If $\Delta \geq 0$, we want
\begin{align*}
-2 <  [(1+\zeta)(1-\eta\lambda_i)  \pm \sqrt{(1+\zeta)^2(1-\eta\lambda_i)^2-4\zeta(1-\eta\lambda_i)} ] < 2.
\end{align*}
For the right hand side, it is equivalent to 
\begin{align*}
& \sqrt{(1+\zeta)^2(1-\eta\lambda_i)^2-4\zeta(1-\eta\lambda_i)} < 2 - (1+\zeta)(1-\eta\lambda_i) \\
\Leftrightarrow&(1+\zeta)^2(1-\eta\lambda_i)^2-4\zeta(1-\eta\lambda_i) < [2 - (1+\zeta)(1-\eta\lambda_i)]^2 \\
\Leftrightarrow&-\zeta(1-\eta\lambda_i) < 1 - (1+\zeta)(1-\eta\lambda_i) \\
\Leftrightarrow& 0 < 1 - (1-\eta\lambda_i) \\
\Leftrightarrow& 0 < \eta\lambda_i,
\end{align*}
where the last inequality holds due to $\eta > 0$ and $\lambda_i > 0$. Next, for the left hand side, it is  equivalent to 
\begin{align*}
& -2 - (1+\zeta)(1-\eta\lambda_i) < - \sqrt{(1+\zeta)^2(1-\eta\lambda_i)^2-4\zeta(1-\eta\lambda_i)}\\
\Leftrightarrow&\sqrt{(1+\zeta)^2(1-\eta\lambda_i)^2-4\zeta(1-\eta\lambda_i)} < 2 + (1+\zeta)(1-\eta\lambda_i)\\ 
\Leftrightarrow&(1+\zeta)^2(1-\eta\lambda_i)^2-4\zeta(1-\eta\lambda_i) < [2 + (1+\zeta)(1-\eta\lambda_i)]^2 \\
\Leftrightarrow&-\zeta(1-\eta\lambda_i) < 1 + (1+\zeta)(1-\eta\lambda_i) \\
\Leftrightarrow& 0 < 1 + (1+2\zeta) (1-\eta\lambda_i) ,
\end{align*}
where the last inequality holds due to $\eta\lambda_i < 1$. This is clearly true because of the assumption of $\eta L_1 < 1$.

Since $A =  \left[\begin{array}{cc}(1 + \zeta)(I - \eta H)& -\zeta (I - \eta H)\\ I & 0 \end{array}\right]$, then we have
\begin{align*}
A \widehat\e_i=& A \left[\begin{array}{cc}\e_i\\ (1/\lambda^h_i(A))\e_i\end{array}\right] =  \left[\begin{array}{cc}\{ (1 + \zeta)(1 - \eta \lambda_i) - (1/\lambda^h_i(A))\zeta (1 - \eta \lambda_i)\} \e_i  \\ \e_i\end{array}\right]\\
= & \left[\begin{array}{cc}\lambda^h_i(A)\e_i  \\ \e_i\end{array}\right],
\end{align*}
where the last equality is ture due to (\ref{M:eigenval}). 

We then set $\zeta = 1 - \sqrt{\eta\gamma}$. For $\lambda_i < 0, i=1,\ldots, k$, we have $ \eta \lambda_i \leq  -\eta \gamma$. 
\begin{align*}
\lambda_i^{h}(A) = &\frac{1}{2} [(1+\zeta)(1-\eta\lambda_i) +\sqrt{(1+\zeta)^2(1-\eta\lambda_i)^2-4\zeta(1-\eta\lambda_i)} ]\\
 \geq& \frac{1}{2} [(1+\zeta)(1+\eta\gamma) +\sqrt{(1+\zeta)^2(1+\eta\gamma)^2-4\zeta(1+\eta\gamma)} ]\\
=& \frac{1}{2} [(2 - \sqrt{\eta\gamma})(1+\eta\gamma)  +\sqrt{(2 - \sqrt{\eta\gamma})^2(1+\eta\gamma)^2-4(1 - \sqrt{\eta\gamma})(1+\eta\gamma)} ]\\
=& \frac{1}{2} [(2 - \sqrt{\eta\gamma})(1+\eta\gamma) +\sqrt{\eta\gamma} \sqrt{(5 - 4 \sqrt{\eta\gamma} + \eta\gamma)(1+\eta\gamma)} ]\\
=& \frac{1}{2} [(2 - \sqrt{\eta\gamma})(1+\eta\gamma) +\sqrt{\eta\gamma} \sqrt{5 - 4 \sqrt{\eta\gamma} + 6\eta\gamma - 4 \sqrt{\eta\gamma}\eta\gamma + \eta^2\gamma^2 } ]\\
=& \frac{1}{2} [(2 - \sqrt{\eta\gamma})(1+\eta\gamma) +\sqrt{\eta\gamma} \sqrt{(2 - 2\sqrt{\eta\gamma} + \eta\gamma)^2 + 1 + 2\sqrt{\eta\gamma}(2-\sqrt{\eta\gamma})} ]\\
\geq & \frac{1}{2} [(2 - \sqrt{\eta\gamma})(1+\eta\gamma) +\sqrt{\eta\gamma} (2 - 2\sqrt{\eta\gamma} + \eta\gamma) ]\\
 = & 1 + \frac{ \sqrt{\eta\gamma}}{2},
\end{align*}
where the first inequality is due to function $h(x) = (1+\zeta)(1-x) +\sqrt{(1+\zeta)^2(1-x)^2-4\zeta(1-x)} $ is a decreasing function in terms of $x < 0$ and $ \eta \lambda_i \leq  -\eta \gamma$. From the above analysis, we can see that when $\zeta=1-\sqrt{\eta\gamma}$, only k eigen-values of A are larger than $1+\sqrt{\eta\gamma}/2$, and others are less than 1 in magnitude. Therefore, the eigen-gap corresponding to the top-k eigen-space is at least $ \sqrt{\eta\gamma}/2$. 
\end{proof}

\end{document}